\documentclass[11pt,oneside]{amsart}
\usepackage{bbm}
\usepackage{mathrsfs}

\usepackage{caption}
\usepackage[english]{babel}
\usepackage[utf8]{inputenc}
\usepackage[T1]{fontenc}
\usepackage[top=3cm,bottom=2cm,left=3cm,right=3cm,marginparwidth=1.75cm]{geometry}
\usepackage{physics}
\usepackage{amsmath,amssymb,amsthm}
\usepackage{thmtools}
\usepackage{esint}
\usepackage{graphicx}
\usepackage[dvipsnames]{xcolor}
\usepackage[colorlinks=true,allcolors=BlueViolet, pagebackref=true]{hyperref}
\renewcommand*\backref[1]{\ifx#1\relax \else (Cited on p.#1) \fi}
\hypersetup{hypertexnames=false}
\usepackage{cleveref}
\usepackage[shortlabels]{enumitem}
\usepackage[all,cmtip]{xy}
\usepackage{relsize}
\usepackage{subcaption}
\usepackage{dsfont}

\usepackage{xspace} 
\usepackage{cancel}
\usepackage{fancyvrb}

\DeclareFontFamily{U}{BOONDOX-calo}{\skewchar\font=45 }
\DeclareFontShape{U}{BOONDOX-calo}{m}{n}{
  <-> s*[1.05] BOONDOX-r-calo}{}
\DeclareFontShape{U}{BOONDOX-calo}{b}{n}{
  <-> s*[1.05] BOONDOX-b-calo}{}
\DeclareMathAlphabet{\mathbdx}{U}{BOONDOX-calo}{m}{n}
\SetMathAlphabet{\mathbdx}{bold}{U}{BOONDOX-calo}{b}{n}
\DeclareMathAlphabet{\mathbbdx}{U}{BOONDOX-calo}{b}{n}

\newtheorem{theorem}{Theorem}[subsection]

\newtheorem{proposition}[theorem]{Proposition}
\newtheorem{corollary}[theorem]{Corollary}
\newtheorem{lemma}[theorem]{Lemma}

\theoremstyle{definition}
\newtheorem{definition}[theorem]{Definition}
\newtheorem{assumption}{Assumption}

\theoremstyle{remark} 

\newtheorem{remark}[theorem]{Remark}
\newtheorem{example}[theorem]{Example}

\crefname{assumption}{\textbf{Assumption}}{\textbf{Assumptions}}
\crefname{equation}{Equation}{Equations}
\crefname{gather}{Equation}{Equations}
\crefname{multline}{Equation}{Equations}
\crefname{figure}{Figure}{Figures}
\crefname{question}{Question}{Question}
\crefname{section}{Section}{Sections}
\crefname{subsection}{Subsection}{Subsections}
\crefname{appendix}{Appendix}{Appendices}
\crefname{lemma}{Lemma}{Lemmas}
\crefname{proposition}{Proposition}{Propositions}
\crefname{theorem}{Theorem}{Theorems}
\crefname{corollary}{Corollary}{Corollaries}
\crefname{definition}{Definition}{Definitions}
\crefname{remark}{Remark}{Remarks}
\crefname{example}{Example}{Examples}
\crefname{claim}{Claim}{Claim}
\crefname{conjecture}{Conjecture}{Conjecture}
\crefname{yauconjecture}{Yau's conjecture}{Yau's conjecture}

\definecolor{bluola}{RGB}{138,43,226}

\newcommand{\Hbar}{\ensuremath{H}%
\kern-0.65em\hbox{\rule[1.8ex]{0.57em}{0.47pt}}}

\newcommand{\R}{\mathbb{R}}

\newcommand{\C}{\mathbb{C}}
\newcommand{\N}{\mathbb{N}}
\newcommand{\PP}{\mathbb{P}}

\newcommand{\E}{\mathbb{E}}

\newcommand{\de}{\partial}

\newcommand{\f}{\varphi}
\renewcommand{\a}{\alpha}
\newcommand{\e}{\varepsilon}

\newcommand{\id}{\mathbbm{1}}
\newcommand{\vol}[1]{\mathrm{Vol}^{#1}}

\newcommand{\Var}{{\mathbb{V}\mathrm{ar}}}
\newcommand{\Cov}{\mathrm{Cov}}

\def\randin{%
  \mathchoice%
    {\raisebox{-.35ex}{$\displaystyle{^\subset}$}\mkern-11.5mu\raisebox{+.45ex}{$\displaystyle{_\subset}$}}
    {\mkern+1mu\raisebox{-.27ex}{$\textstyle{^\subset}$}\mkern-11.7mu\raisebox{+.45ex}{$\textstyle{_\subset}$}}
    {\raisebox{.35ex}{$\scriptstyle\subset$}\mkern-14mu\raisebox{-.15ex}{$\scriptstyle\subset$}}
    {\raisebox{.3ex}{$\scriptscriptstyle\subset$}\mkern-13.5mu\raisebox{-.10ex}{$\scriptscriptstyle\subset$}}
}

\DeclareMathOperator{\supp}{supp}

\makeatletter
\newcommand{\tpitchfork}{%
  \raise-0.1ex\vbox{
    \baselineskip\z@skip
    \lineskip-.52ex
    \lineskiplimit\maxdimen
    \m@th
    \ialign{##\crcr\hidewidth\smash{$-$}\hidewidth\crcr$\pitchfork$\crcr}
  }%
}

\newcommand{\mC}{\mathcal{C}}
\newcommand{\m}[1]{\mathcal{#1}}
\newcommand{\be}{\begin{equation}}
\newcommand{\ee}{\end{equation}}
\numberwithin{equation}{section}

\newcommand{\bega}{\begin{equation}\begin{aligned}}
\newcommand{\eega}{\end{aligned}\end{equation}}

\newcommand{\begt}{\begin{equation}\begin{gathered}}
\newcommand{\eegt}{\end{gathered}\end{equation}}
\newcommand{\kop}{\left\{}
\newcommand{\pok}{\right\}}
\newcommand{\tyu}{\left(}
\newcommand{\uyt}{\right)}
\newcommand{\qwe}{\left[}
\newcommand{\ewq}{\right]}

\usepackage{cite}
\newcommand{\FGF}{\mathrm{FGF}}
\newcommand{\Wnoise}{\mathbb{W}}

\newcommand{\AG}{\mathbb{A}}
\newcommand{\TAG}{\mathcal{A}}
\newcommand{\origin}{\mathrm{o}}
\newcommand{\Sch}{\mathcal{S}}
\newcommand{\Fou}{\mathcal{F}}
\newcommand{\Fouvar}[2]{\mathcal{F}_{#1\to #2}}
\newcommand{\Rad}{\mathcal{R}}
\newcommand{\Xray}{\mathcal{X}}

\newcommand\jump{\par\medskip}
\newcommand\quand{\quad\text{and}\quad}
\newcommand\qsothat{\quad\text{so that}\quad}

\newcommand{\enstq}[2]{\left\{#1~\middle|~#2\right\}}
\newcommand\one{\mathds{1}}

\newcommand\Id{\mathrm{Id}}
\newcommand{\Hurst}{H}
\newcommand{\Hil}{\mathfrak{H}}
\newcommand{\Hdot}{\dot{H}}
\newcommand{\ei}[1]{\frac{e^{ i\langle {#1} \rangle}}{(2\pi)^{d}}}
\newcommand{\eim}[1]{e^{-i\langle {#1} \rangle}}
\newcommand{\eiR}[1]{\frac{e^{ i {#1} }}{2\pi}}
\newcommand{\eimR}[1]{e^{-i {#1} }}
\newcommand{\scal}[1]{\langle {#1}\rangle}
\newcommand{\RUS}{TS^{d-1}}
\newcommand{\RUSR}{TS^{d-1}\times \R}
\newcommand{\TS}{TS^{d-1}}
\newcommand{\TSR}{TS^{d-1}\times \R}
\newcommand{\wick}[1]{:\!{#1}\!:}
\newcommand{\wxrayf}{\wick{|\Xray f|^2}}
\newcommand{\wxrayfR}{\wick{|\Xray_R f|^2}}

\newcommand{\Haus}{\mathcal{H}}

\newcommand{\sym}{\text{sym}}
\newcommand{\Jac}{\mathrm{Jac}}
\newcommand{\Ito}{I}
\newcommand{\lawR}[1]{\xrightarrow[R\to\infty]{\ \mathrm{law}-{#1}\ } }

\newcommand{\minv}{\mu}
\newcommand{\spec}{\omega}
\newcommand{\embf}[1]{\textbf{#1}}

\newcommand{\scrd}{\Sigma_R(f)}

\title{Scars in random waves and the ${\bf FGF(\frac12)}$ universality class}
\author{Louis Gass, Giovanni Peccati, Michele Stecconi}
\date{March 2026}

\begin{document}

\maketitle
\begin{abstract}
{We study the large-domain asymptotics of geometric observables in Berry's random wave model on $\R^d$. We show that, in sharp contrast with the behavior of stationary random fields with absolutely continuous spectral measures, any observable whose fluctuations are asymptotically fully correlated with its second Wiener chaos projection --- under suitable non-degeneracy assumptions --- belongs to a common universality class governed by a fractional Gaussian field with Hurst index $H=\frac{1-d}{2}$. This class also includes the classical stationary Poisson line process in $\R^d$.}

{Our findings show that suitable raw observables of Berry's random wave (such as critical point counts or non-nodal level set volumes) have large-domain fluctuations that become arbitrarily close --- in the sense of random tempered distributions --- to those generated by a (possibly noisy) Poisson line process. This probabilistic approximation provides evidence that the large-scale filamentary patterns observed in numerical simulations of random waves—often referred to as ``scars'' or ``scarlets'' following the numerical investigations of Heller, O'Connor and Gehlen (1987)—may admit a natural probabilistic interpretation.}

{In the second part of our work, we characterize the scaling limit --- in a distributional sense --- of suitable quadratic transformations of the {\it Radon--Fourier coefficients} associated with a large class of stationary fields. We show that random waves are characterized by the property that such a scaling limit is a generalized random field obtained by composing white noise on the affine Grassmannian of lines with a dimension-dependent deterministic operator.}

{As an application of our main results, we derive explicit conditions ensuring that quadratic functionals of pullback monochromatic waves on compact Riemannian manifolds exhibit distributional limits in the fractional Gaussian universality class described above.}

\smallskip

\noindent{\bf Keywords}: Random fields; Random waves; scars; Poisson line processes; fractional Gaussian fields; Wiener chaos; stochastic geometry; Radon transform

\smallskip

\noindent{\bf AMS Classification}: MSC 2020: Primary 60G60, 60F05; Secondary 60D05, 44A12, 60G57.

\end{abstract}

\tableofcontents

\section{Introduction}\label{ss:overviewintro} For $d\geq 2$, {\bf Berry's random wave} (BRW) model on $\R^d$ is the unique (in law) centered, isotropic, unit-variance Gaussian field $B=\{B(x):x\in\R^d\}$ satisfying almost surely the Helmholtz equation
\begin{equation}\label{e:helmholtz}
\Delta B(x)+B(x)=0,\qquad x\in\R^d.
\end{equation}
Equivalently, the covariance function of $B$ is given by its spectral representation
\begin{equation}\label{e:berrycov}
\mathbb{E}\big[B(x)B(y)\big]
=\int_{S^{d-1}} e^{i\langle x-y, \theta\rangle }\,\sigma(\dd\theta)
=2^{\nu}\Gamma(\nu+1)\frac{J_{\nu}(\|x-y\|)}{\|x-y\|^{\nu}},
\qquad x,y\in\R^d,
\end{equation}
where $\sigma$ denotes the uniform probability measure on the unit sphere $S^{d-1}\subset\R^d$, and $J_{\nu}$ is the Bessel function of the first kind of order $\nu=\frac{d}{2}-1$.

\smallskip

The model defined by \eqref{e:helmholtz}---\eqref{e:berrycov} goes back at least to M.V. Berry's seminal paper \cite{Berry1977}, where it was proposed as a universal ansatz for the local high-energy behavior of Laplace eigenfunctions on manifolds with ergodic geodesic flow, supported by a detailed semiclassical argument. Commonly referred to as {\bf Berry's random wave conjecture}, the prediction that high-energy eigenfunctions are locally described by the model \eqref{e:helmholtz}---\eqref{e:berrycov} remains one of the central open problems in quantum chaos and in the geometry of random fields; see \cite{garcíaruiz2023relation, ingremeauBerry, Abert2018, nalini_book_quantumergodicity, Gass2020} for recent advances on the matter. In the last two decades, the geometric properties of BRW and of related models --- including {\bf random spherical harmonics} and {\bf arithmetic random waves} --- have been studied extensively, often in connection with the {\bf variance cancellation phenomena} first uncovered in \cite{Berry2002a}, and with outstanding open problems in differential geometry, like e.g. {\bf Yau's conjecture} \cite{yauconjecture, logunov_reviewYau}. We refer to \cite{wigman_survey_2024} for a recent survey of this line of research, and to \cite{CH20, ORW2008, Rudnick2008, benatarMW, NazSodin, NazarovSodin2016, Bel19, NourdinPeccatiRossi2019, cgv2025StecconiTodino, Marinucci2016, Wigman_2010, WigmanAnnMath, DNPR23, NPV, PV} for a selection of representative contributions.

\smallskip

\begin{figure}[htbp]
\centering
\resizebox{0.75\textwidth}{!}{%
\begin{tabular}{ccc}
\includegraphics{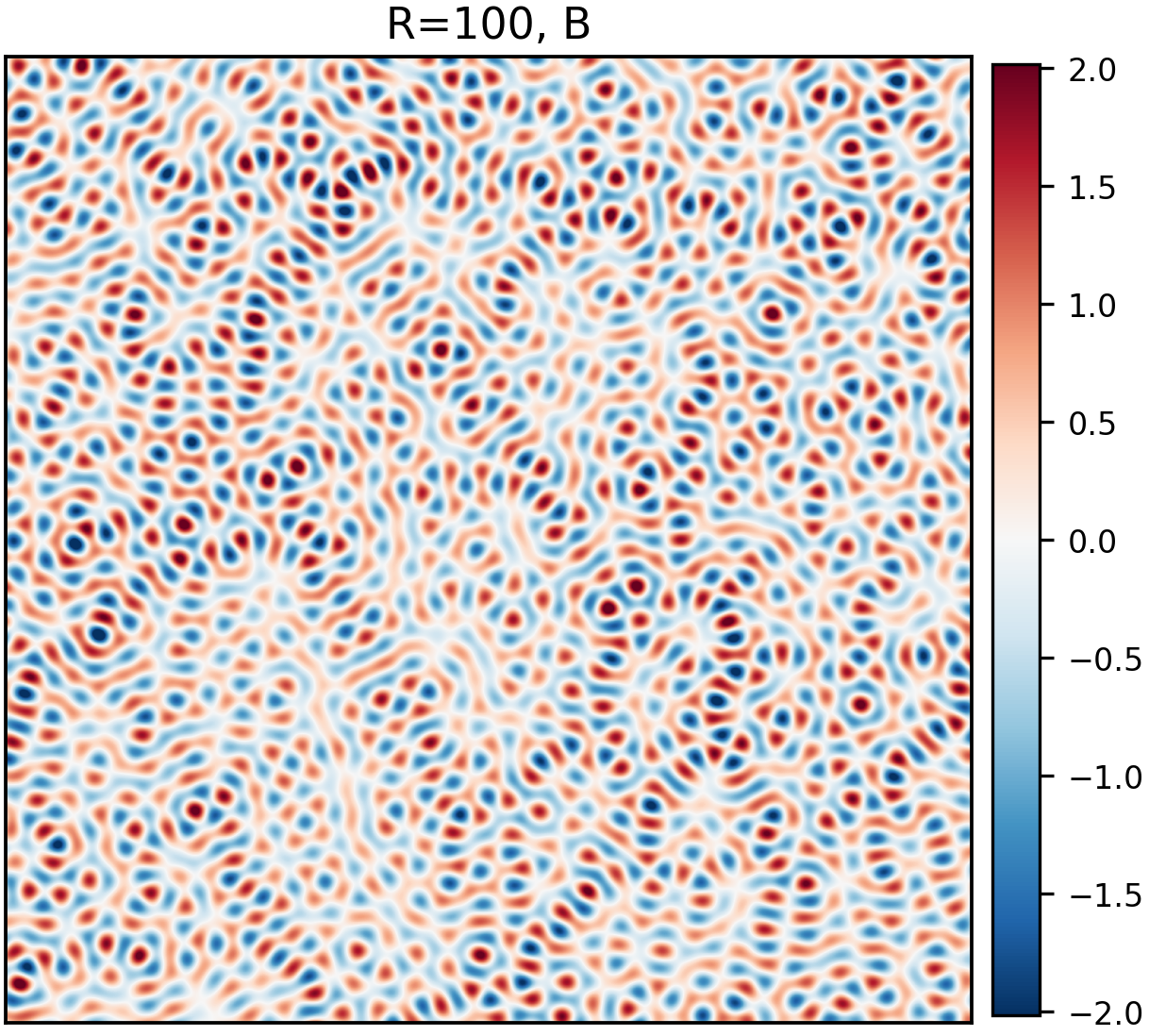} &
\includegraphics{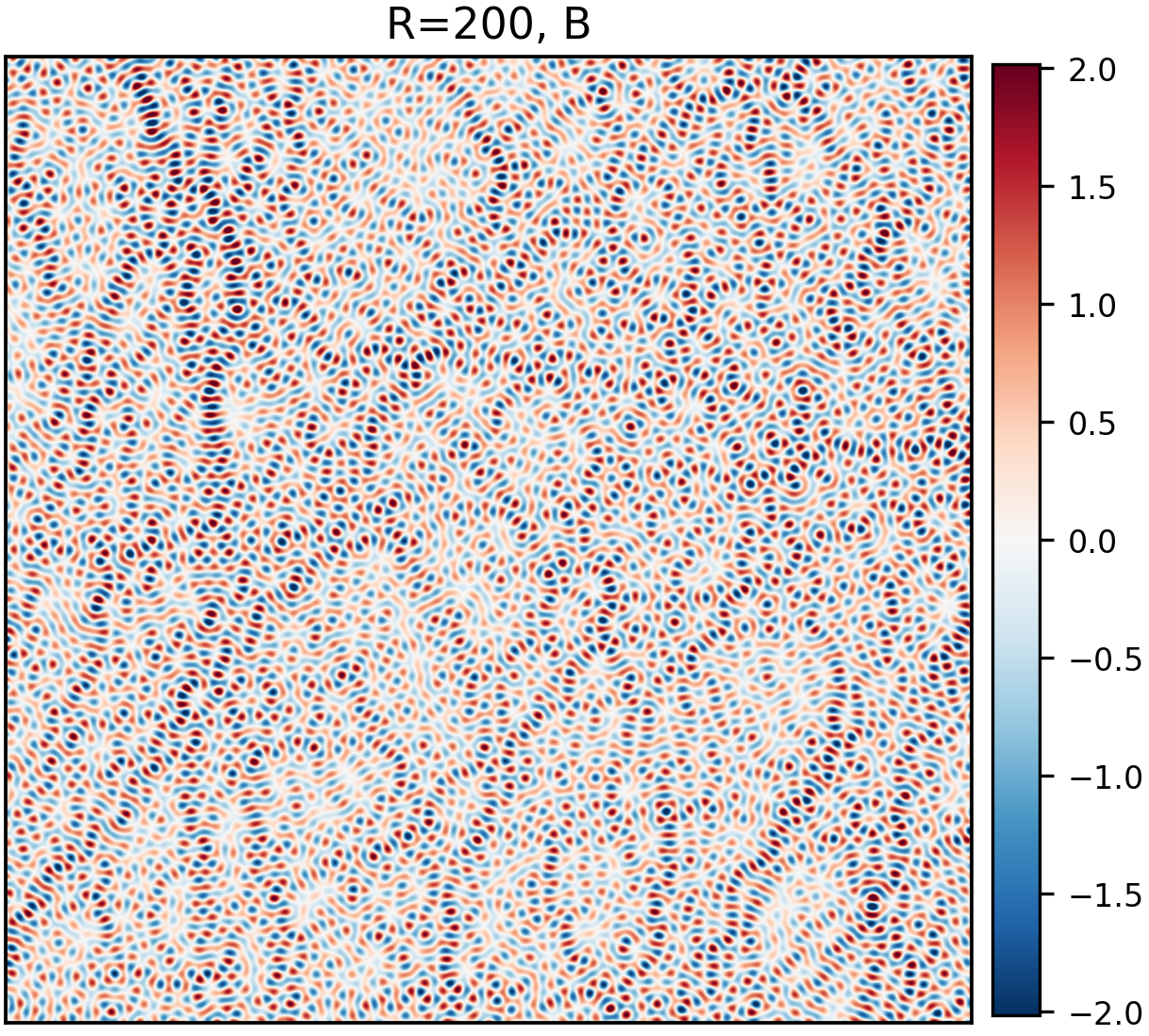} &
\includegraphics{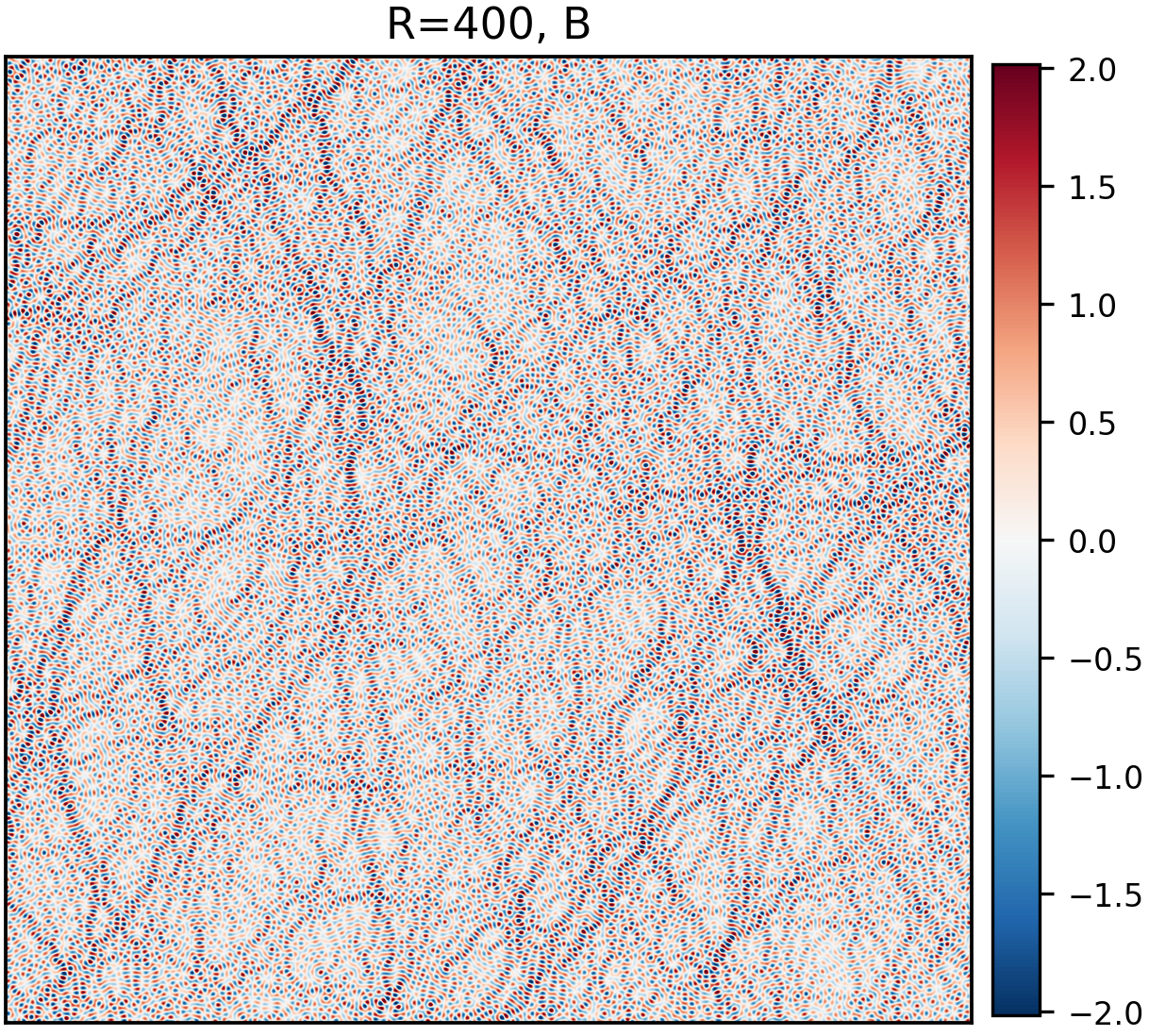} \\[-1mm]
{{\fontsize{17}{18}\selectfont (a)}} &
{{\fontsize{17}{18}\selectfont (b)}} &
{{\fontsize{17}{18}\selectfont (c)}} \\[1mm]

\includegraphics{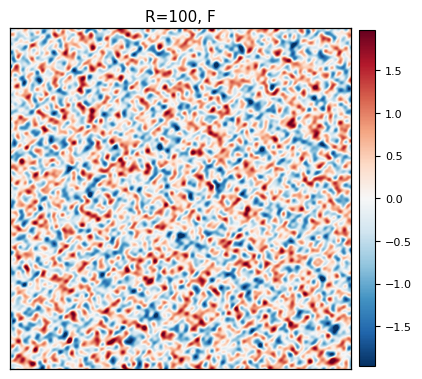} &
\includegraphics{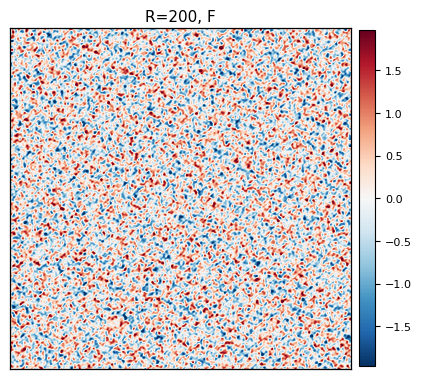} &
\includegraphics{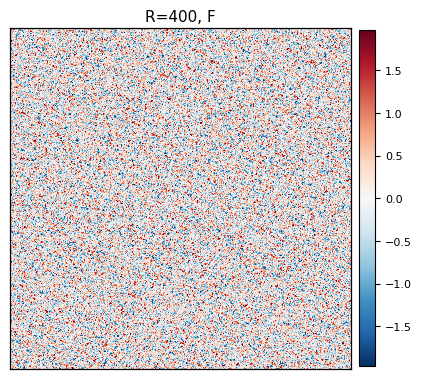} \\[-1mm]
{{\fontsize{17}{18}\selectfont (d)}} &
{{\fontsize{17}{18}\selectfont (e)}} &
{{\fontsize{17}{18}\selectfont (f)}}
\end{tabular}%
}
\caption{\tiny First row ((a)---(c)): heatmap of the BRW field $B$ over the window $[-R,R]^2$, for $R =100, 200, 400$. Second row ((d)---(f)): heatmaps over the same windows of the Bargmann-Fock field $F$ (with covariance $\mathbb{E}[F(0)F(x)] \asymp e^{-\|x\|^2}$). }
\label{fig:brw-bf-grid}
\end{figure}

One striking empirical feature of the two-dimensional BRW model is that numerical simulations over large domains display characteristic quasi-linear filamentary patterns --- typically described as {\bf scars} or {\bf scarlets} --- that persist across scales and roughly correspond to chains of narrow high-level excursions of alternating sign. Such patterns are completely absent from simulations of other natural isotropic fields (such as the {\bf Bargmann--Fock field}, see Figure~\ref{fig:brw-bf-grid}) and are clearly revealed by observables such as the square of Berry's field (Figure~\ref{fig:brw-four-panels}-(a)), critical point configurations (Figure~\ref{fig:brw-four-panels}-(b)), non-zero level sets (Figure~\ref{fig:brw-four-panels}-(c)), and non-zero excursion sets (Figure~\ref{fig:brw-four-panels}-(d)). By contrast, they are essentially invisible at the level of nodal domains and nodal lines; see Figure~\ref{fig:brw_nodal}. 

\begin{remark}\label{r:whitelies}{\rm
From a conceptual viewpoint, one may distinguish between two aspects of the phenomenon: the geometry of the filamentary structures themselves, and the oscillatory patterns carried by them. The first aspect leads to what one might {informally} call {\bf white scars}, corresponding to observables\footnote{For the rest of the paper we use the expression ``observable'' to denote any distribution-valued functional of a given random function; see Section \ref{ss:convintro}.} that primarily encode the geometry of the filaments. The second aspect leads to observables that may be viewed as {\bf coloured scars}, in the sense that they retain information about the oscillatory profile of the wave along the filaments.
}
\end{remark}

\begin{figure}[htbp]
\centering
\begin{tabular}{cccc}
\includegraphics[height=3.5cm]{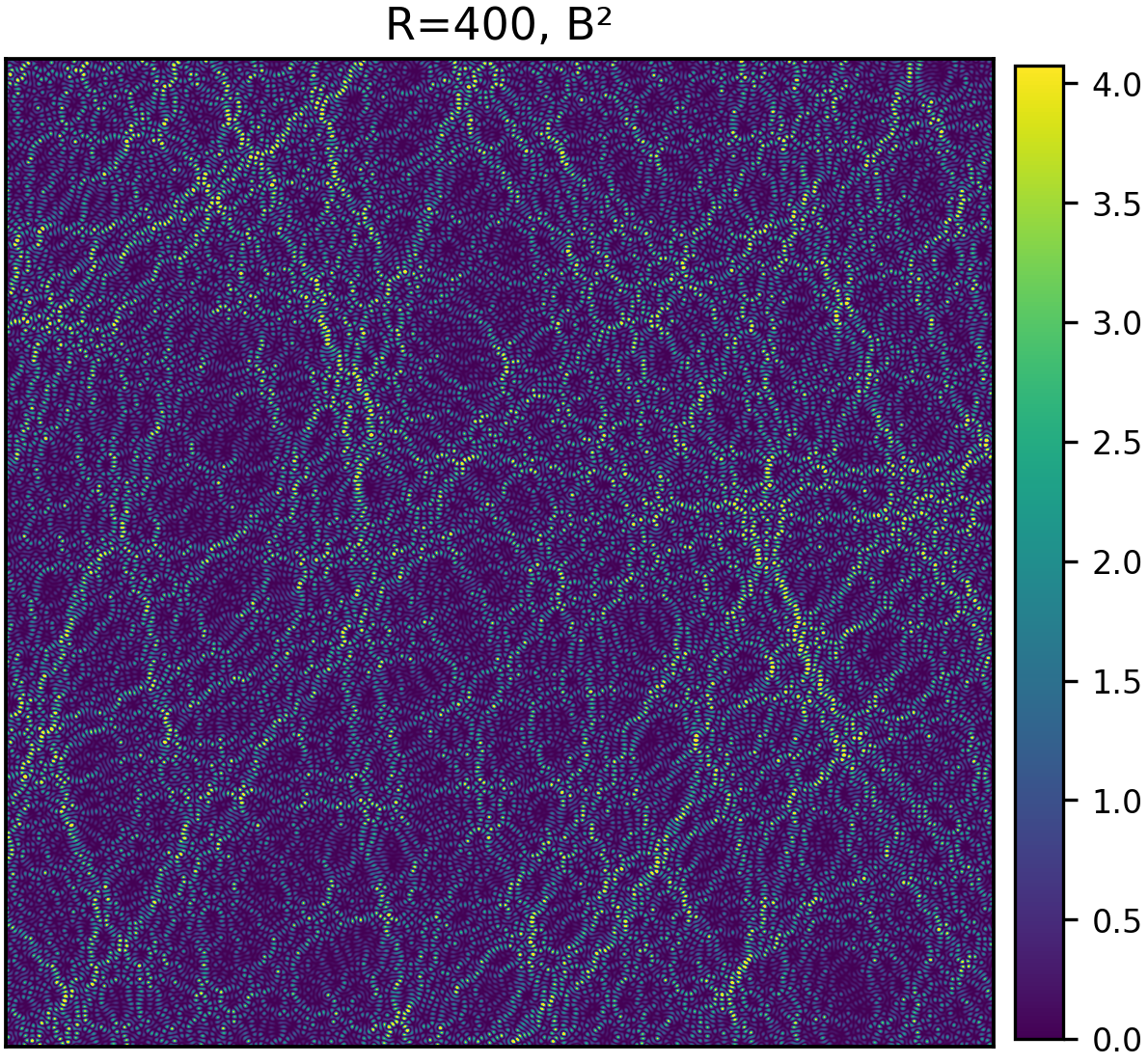} &
\includegraphics[height=3.5cm]{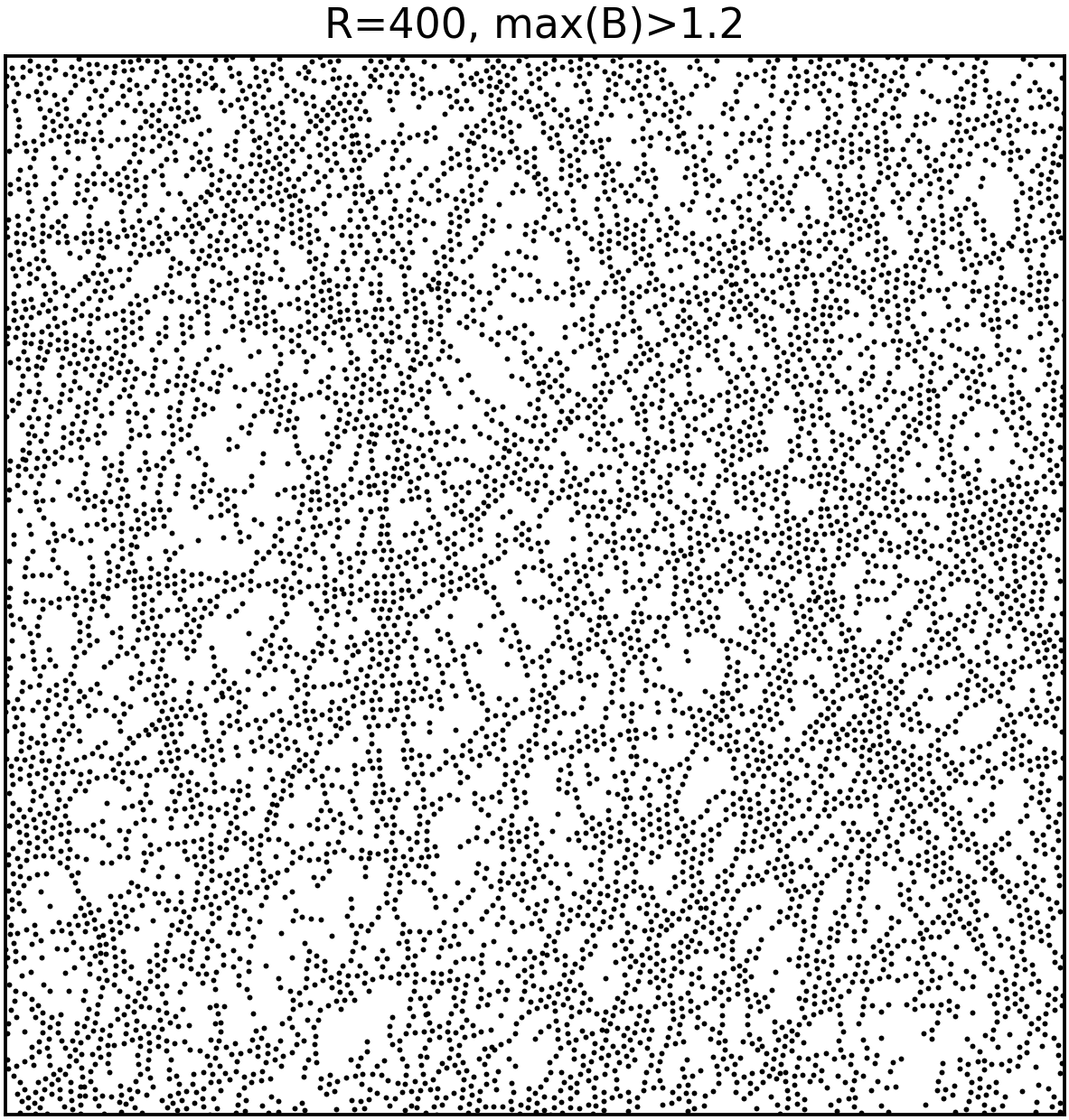} &
\includegraphics[height=3.5cm]{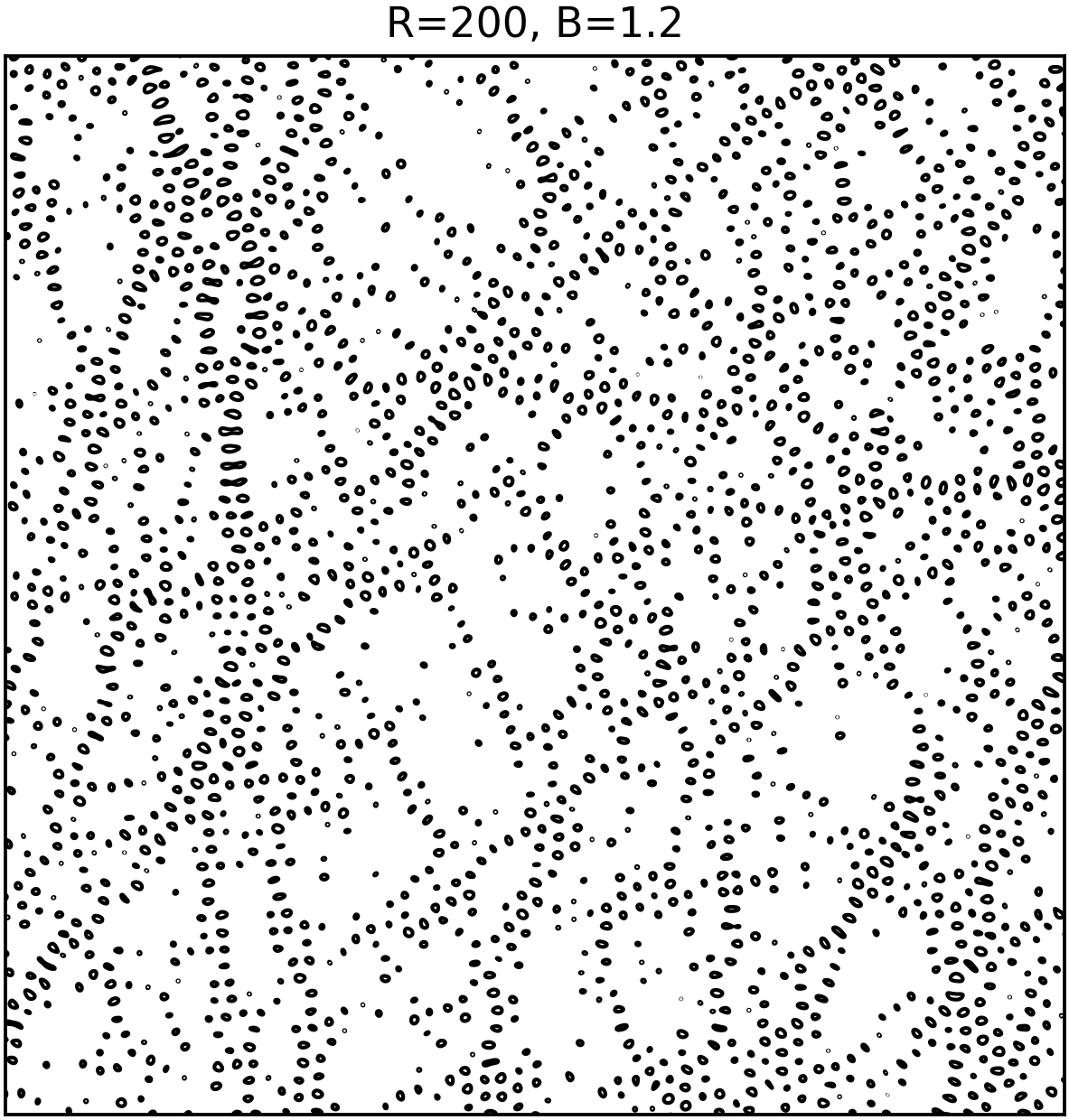} &
\includegraphics[height=3.5cm]{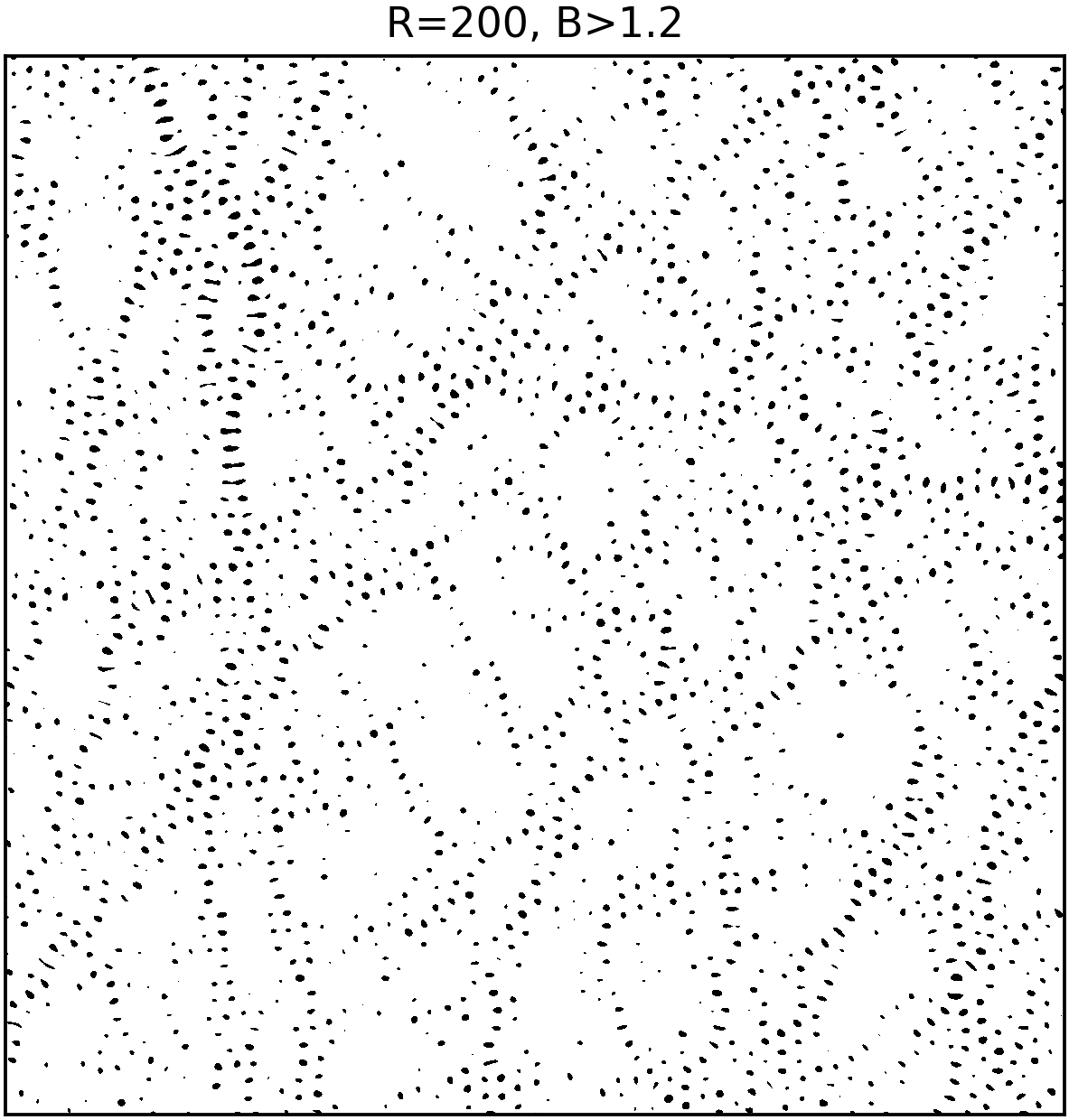} \\[-1mm]
{{\fontsize{9}{10}\selectfont (a)}} &
{{\fontsize{9}{10}\selectfont (b)}} &
{{\fontsize{9}{10}\selectfont (c)}} &
{{\fontsize{9}{10}\selectfont (d)}}
\end{tabular}
\caption{\tiny Several observables of BRW over the window $[-R,R]^2$: Figures (a) and (b) correspond to $R=400$ and are associated with Fig.~\ref{fig:brw-bf-grid}-(c), whereas Figures (c) and (d) correspond to $R=200$ and to Fig.~\ref{fig:brw-bf-grid}-(b).}
\label{fig:brw-four-panels}
\end{figure}

\smallskip

The appearance of scars in simulations of the BRW model was first pointed out in the seminal paper of Gehlen, Heller and O'Connor \cite{OConnorGehlenHeller}, where a mechanism for their formation was proposed, at a physics level of rigor, based on Gabor-type local filters \cite{FS_Gabor_98} and the approximate invariance of their moduli under the Laplace propagator; see also \cite[Chapter 5]{HellerLesHouches1989} and \cite[Chapter 23]{HellerBook2018}, as well as \cite{BDRWater} for related experimental evidence. Following the work of Gehlen, Heller and O'Connor, the appearance of scars in simulations of the BRW model has remained a recurring yet elusive topic. Striking visualisations of filamentary patterns in random waves were featured, for instance, on the cover of the {\it Notices of the AMS} in January 2008 \cite{noticesAMS_2008_01} (see also the associated paper \cite{Barnett_CPAM}), and were a central theme of a workshop at the American Institute of Mathematics in 2009 (see \cite{AIM_Workshop_2009} for a particularly vivid report). A.~Barnett's webpage \cite{barnettWebpage} provides a rich collection of numerical experiments illustrating these phenomena. More recently, Tacy \cite{Tacy23} showed that, while coefficients of X-ray transforms associated with BRW do not exhibit anomalous concentration, the $L^2$ norms of certain {\bf semiclassical localizers}---closely related in spirit to the Gabor-type filters evoked above---do display concentration properties consistent with the presence of filamentary structures. On a different front, the work of Beliaev and Hegde \cite{beliaevHegde2026, hegde_PHD_2025} establishes that high maxima of a broad class of isotropic fields, including BRW, exhibit Poissonian scaling limits, suggesting that the mechanism underlying scars should instead be sought within the bulk of critical values.

To the best of our knowledge, however, no rigorous mathematical framework---let alone a formal definition of {\it scar} for random waves---capturing the visual patterns first observed in \cite{OConnorGehlenHeller} has been attempted so far.

\begin{figure}[htbp]
\centering
\resizebox{0.55\textwidth}{!}{%
\begin{tabular}{cc}
\includegraphics{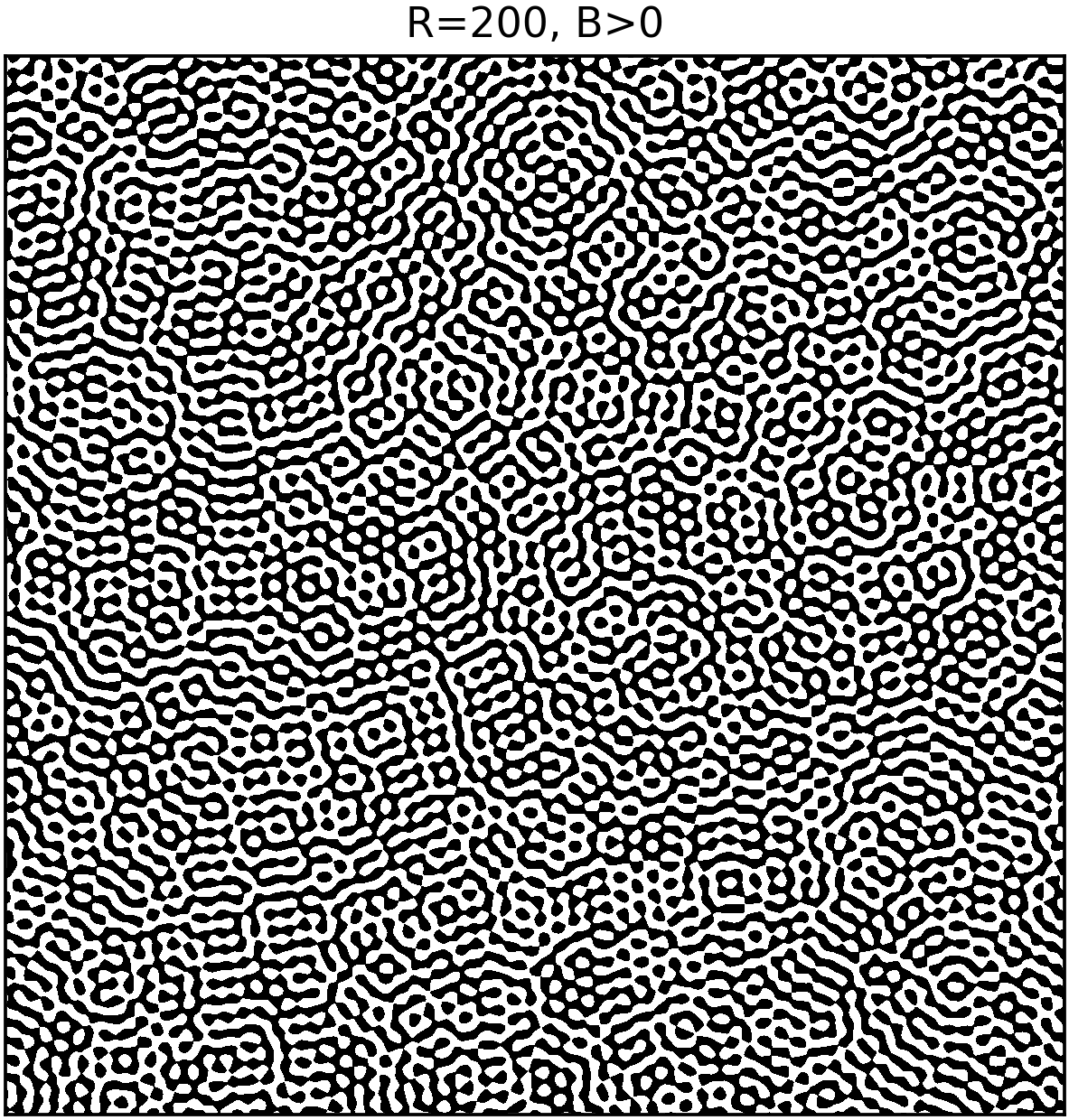}\hspace{2.2cm} &
\includegraphics{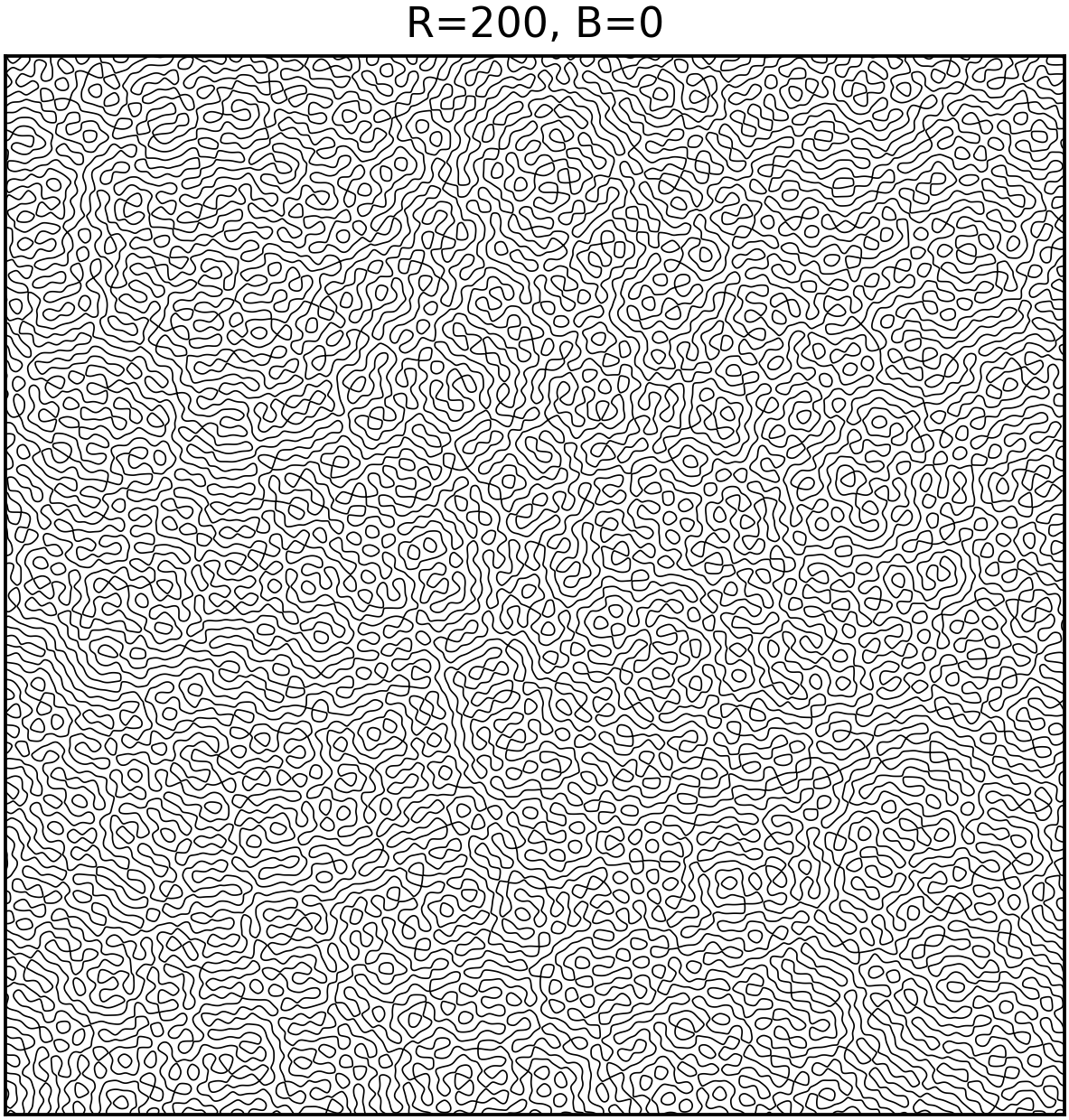} \\[-1mm]
\end{tabular}
}
\caption{\tiny Nodal domains and nodal lines of BRW over the window $[-R,R]^2$, for $R=200$. They both correspond to the realization in Fig. \ref{fig:brw-bf-grid}-(b).}
\label{fig:brw_nodal}
\end{figure}

The aim of this paper is to investigate the universal probabilistic mechanisms underlying the geometry and oscillatory structure of observables that empirically reveal the filamentary patterns exhibited by Berry's random wave.
\medskip

Our main findings may be summarized as follows.

\begin{itemize}

\item[\bf (1)] In Sections \ref{ss:ucintro} and \ref{ss:mainintro}, we identify a {\bf universality class} containing a broad family of observables associated with the BRW $B$, including the square of the field, suitably normalized critical point counts, level-set observables, and excursion-type observables revealing the filamentary structures visible in Figure~\ref{fig:brw-four-panels}. As proved in Theorem \ref{t:squarecorr} and Proposition \ref{p:squarecorrex}, the large-scale behavior of every observable in this class is governed by the {\bf fractional Gaussian field} of index $\frac12$ (whose law is denoted by $\FGF_{1/2}$; see \cite{FGFSurvey}, as well as Sections \ref{ss:limitintro} and \ref{ss:fgfs} below). More precisely, every member of the class satisfies both a law of large numbers under a classical volume-type normalization and a non-trivial fluctuation theorem under the normalization dictated by the self-similarity of the limiting field:
\[
R^{\frac{d-1}{2}}\bar{X}(B(R\cdot))
\stackrel{\rm law}{\longrightarrow}
\FGF_{1/2},
\]
where $\bar{X}$ denotes the centered observable, the convergence holds in the sense of random tempered distributions on $\R^d$, and the corresponding covariance structures converge as well. A common feature of all observables in the class is that, in the large-domain limit, their fluctuations are asymptotically dominated by their {\bf second Wiener chaos projection} (see \cite[Chapter 2]{nourdinpeccatibook}), which completely determines their limiting behavior. As discussed in Remark \ref{r:leonardo}, these findings are consistent with earlier second-order results for quadratic functionals of planar random waves \cite{MainiCov, Barnett_CPAM}. Finally, the above results are applied in Section \ref{ss:pbintro} to characterize the $\FGF_{1/2}$ scaling limits of quadratic functionals of {\bf pullback monochromatic random waves} on compact manifolds, as defined for instance in \cite{CH20, canzani_MRWsurvey, Zel09}.

\item[\bf (2)] As discussed in Section \ref{ss:ucintro}, we show that a defining feature of the above universality class is that it also contains the random tempered distribution canonically generated by a {\bf stationary Poisson process} $\eta$ on the {\bf affine line Grassmannian} $\AG_d$ of $\R^d$ (see Section \ref{ss:affinegrass}). In particular, if
\begin{equation}\label{e:radeta}
Y=\Rad^*\eta
\end{equation}
denotes the corresponding {\bf Poisson line field}, that is, the distribution on $\R^d$ obtained by integrating against the affine lines in the support of $\eta$ (with $\Rad^*$ indicating the adjoint of the classical {\bf X-ray transform}; see Definition \ref{d:xray}), then
\[
R^{\frac{d-1}{2}}\bar{Y}(R\cdot)
\stackrel{\rm law}{\longrightarrow}
\FGF_{1/2},
\]
see Theorem \ref{t:poisson}-{\bf (a)}. We further establish in Theorem \ref{t:poisson}-{\bf (b)} that the above universality class is rich enough to permit a direct comparison between BRW observables and (possibly noisy) Poisson line process observables, in the sense that the finite-dimensional distributions of their large-domain fluctuations can be made arbitrarily close. As discussed in Remark \ref{r:lafolie}-(i), these results suggest that the large-scale filamentary patterns observed in numerical simulations of random waves may admit a natural probabilistic interpretation in terms of Poisson line processes, after averaging out the oscillatory profile carried by the wave (see also Remark \ref{r:whitelies}).

\item[\bf (3)] We will show that the phenomena described in Points {\bf (1)} and {\bf (2)} are by no means universal among stationary random fields. As demonstrated in Theorem \ref{t:integrablefields}, for a broad class of stationary Gaussian fields whose spectral measure admits a density, the scaling limits of the corresponding observables are instead described by white noise. This sharp contrast between the two universality classes is consistent with the markedly different large-scale patterns observed in the simulations displayed in Figure~\ref{fig:brw-bf-grid}.

\item[\bf (4)] Let $f$ be a smooth stationary Gaussian field on $\R^d$, consider the random distribution generated by the Wick square of $f$ through the relation
\begin{equation}\label{e:sqintrointro}
\varphi \mapsto \int_{\R^d} \varphi(x)\,:\!f(Rx)^2\!:\,\dd x,\quad R>0,
\end{equation}
and notice that, when $f=B$, the observable in \eqref{e:sqintrointro} (with $R=1$) is the most basic representative of the universality class described at Point {\bf (1)}. A key result established in Section \ref{ss:ricovero} is that the operator in \eqref{e:sqintrointro} admits a natural representation as the restriction to $\R^d$ of a random distribution on $\R^d\times \R$ of the form
\begin{equation}\label{e:tierre}
\scrd := \mathcal{R}^* \sqrt{-\Delta_r} : \! | \mathcal{X}_R f|^2 \!:,
\end{equation}
where $\sqrt{-\Delta_r}$ is a suitable fractional Laplacian, and $: \! | \mathcal{X}_R f|^2 \!:$ is the formal Wick square of the distribution induced by the one-dimensional {\bf Radon--Fourier coefficients} associated with $f$.{
We call $\scrd$ the \textbf{scar field} of $f$ at scale $R$. 
}

As demonstrated in Theorem \ref{t:maincolouredscars} (and discussed in greater detail in Sections \ref{s:linespace} and \ref{sec:ultraproof}), one of the main achievements of our work is a full characterization of the scaling limit of the distribution $\scrd$, depending on both the dimension $d\geq 2$ and the nature of the spectral measure associated with the law of $f$. We will see that, in the special case $f=B$, the limiting object admits a description in terms of the composition of a white noise on the {affine line Grassmannian} $\AG_d$ with a dimension-dependent deterministic operator. In this sense, the limiting random distribution provides a natural probabilistic counterpart of the {\em coloured scar-like patterns} informally discussed in Remark \ref{r:whitelies}.

\end{itemize}

\begin{remark}\label{r:overview}{\rm
\begin{itemize}

\item[(i)] Our results reveal an interesting dimensional dichotomy. The scaling limits associated with white scars are universal across dimensions and always belong to the ${\rm FGF}_{1/2}$ universality class. By contrast, once frequency information is retained, the limiting objects become dimension-dependent; see also Remark \ref{r:esempiodelsecolo} and \cref{ss:interpretationofultralimit}. We will see in Section \ref{ss:livingcoloursintro} and Sections \ref{s:linespace}--\ref{sec:ultraproof} that in dimension $d=2$, the limiting frequency distribution is concentrated at the two monochromatic frequencies $\pm 1$, whereas in dimensions $d\geq 3$ it is described by a non-trivial density depending explicitly on the dimension. In this sense, white scars exhibit a universal geometric behaviour, whereas coloured scars retain a finer signature of the ambient dimension. {We stress that the terminology of ``white'' and ``coloured'' scars is only intended to informally emphasize the distinction between geometric information and oscillatory information. The mathematical content of the paper is entirely formulated in terms of the universality classes and scaling limits discussed above.}
\item[(ii)] The filamentary patterns in random waves discussed above are sometimes compared with the phenomenon of {\it quantum scars} observed in eigenfunctions of classically chaotic systems (see e.g. \cite{OConnorGehlenHeller, HellerLesHouches1989, HellerBook2018, BDRWater}). Beginning with Heller's seminal work \cite{Heller1984}, the literature on quantum scars has largely developed around their relation to unstable periodic orbits; see, for instance, \cite{KaplanHeller1998, Nonnenmacher2013, RudnickSarnak}. The present paper does not address the relation between this semiclassical notion of scar and the probabilistic phenomena uncovered by our findings. 
\item[(iii)] Approximately one week after the first version of the present paper was posted on the arXiv, the independent work of Beliaev and Hegde \cite{beliaevHehgde2026} appeared. Among other results, it establishes a particular case of the $\FGF_{1/2}$ universality phenomenon described at Point {\bf (1)} above. Some further findings from \cite{beliaevHehgde2026}, and their relation to the present work, are discussed in Remark~\ref{r:lafolie}-(i).
\end{itemize}
}
\end{remark}

 From now on, every random object is defined on a common probability space $(\Omega, \mathscr{E}, \mathbb{P})$, with $\mathbb{E}$ indicating an expectation with respect to  $\mathbb{P}$.

\medskip 

The next section contains a discussion of the main results established
in our work, and implicitly serves as a plan for the rest of the paper

\subsection{Acknowledgements} Research supported by the  Luxembourg National Research Fund (Grant: O24/18972745/GFRF). We are grateful to Nathanaël Berestycki, Michael V. Berry, Francesco Caravenna, Eric J. Heller, Daniel Hug, Leonardo Maini, Zeev Rudnick, Wioletta Ruszel and Joe Yukich for useful discussions.

\section{Formal description of main results}

\subsection{Some conventions}\label{ss:convintro}

All objects discussed below are presented in more detail in Section \ref{s:preliminaries}. For $d\geq 2$, we denote by $\mathcal{S} =\mathcal{S}(\R^d)$ the class of Schwartz functions on $\R^d$, whereas the class of tempered distributions (the topological dual of $\mathcal{S}$) will be written $\mathcal{S}' = \mathcal{S}'(\R^d)$. We endow $\mathcal{S}'$ with the Borel $\sigma$-field $\mathcal{B}(\mathcal{S}')$ generated by the weak-$\star$ topology. Given $X\in\mathcal S'$ and $\varphi\in\mathcal S$, we will use interchangeably the notations
\begin{equation}\label{e:pairing}
X(\varphi)
\qquad\text{and}\qquad
\langle X,\varphi\rangle,
\end{equation}
both denoting the canonical duality pairing between
$\mathcal S'$ and $\mathcal S$.

\smallskip

In this paper, we use the term {\bf random distribution} to indicate a $\left(\mathcal{S}', \mathcal{B}(\mathcal{S}')\right)$-valued random element. {The existence of such objects is guaranteed by the Bochner-Minlos theorem (see \cite[Corollary 2.1]{Bierme_CSA}, as well as \cref{thm:bochnerminlos} below).} As anticipated, when a random distribution $X$ is measurable with respect to the $\sigma$-field $\sigma(Z)$ generated by a random field $Z =\{Z(x) : x\in \R^d\}$, we will say that $X$ is an {\bf observable} of $Z$.  As usual, a sequence of random distributions $\{X_n : n\geq 1\}$ is said to {\bf converge in law} to a random distribution $X$, written $X_n\stackrel{\rm law}{\longrightarrow} X$, if
$$
\mathbb{E}[\Phi(X_n)] \longrightarrow \mathbb{E}[\Phi(X)], \quad n\to\infty,
$$
for every $\Phi : \mathcal{S}'\to \R$ bounded and continuous. According e.g. to \cite[Corollary 2.4]{Bierme_CSA} and to the classical results by Fernique \cite{Fernique1968} {(see \cref{thm:levy_Sprime} below)}, one has that $X_n\stackrel{\rm law}{\longrightarrow} X$ if and only if $X_n(\varphi) \stackrel{\rm law}{\longrightarrow} 
X(\varphi)$ for all $\varphi\in \mathcal{S}$, where, here, $\stackrel{\rm law}{\longrightarrow}
$ indicates convergence in distribution of random variables. Given $\varphi : \R^d \to \R$ and $R>0$, we write $\varphi_R(\cdot) :=\varphi(\frac{\cdot}{R})$ to indicate the mapping
\begin{equation}\label{e:rescaled}
x\mapsto \varphi_R(x) := \varphi\left(\frac{x}{R}\right), \quad x\in \R^d.
\end{equation}

\smallskip

Given $t>0$, we will write $\Pi_t$ to denote a {\bf homogeneous Poisson point process} on $\R^d$ with intensity $t$ (see Section \ref{ss:poissonstuff} for definitions), that we identify with a random point measure on $\R^d$, and therefore with the random distribution $\Pi_t =\{\Pi_t(\varphi) : \varphi \in \mathcal{S}\}$, where $\Pi_t(\varphi) := \int_{\R^d} \varphi(x) \Pi_t(\dd x)$ (see Remark \ref{r:poissdist}). We will denote by $\Pi_0$ the {\bf trivial null measure}, with total mass equal to zero. Observe the following relation, that one can e.g. deduce from the general results of \cite[Section 7.3]{Pec11}: for every $q\geq 1$, the $q$th {\bf cumulant} of $\Pi_t(\varphi)$ is given by 
\begin{equation}\label{e:cumulantspoisson}
\kappa_q(\Pi_t(\varphi)) = t \int_{\R^d}\varphi(x)^q \dd x, \quad\mbox{and therefore}\quad \kappa_q\left(\Pi_t\left(\varphi_R\right)\right) = t R^d \int_{\R^d}\varphi(x)^q \dd x.
\end{equation}

Throughout the paper, and unless explicitly specified, the symbol $\|\cdot\|$ denotes the Euclidean norm on $\R^d$. Given two random vectors $U,V$ with values in $\R^m$, the {\bf convex distance} between the laws of $U$ and $V$ is defined as
\begin{equation}\label{e:convexd}
d_{\rm conv}(U,V) := \sup_C
\Big|\, \mathbb{P}[U\in C] - \mathbb{P}[V\in C]\, \Big|,
\end{equation}
where the supremum is taken over all convex sets $C\subset \R^m$. See, e.g., \cite{NPY_CONVEX}, and the references therein, for a recent discussion of the properties of $d_{\rm conv}$. 

Given numerical sequences $a_n, b_n$, we write (i) $a_n\simeq b_n$, and (ii) $a_n\asymp b_n$ to indicate, respectively, that (i) $b_n\neq 0$ and $a_n/b_n\to 1$, and (ii) there exist constants $0<c<C$ such that $c\, b_n\leq a_n \leq C\, b_n $.
\subsection{The scaling limit: fractional Gaussian fields of index $\frac12$}\label{ss:limitintro}

For $d\geq 2$, the {\bf Fractional Gaussian Field} of index $s=\frac12$ (in symbols, ${\rm FGF}_{1/2}(\R^d)$) is the unique (in law) centered Gaussian random distribution $h = \{h(\varphi) : \varphi\in \mathcal{S}\}$ with covariance given by
\begin{equation}\label{e:fgfcov1}
\mathbb{E}[h(\varphi)h(\psi)] = C(d)\, \int_{\R^d}\int_{\R^d} \frac{\varphi(x)\psi(y)}{\|x-y\|^{d-1}} \dd x \dd y, \quad \varphi,\psi\in \mathcal{S},
\end{equation}
with the constant $C(d)$ defined as
\begin{equation}\label{eq:C12d}
C\!\left(d\right)
:= \frac{1}{2}\,\pi^{-(d+1)/2}\,\Gamma\!\left(\frac{d-1}{2}\right)
= \frac{1}{\pi\,\Omega_{d-2}},
\end{equation}
and where $\Omega_{d-2}$ denotes the surface measure of the unit sphere $S^{d-2}\subset\mathbb{R}^{d-1}$. In what follows, we will write $$h\sim {\rm FGF}_{1/2}(\R^d)$$ to indicate the fact that $h$ is a centered random Gaussian distribution with covariance \eqref{e:fgfcov1}. As explained in full detail in Section \ref{ss:fgfs}, a field $h\sim {\rm FGF}_{1/2}(\R^d)$ is a member of the large family of {\bf fractional Gaussian fields} categorized in the influential survey \cite{FGFSurvey} (containing both the white noise and the Gaussian Free Field as special cases --- see also \cite{CaoSheffield}); as such, $h$ admits the distributional representation 
$$
h = (-\Delta)^{-1/4} \, \mathbb{W},
$$
where $\mathbb{W}$ indicates a standard white noise on $\R^d$, and $\Delta$ is the usual Laplacian. It is not difficult to deduce from \eqref{e:fgfcov1} that, if $h\sim {\rm FGF}_{1/2}(\R^d)$, then $h$ is a {\bf self-similar random distribution} with negative {\bf Hurst index} $H = \frac{1-d}{2}$, from which one deduces the crucial relation
\begin{equation}\label{e:selfsim}
 h \stackrel{\rm law}{=} \left\{ \frac{h\left(\varphi_R\right)}{R^{(d+1)/2}}  : \varphi\in \mathcal{S}\right\}, \quad R>0,
\end{equation}
where we have adopted the notation \eqref{e:rescaled}, and the equality in law is in the sense of random distributions. 

One fundamental fact exploited in the present paper (that, to the best of our knowledge, has been overlooked in the probabilistic literature so far) is that fractional Gaussian fields of index 1/2 can be naturally related to linear structures through the use of {\bf Blaschke–Petkantschin formulae}, such as the ones proved in \cite[formula (27)]{Miles1971} (see also \cite[Proposition 3.2]{KabSThaeleBook}, or \cite[formula (2)]{nikitenko2019BPsurvey}), yielding that, for all $\varphi, \psi\in \mathcal{S}$,
\begin{equation}\label{e:fgfcov2}
C(d) \int_{\R^d}\int_{\R^d} \frac{\varphi(x)\psi(y)}{\|x-y\|^{d-1}} \dd x \dd y = \int_{\AG_d} \mathcal{R}\varphi (\ell) \mathcal{R}\psi (\ell)\, \mu(\dd\ell),
\end{equation}
where $C(d)$ is given in \eqref{eq:C12d}, $\mathbb{A}_d$ indicates the {\bf affine Grassmannian} of affine lines in $\R^d$ (see Section \ref{ss:affinegrass}), $\mu$ is its {\bf invariant measure}\footnote{{The invariant measure on $\AG_d$ is uniquely defined up to a multiplicative constant. In this paper, for every $d$, such a constant will be implicitly chosen in such a way that the identity \eqref{e:fgfcov2} holds exactly. See e.g. \cite[Section 5.1]{KabSThaeleBook} for a discussion of alternative normalization choices.}} and 
\begin{equation}\label{e:radonintro}
\mathcal{R}\varphi (\ell) := \int_\ell \varphi(x) \mathcal{H}^1(dx), \quad \ell \in \AG_d,
\end{equation}
indicates the usual {\bf X-ray transform} of $\varphi$. See Section \ref{ss:xrayradon} for definitions, as well as Proposition \ref{p:spiegone} for a discussion of the fact that \eqref{e:radonintro} yields the following distributional representation: if $h\sim {\rm FGF}_{1/2}(\R^d)$, then $ h = \mathcal{R}^* \widetilde{\mathbb{W}}$, where $\widetilde{\mathbb{W}}$ indicates a white noise on $\AG_d$ with intensity measure $\mu$ (see \eqref{eq:C12d}), and $\mathcal{R}^*$ is the {\bf adjoint X-ray transform} operator; see also Figure \ref{fig:fgf-wn}.

 \begin{figure}[htbp]
\centering
\includegraphics[width=0.30\textwidth]{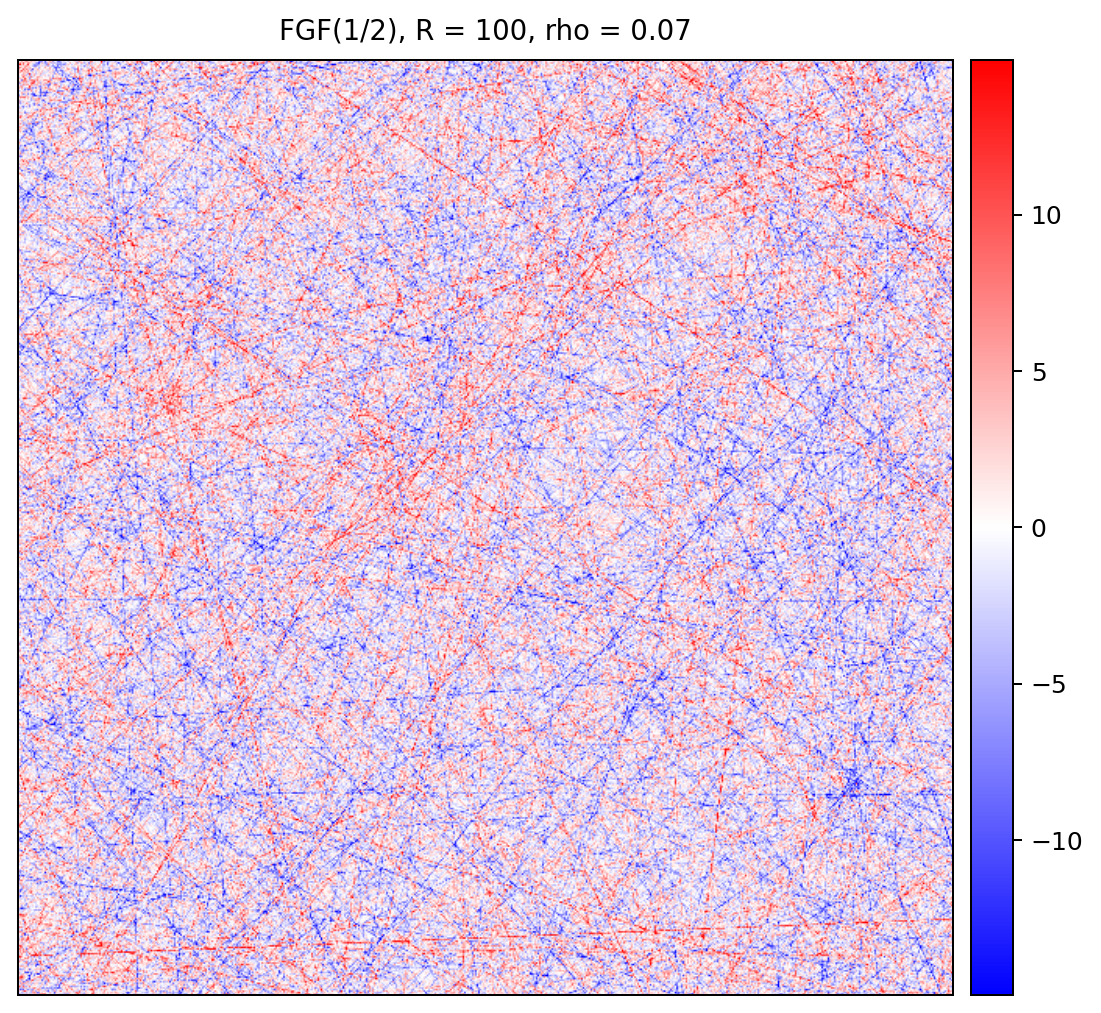}
\hspace{0.8cm}
\includegraphics[width=0.30\textwidth]{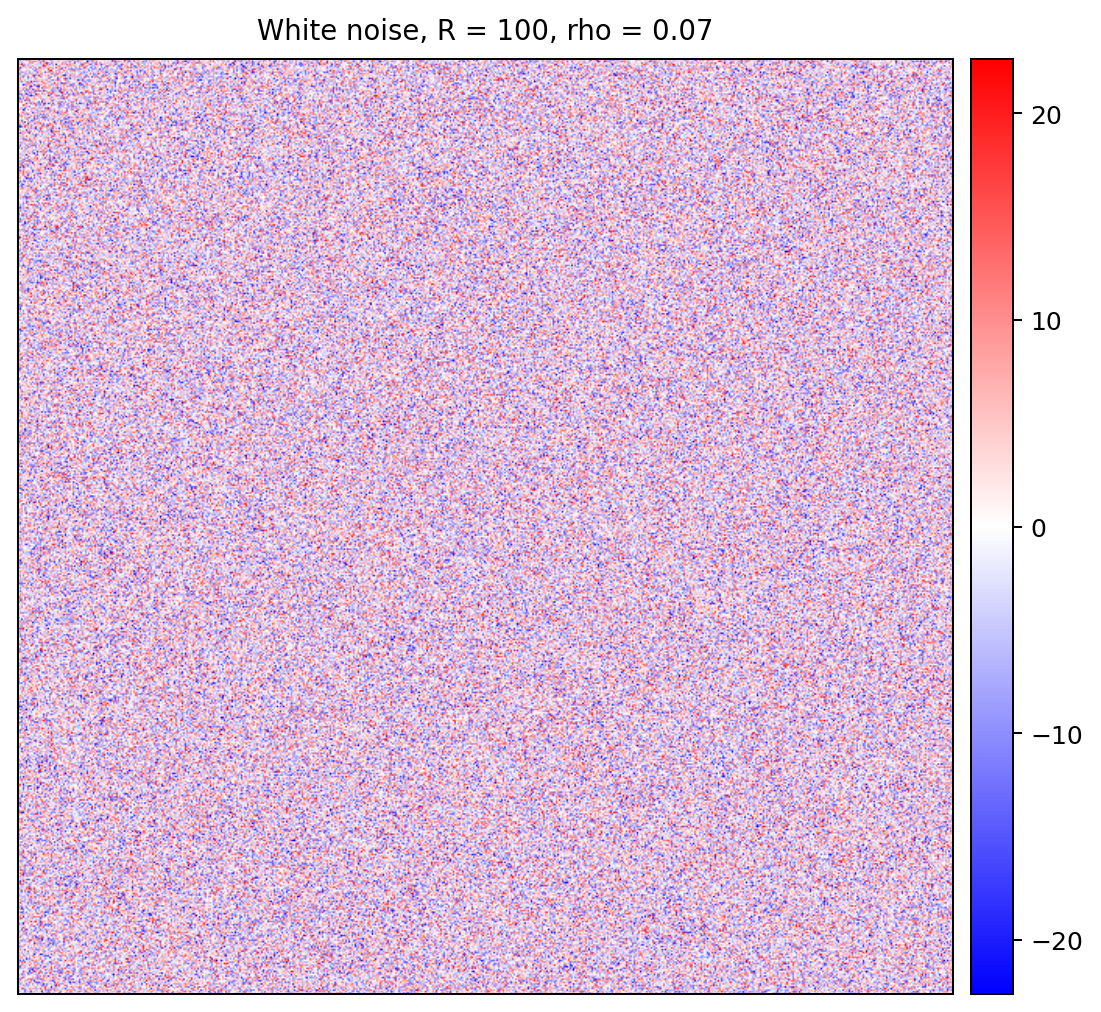}
\caption{\tiny {On the left, a simulation of the mapping $x\mapsto h(\varphi_x)$, $x\in[-100,100]^2$, where $h\sim {\rm FGF}_{1/2}(\mathbb R^2)$, and where $\varphi_x$ is an approximation of the Dirac mass at $x$, compactly supported in the ball $B(x,\varrho)$, with $\varrho=0.07$. On the right, a simulation of the mapping $x\mapsto W(\varphi_x)$, $x\in[-100,100]^2$, where $W$ is a white noise on $\mathbb R^2$, and $\varphi_x$ is the same Dirac mass approximation.}}
\label{fig:fgf-wn}
\end{figure}

\subsection{The ${\rm FGF}_{1/2}(\R^d)$ universality class}\label{ss:ucintro} Fractional Gaussian fields, including those of the ${\rm FGF}_{1/2}(\R^d)$ type, have appeared in recent years as scaling limits in a variety of frameworks, such as sandpile models \cite{ChiariniJaraRuszel_torus}, local statistics of Riesz gases \cite{SerfatyPeilen2025}, Boolean models with random radii \cite{BiermeEstradeKaj}, and occupation functionals of branching processes \cite{BGT_Cauchy_1, BGT_Cauchy_2}. In this paper, we instead focus on families of random distributions on $\R^d$ satisfying the following three properties: (a) their scaling limit under homothety is ${\rm FGF}_{1/2}(\R^d)$; (b) their scaling exponent is consistent with the self-similarity relation \eqref{e:selfsim}; and (c) they satisfy a law of large numbers under the usual volume-type normalization. We formalize this in the next definition through the notion of a {\bf universality class}.

\begin{definition}[The Universality Class]\label{d:UC}{\rm
 Fix $d\geq 2$, and let $X = \{X(\varphi) : \varphi\in \mathcal{S}\}$ be a random distribution on $\R^d$ such that $X(\varphi)\in L^1(\mathbb{P})$, for all $\varphi\in \mathcal{S}$. We say that $X$ {\bf belongs to the universality class of} ${\rm FGF}_{1/2}(\R^d)$, written
 $$
 X\in {\bf UC}\left\{{\rm FGF}_{1/2}(\R^d)\right\},
 $$
 if there exist $\alpha\in \R$ and $\beta^2>0$ such that
 \begin{enumerate}
 \item[\bf (a)] as $R\to \infty$,
 \begin{equation}\label{e:UCmean}
\mathbb{E}[X(\varphi_R)] = \alpha R^d \int_{\R^d} \varphi(x)\dd x + o(R^{(d+1)/2}), \quad \forall \varphi\in \mathcal{S},
 \end{equation}
where we have adopted the notation \eqref{e:rescaled}, and the constants implicit in the $o(\cdot)$ notation might depend on $\varphi$;
\item[\bf (b)] again as $R\to \infty$,
\begin{equation}\label{e:UCscaling}
\widetilde{X}_R(\varphi):=\frac{X(\varphi_R) - \alpha R^d \int_{\R^d}\varphi(x)\dd x }{\beta R^{(d+1)/2} }\stackrel{\rm law}{\longrightarrow} h(\varphi), \quad \forall \varphi\in \mathcal{S},
 \end{equation}
 where $h\sim {\rm FGF}_{1/2}(\R^d)$.
 \end{enumerate}
 }
\end{definition}
Note that the remainder term in {\bf (a)} is chosen so as to be negligible at the fluctuation scale appearing in {\bf (b)} and that, in most examples considered in this paper, the $o(R^{(d+1)/2})$ term will actually be equal to zero. As a consequence of tightness and of the (already recalled) results from \cite{Fernique1968, Bierme_CSA}, one verifies immediately that, if $X\in {\bf UC}\left\{{\rm FGF}_{1/2}(\R^d)\right\}$, then
\begin{equation}\label{e:llnuc}
    \frac{X(\varphi_R) }{R^d} \stackrel{\mathbb{P}}{\longrightarrow} \alpha \int_{\R^d}\varphi(x)\dd x, \quad R\to\infty,
\end{equation}
where $\stackrel{\mathbb{P}}{\longrightarrow}$ indicates convergence in probability, and
\begin{equation}\label{e:fullcovuc}
\left\{ \widetilde{X}_R(\varphi) : \varphi\in \mathcal{S}\right\} \stackrel{\rm law}{\longrightarrow} h\sim {\rm FGF}_{1/2}(\R^d), \quad R\to\infty,
\end{equation}
where we have used the notation \eqref{e:UCscaling}, and the convergence is in the sense of random distributions. Also, one easily deduces from \eqref{e:cumulantspoisson} that, if $X\in {\bf UC}\left\{{\rm FGF}_{1/2}(\R^d)\right\}$ and if $\Pi_t$ is a homogeneous Poisson process with intensity $t>0$ independent of $X$, then
$$
X+\tau \Pi_t\in {\bf UC}\left\{{\rm FGF}_{1/2}(\R^d)\right\}, \quad \forall \tau\in \R,
$$
that is, ${\rm FGF}_{1/2}$ universality classes are {\it stable} with respect to the addition of an independent Poisson noise.\footnote{Plainly, the stability of ${\rm FGF}_{1/2}$ classes is still verified if the Poisson random measure $\Pi_t$ is replaced by any independent random distribution $M(\cdot)$ such that the quantity $M(\varphi_R)$ has expectation and variance that scale, respectively, as in \eqref{e:UCmean} and as $o(R^{d+1})$. For instance, one can take $M(\varphi) = \int_{\R^d} \varphi(x) \, \dd x$.  }

\begin{remark}
We emphasize that Definition \ref{d:UC} does not require the covariance
structure of $\varphi\mapsto X(\varphi_R)$ to converge to that of $h$.
While such a covariance convergence is satisfied by all examples
considered in the present paper, we do not regard it as a defining
feature of a universality class, whose role is primarily to encode
common scaling limits. In fact, all random distributions
$X\in {\bf UC}\{\FGF_{1/2}(\R^d)\}$ studied below satisfy
$X(\varphi)\in L^2(\mathbb P)$ for every $\varphi\in\mathcal S$, and
moreover
\begin{equation}\label{e:covarianceUC}
\frac{1}{\beta^2 R^{d+1}}
{\bf Cov}\left(X(\varphi_R),X(\psi_R)\right)
\sim
{\bf Cov}\left(\widetilde{X}_R(\varphi),\widetilde{X}_R(\psi)\right)
\longrightarrow
\mathbb{E}[h(\varphi)h(\psi)],
\quad R\to\infty,
\end{equation}
where ${\bf Cov}$ stands for the usual covariance between random
variables. Likewise, although Definition \ref{d:UC} also includes a first-order
condition describing the asymptotic behavior of the mean, the
terminology ${\bf UC}\{\FGF_{1/2}(\R^d)\}$ is motivated by the
fluctuation theory: the normalization $R^{-(d+1)/2}$ appearing in
Definition \ref{d:UC} is dictated by the self-similarity properties of
$\FGF_{1/2}(\R^d)$, and every member of the class exhibits fluctuations
governed by this limiting field. In this sense, the defining feature of
the class is the behaviour of its fluctuations rather than its
first-order asymptotics.
\end{remark}

\smallskip

As already evoked in Section \ref{ss:overviewintro}, the universality classes introduced in Definition \ref{d:UC} contain some of the canonical observables associated with a family of classical geometric models, namely the class of {\it stationary Poisson line processes} in $\R^d$. To define such a model for $d\geq 2$ (see Section \ref{ss:poissonstuff} for a full discussion, as well as \cite[Chapter 4]{HugSchnieder} or \cite[Section 4.4]{SchneiderWeil} for standard presentations of these objects), we consider as before the affine Grassmannian of lines $\AG_d$ endowed with its invariant measure $\mu$; then, a {\bf stationary Poisson line process} in $\R^d$ with intensity $t>0$ is a Poisson point process on $\AG_d$ with control measure given by $t \cdot \mu$. In what follows, we write $\eta_t \sim {\rm PLP}_t(\R^d)$ to indicate that $\eta_t$ is a point process on $\AG_d$ distributed as a stationary Poisson line process with intensity $t$. The reader is referred to Figure \ref{fig:poisson-lines} for a visualization of a Poisson line process restricted to a sequence of growing domains. 

\begin{figure}[htbp]
\centering
\includegraphics[width=0.9\textwidth]{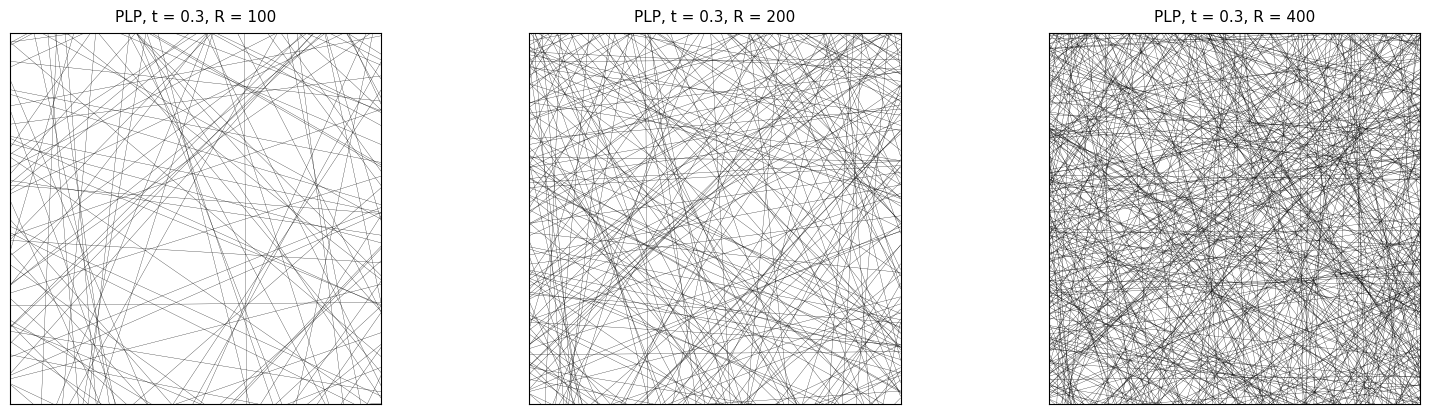}
\caption{\tiny One realization of a stationary Poisson line process with intensity $t = 0.3$, over the window $[-R,R]^2$, for $R=100, 200, 400$. }
\label{fig:poisson-lines}
\end{figure}

\begin{remark}\label{r:locfinite}{\rm
It is important to notice that Poisson line processes are {\it locally finite} objects in the sense that, if $\eta_t \sim {\rm PLP}_t(\R^d)$, then, with probability one, every bounded set $A\subset \R^d$ intersects only finitely many lines in the support of $\eta_t$ (this is a direct consequence of the fact that $\mu$ is locally finite). In this respect, Poisson line processes of the kind considered in this paper are substantially different from the {\bf scale-invariant Poisson stick soups} recently introduced in \cite{teixeira2026stick}, whose support is instead dense in $\R^d$. }
\end{remark}

Given $\eta_t \sim {\rm PLP}_t(\R^d)$, we introduce the associated {\bf Poisson Line Field} on $\R^d$, which is defined as the random distribution on $\R^d$ given by
\begin{equation}\label{e:radonsoup1}
\mathcal{R}^* \eta_t = \{ \mathcal{R}^*\eta_t(\varphi) : \varphi \in \mathcal{S} \}
\end{equation}
where
\begin{equation}\label{e:radonsoup2}
\mathcal{R}^* \eta_t(\varphi) := \sum_{\ell\in \eta_t} \mathcal{R}\varphi(\ell),
\end{equation}
$\mathcal{R}\varphi( \ell)$ is given in \eqref{e:radonintro}, $\Rad^*$ is the adjoint of $\Rad$, and $\ell \in \eta_t$ indicates that $\ell\in \AG_d$ is in the support of $\eta_t$.

The next statement (that, to the best of our knowledge, is new) shows that $\mathcal{R}^* \eta_t$ belongs to the ${\rm FGF}_{1/2}$ universality class, and that the raw\footnote{That is, without centering and rescaling.} fluctuations of any element of this class can be made arbitrarily close, in the sense of the discrepancy of finite-dimensional marginals evaluated in the convex distance, to those of a random distribution of the form
\[
\mathcal{R}^* \eta_{t_1} \pm \Pi_{t_2},
\]
for suitable $t_1>0$ and $t_2\geq 0$, with $\Pi_{t_2}$ an independent Poisson point process. See Section \ref{ss:PoissonProof} for a proof.

\smallskip

\begin{theorem}\label{t:poisson}
\begin{enumerate}
\item[\bf (a)] For every $d\geq 2$ and every $t>0$, one has that
$$
\mathcal{R}^* \eta_t \in {\bf UC}\left\{{\rm FGF}_{1/2}(\R^d)\right\},
$$
with $\alpha =t\gamma$ and $\beta^2 = C(d)^{-1}t$, where $\gamma = \gamma(d)$ is a constant uniquely depending on $d$, and $C(d)$ has been defined in \eqref{eq:C12d}.
\item[\bf (b)] Let $Y\in {\bf UC}\left\{{\rm FGF}_{1/2}(\R^d)\right\}$. Then, there exists a unique\footnote{Except for when $t_2 = 0$, in which case the choice of $\varepsilon$ is immaterial.} triple $$(t_1, \varepsilon, t_2)\in (0,\infty)\times\{-1,1\}\times [0,\infty)$$ such that, for all $m\geq 1$ and all $\varphi^{(1)},...,\varphi^{(m)}\in \mathcal{S}$ such that the covariance matrix of $(h(\varphi^{(1)}),..., h(\varphi^{(m)}))$ is non-singular, one has that 
\begin{equation}\label{e:convexbeast}
d_{\rm conv}\left( {\bf Y}_R, {\bf L}_R\right)\longrightarrow 0,\quad R\to\infty,
\end{equation}
where
\begin{eqnarray*}
&& {\bf Y}_R := \left(Y(\varphi^{(1)}_R), \ldots, Y(\varphi^{(m)}_R)\right),\\
&& {\bf L}_R := \left(\mathcal{R}^* \eta_{t_1}(\varphi^{(1)}_R) +\varepsilon\, \Pi_{t_2}(\varphi_R^{(1)}), \ldots, \mathcal{R}^* \eta_{t_1}(\varphi^{(m)}_R) +\varepsilon\, \Pi_{t_2}(\varphi_R^{(m)})\right),
\end{eqnarray*}
the random distributions $\mathcal{R}^* \eta_{t_1}$ and $\Pi_{t_2}$ are stochastically independent, and we have adopted the notation \eqref{e:rescaled} and \eqref{e:convexd}. 
\end{enumerate}
\end{theorem}

\smallskip

\begin{remark}{\rm We stress that the families of probability laws on $\R^m$ associated with the classes $\{{\bf Y}_R : R>0\}$ and $\{{\bf L}_R : R>0\}$ appearing in \eqref{e:convexbeast} are {\it not} tight. See Davydov and Rotar \cite{DavidovRotarJTP} for a general reference on the topic of measuring the proximity of non-tight families of probability distributions on metric spaces.

}
\end{remark}

We are now in a position to state the main results of the paper.

\subsection{Main results}\label{ss:mainintro}

\subsubsection{Scaling limit of BRW observables, and comparison with other fields}\label{ss:fixedobservablesintro} Fix $d\geq 2$, and let $B = \{B(x) : x\in \R^d\}$ be the BRW model defined in \eqref{e:helmholtz}--\eqref{e:berrycov}. The next statement (proved in Section \ref{e:proofsquarecorr}) is one of the main achievements of our work. Recall that an {\it observable} of $B$ is a $\sigma(B)$-measurable random distribution, where $\sigma(B)$ is the $\sigma$-field generated by $B$.

\begin{theorem}\label{t:squarecorr} Let $X = \{X(\varphi) : \varphi\in \mathcal{S}\}$ be an observable of $B$ such that, for all $\varphi\in \mathcal{S}$, $X(\varphi)\in L^2(\mathbb{P})$ and \eqref{e:UCmean} is verified for some $\alpha\in \R$. Moreover, assume that there exist real constants $\{c_j : j=0,1,...,k\}$ such that $\sum_{j=0}^k c_j\neq 0$ and
\begin{align}\label{e:fullcorr}
X(\varphi_R) - \mathbb{E}[X(\varphi_R)] &= \sum_{j=0}^k c_j\int_{\R^d} \varphi_R(x) \left(\|\nabla^j B(x)\|^2 - 1\right)\dd x +T_R,
\end{align}
where we used the notation \eqref{e:rescaled} and
\begin{equation}\label{e:trpiccolo}
{\bf Var}(T_R) = o(R^{d+1}), \,\,\mbox{as}\,\, R\to \infty.
\end{equation}
Then, $X\in {\bf UC}\left\{{\rm FGF}_{1/2}(\R^d)\right\}$, with $\beta = C(d)^{-1/2}\,K_d \, | \sum c_j |$, where $C(d)$ is defined in \eqref{eq:C12d}, and
\begin{equation}\label{e:kappamaiuscolo}
K_d := \frac{2^{\frac{d-1}{2}}\Gamma\left(\frac{d}{2}\right)}{\sqrt{\pi}}, \quad d\geq 2.
\end{equation}
\end{theorem}

\medskip 

We note that Theorem~\ref{t:squarecorr} yields that the random distribution generated by the density \(x \mapsto B^2(x)\) --- corresponding to the choice \(k=0\), \(c_0=1\), and \(T_R=0\), and illustrated in Figure~\ref{fig:brw-four-panels}-(a) --- belongs to the universality class
\(
{\bf UC}\left\{{\rm FGF}_{1/2}(\R^d)\right\}.
\)
The forthcoming Proposition \cref{p:squarecorrex} shows that the same conclusion holds for the three remaining observables displayed in Figure~\ref{fig:brw-four-panels}. For a smooth function $f$ and $u\in \R$, we define the following quantities:

\smallskip

\begin{itemize}
    \item[--] $X_1(f,u,\varphi) := \sum_{x\in {\bf CP}(u)} \varphi(x)$, where ${\bf CP}(u)$ indicates the collection of all critical points $x$ of $f$ such that $f(x)>u$;
    \item[--] $X_2(f,u,\varphi) := \sum_{x\in {\bf M}(u)} \varphi(x)$, where ${\bf M}(u)$ indicates the collection of all local maxima $x$ of $f$ such that $f(x)>u$;
    \item[--] $X_3(f,u,\varphi) := \int_{f^{-1}(u)} \, \varphi(x)\, \dd \mathcal{H}^{d-1}(x)$, where $\mathcal{H}^{d-1}$ is the $(d-1)$-dimensional Hausdorff measure;
    \item[--] $X_4(f,u,\varphi) := \int_{\R^d} {\bf 1}_{\{f(x) >u\}}\, \varphi(x)\, \dd x$.
\end{itemize}

\medskip

\begin{remark}\label{r:leonardo}{\rm The fact that, after the appropriate normalization, the asymptotic variance of
\[
\int_{\mathbb R^2}\varphi_R(x)(B(x)^2- 1)\,\dd x
\]
coincides with the covariance functional of ${\rm FGF}_{1/2}(\mathbb R^2)$ can already be deduced from \cite[formula (16)]{Barnett_CPAM} (see also \cite[Section III.B]{Keating}). Moreover, a version of Theorem \ref{t:squarecorr} in the case where $d=2$, $k=0$ and $\varphi$ runs over the class of indicators of bounded convex sets with non-vanishing Gauss curvature can be inferred from \cite[Example 4.10]{MainiCov}, a result that is partially based on the one-dimensional spectral CLTs proved in \cite{maini2024spectral}.} 
\end{remark}

\medskip

\begin{proposition}\label{p:squarecorrex} Let the above notation prevail. For $\ell = 1,...,4$, there exist finite subsets $E_\ell \subset \R$ {such that $0\in E_\ell$} and, for $u\in \R\setminus E_\ell$, the observable $X_\ell(B,u,\cdot)$ is a member of the class ${\bf UC}\left\{{\rm FGF}_{1/2}(\R^d)\right\}$. For these observables, {the convergence relations \eqref{e:UCmean}--\eqref{e:UCscaling} as well as the conclusion of Theorem \ref{t:poisson}-{\bf (b)} continue to hold} if one replaces the test functions $\varphi, \varphi^{(i)}\in \mathcal{S}$ with indicators of bounded convex sets having a $C^\infty$ boundary with non-vanishing Gauss curvature; {in this case, one has to interpret $h(\varphi)$ and $(h(\varphi^{(1)}), ...., h(\varphi^{(m)}))$ as centered Gaussian elements whose variance/covariance structures are given by \eqref{e:fgfcov1}.}
\end{proposition}
\begin{remark}\label{r:lafolie}{
 \item[(i)] The combination of Theorem~\ref{t:squarecorr}, Proposition~\ref{p:squarecorrex} and Theorem~\ref{t:poisson}-{\bf (b)} shows that, in the large-domain limit, the fluctuations of the raw observables generated by Berry's random wave and those generated by a homogeneous (possibly noisy) Poisson line process become asymptotically indistinguishable, both as random tempered distributions and, via their measure-valued extensions, when evaluated on indicators of smooth convex sets. An interesting question is whether asymptotic indistinguishability, in the above sense, fully accounts for the filamentary patterns observed in numerical simulations. The subtlety of this issue is illustrated by the recent work of Beliaev and Hegde~\cite[Theorem~2.4 and Remark~2.6]{beliaevHehgde2026}. They consider a stationary Gaussian field on $\R^2$ whose spectral density satisfies
\[
\rho(\xi)\asymp\frac{\mathbf 1_{\{|\xi|<1\}}}{|\xi|},
\]
and prove that the fluctuations of the excursion areas (at any level $u$) belong to the $\FGF_{1/2}(\R^2)$ universality class. Nevertheless, numerical simulations of this field do not exhibit the striking scar-like filamentary patterns observed for Berry's random wave; see Figure~\ref{fig:loosers}. It is plausible that Berry's random wave and homogeneous Poisson line processes are asymptotically closer under a stronger mode of convergence, one that would allow for a formal statistical detection of scarred patterns.

\begin{figure}[htbp]
\centering
\includegraphics[width=0.30\textwidth]{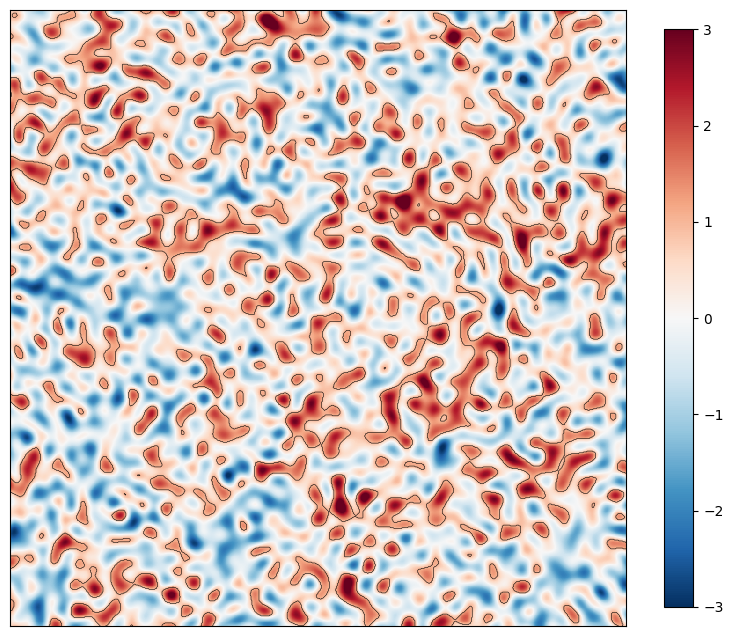}
\hspace{0.8cm}
\includegraphics[width=0.30\textwidth]{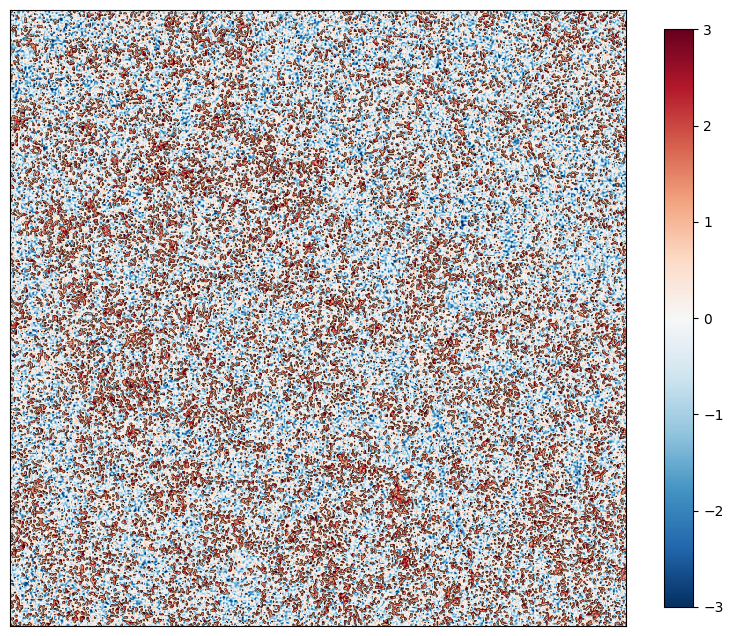}
\caption{\tiny The singular random field from \cite{beliaevHehgde2026}, as evoked in Remark \ref{r:lafolie}-(i), over the window $[-R,R]^2$, for $R=100, 700$.}
\label{fig:loosers}
\end{figure}

\item[(ii)] We stress that the conclusions of Proposition~\ref{p:squarecorrex} hold at a {\it fixed} level $u$, suggesting, in particular, that one mechanism contributing to the formation of scarred patterns is already present in the bulk of critical points and maxima. Figure \ref{fig:max_large} provides an illustration of this fact.

\begin{figure}[htbp]
\centering
\includegraphics[width=0.27\textwidth]{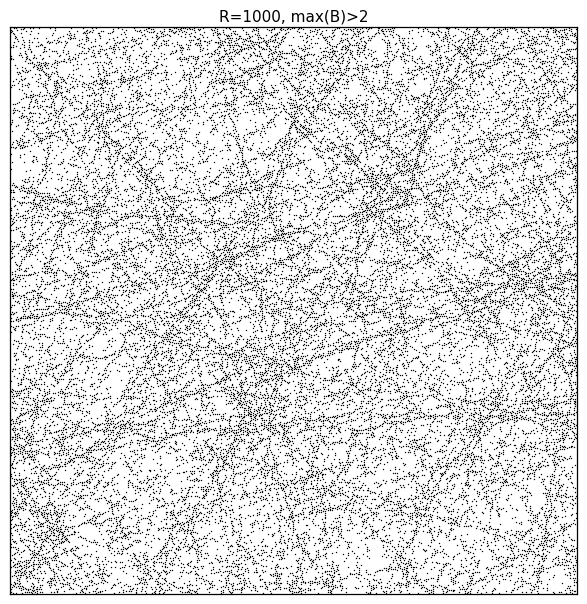}
\hspace{1.3cm}
\includegraphics[width=0.27\textwidth]{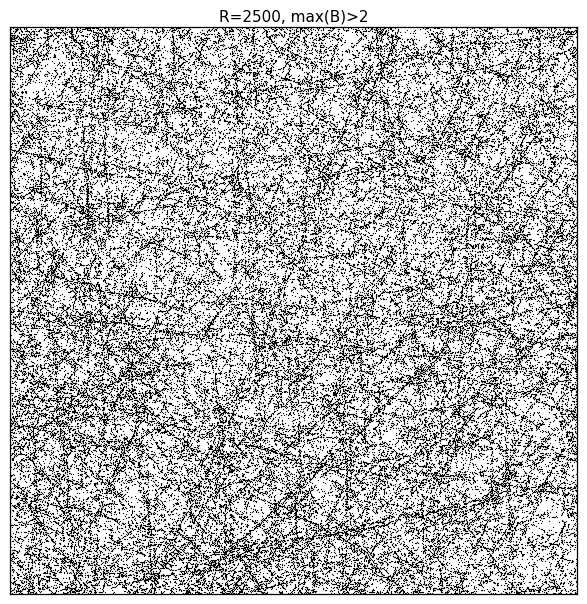}
\caption{\tiny Simulation of the local maxima of BRW above the threshold $u=2$ over the window $[-R,R]^2$, for $R=1000, 2500$.}
\label{fig:max_large}
\end{figure}

\item[(iii)] {Relation~\eqref{e:fullcorr} implies that the random variable $X(\varphi_R)$ is asymptotically dominated by its projection onto the {\bf second Wiener chaos} associated with $B$ (see \cite[Chapter~2]{nourdinpeccatibook} for definitions). In particular, the conclusion of Theorem~\ref{t:squarecorr} {\it does not apply} to observables exhibiting a form of anti-symmetry or of {\it Berry's cancellation}, such as those studied for instance in \cite{Berry2002a, Wigman_2010, NourdinPeccatiRossi2019, PV, NPV, cgv2025StecconiTodino, Marinucci2016, Marinucci2020, WigmanAnnMath, gass2025spectralcriteriaasymptoticslocal}, which are typically fully correlated with their projection onto the {\it fourth} Wiener chaos. This is consistent with the fact that, according for example to \cite{Nou19, PV, NPV}, the scaling limit of the centered and rescaled nodal length associated with $B$ (corresponding to $X_3(B,0,\cdot)$) is given by white noise (see \cite[Proposition~1.3]{NPV} for a precise statement), as well as with the empirical observation that scars are not detected by nodal patterns; see Figure~\ref{fig:brw_nodal}. A similar observation can be made for the point process of all critical points of BRW, formally corresponding to the functional $X_1(B,-\infty,\cdot)$ (see \cite[Sec. 1.3]{gass2025spectralcriteriaasymptoticslocal}); see also \cite[Section 2.6]{UniformRW}. {Note that the above arguments straightforwardly justify the fact that, in Theorem \ref{t:squarecorr}, one has that $0\in E_1\cap  E_3\cap E_4$. The fact that, also, $0\in E_2$ can be deduced from a standard homogeneity argument that is left to the reader.}}
\item[(iv)] It is instructive to reinterpret Theorem~\ref{t:squarecorr} (case $k=0$) in a high-energy regime, using the rescaled field
\[
B_R(x) := B(Rx), \qquad x\in \R^d,
\]
which satisfies the equation
\[
\Delta B_R(x) + R^2 B_R(x) = 0.
\]
In this case, Theorem~\ref{t:squarecorr} implies the statements
\begin{equation}\label{e:BRLLN}
\int_{\R^d} B_R^2(x)\,\varphi(x)\,dx \xrightarrow{\mathbb{P}} \int_{\R^d} \varphi(x)\,dx, \qquad R\to\infty,
\end{equation}
and
\begin{equation}\label{e:CLT}
c\,R^{(d-1)/2}\int_{\R^d} \big(B_R^2(x)-1\big)\,\varphi(x)\,dx \xrightarrow{\rm law} h(\varphi),
\end{equation}
for some $c>0$ and $h\sim {\rm FGF}_{1/2}(\R^d)$. These relations show that, although $B_R^2$ equidistributes in the high-energy limit $R\to\infty$, the natural scale at which any large-scale filamentary behaviour may emerge is that of second-order fluctuations, namely $R^{-(d-1)/2}$}.
\end{remark}

\medskip

The techniques developed in Section \ref{s:observableproofs} will also allow us to determine a precise asymptotic relation for the variance of $X_1(B,u, \varphi_R)$, when $u,R$ tend simultaneously to infinity. To state such a result, we define $X_1^{(2)}(B,u, \varphi_R)$ to be the projection on the second Wiener chaos of $X_1(B,u, \varphi_R)$ (see Section \ref{ss:chaosProof1} for more details). To simplify the discussion, we also introduce the shorthand notation
\begin{equation}\label{e:lnormINTRO}
\|\varphi\|_L^2 := \int_{\R^d}\int_{\R^d}\frac{\varphi(x)\varphi(y)}{\|x-y\|^{d-1}}\dd x\dd y, \end{equation}
and observe that, according to \eqref{e:fgfcov1} and using \eqref{e:lnorm} one has that, if $h\sim{\rm FGF}_{1/2}(\R^d)$ and $\varphi\in \mathcal{S}$, then ${\bf Var}( h(\varphi) ) = C(d)\,  \|\varphi\|^2_L$. 
\begin{proposition}
\label{lemm14}
Let the above notation prevail, fix $d\geq 2$, and let $\varphi$ be either an element of $\mathcal{S}$ or a {finite linear combination of indicators of} bounded convex sets with a $C^\infty$ boundary and non-vanishing Gauss curvature.
Assume that $u,R\to\infty$, in such a way that, for some positive constant $c$,
\begin{equation}\label{e:lowboundc}|\E[X_1(B,u, \varphi_R)]|\geq c\left|\int_{\R^d} \varphi(x)\dd x\right|.\end{equation}
Then,
\begin{align*}
\Var(X_1(B,u, \varphi_R)) &\simeq \Var\left(X_1^{(2)}(B,u, \varphi_R)\right)  + \E[X_1(B,u, \varphi_R^2)],\\
&\simeq \frac{\Gamma\left(\frac{d}{2}\right)^2}{16d^d\pi^{d+2}}\, R^{d+1}u^{2d+2}e^{-u^2}\|\varphi\|_L^2  + \frac{1}{d^{d/2}(2\pi)^{\frac{d+1}{2}}}R^du^{d-1}e^{-\frac{u^2}{2}}\int_{\R^d} \varphi(x)^2\dd x.
\end{align*}
\end{proposition}

\smallskip

\begin{remark}
{\rm By virtue of the forthcoming Lemma \ref{lemm5}, Condition
\eqref{e:lowboundc} is equivalent to
\[
\frac{1}{d^{d/2}(2\pi)^{\frac{d+1}{2}}}
R^d u^{d-1} e^{-u^2/2}\ge c.
\]
Under this assumption, three asymptotic regimes emerge for the number
of critical points above a growing level $u$.

\smallskip

\noindent
{\bf (i)} If $u$ grows sufficiently slowly, namely
\[
\frac{1}{u^{d+3}}e^{u^2/2}=o(R),
\]
then the second Wiener chaos projection dominates. In this regime, one
can adapt the proof of Theorem \ref{t:squarecorr} to show that the
$\FGF_{1/2}$ scaling limit remains in force, namely

\[
\sqrt{C(d)}
\frac{X_1(B,u,\varphi_R)-\E[X_1(B,u,\varphi_R)]}
{\sqrt{H(R,u)}}
\stackrel{\rm law}{\longrightarrow} h(\varphi),
\qquad h\sim \FGF_{1/2}(\R^d),
\]
where $H(R,u):= \frac{\Gamma\left(\frac{d}{2}\right)^2}{16d^d\pi^{d+2}}\, R^{d+1}u^{2d+2}e^{-u^2}$.
\smallskip

\noindent
{\bf (ii)} If $u$ grows sufficiently fast, namely
\[
R=o\!\left(\frac{1}{u^{d+3}}e^{u^2/2}\right),
\]
then the variance is expected to exhibit the standard Poisson scaling.
In view of the estimates provided by Lemmas \ref{lemm8}, \ref{lemm9}
and \ref{lemm10}, it is plausible that the techniques developed in
\cite{Aza09,beliaevHegde2026} could be used to show that the point
process of critical points above level $u$ becomes asymptotically
indistinguishable from a Poisson point process with intensity
\[
\frac{1}{d^{d/2}(2\pi)^{\frac{d+1}{2}}}
R^d u^{d-1} e^{-u^2/2}.
\]

\smallskip

\noindent
{\bf (iii)} In the intermediate regime
\[
R\asymp \frac{1}{u^{d+3}}e^{u^2/2},
\]
the two mechanisms appear to contribute at the same scale. This
suggests the conjecture that the limiting critical point measure should
exhibit a mixed behavior, combining a Poisson component and an
independent $\FGF_{1/2}(\R^d)$ component.}
\end{remark}

\medskip 

The forthcoming Theorem \ref{t:integrablefields} demonstrates that, for stationary fields with integrable correlation function, the natural scaling limit of the observables studied in Proposition \ref{p:squarecorrex} is white noise. In particular, this result applies to stationary random fields whose spectral density is supported in an annulus. Figure \ref{fig:annulus} provides an
illustration of this phenomenon.

\begin{figure}[htbp]
\centering
\includegraphics[width=0.24\textwidth]{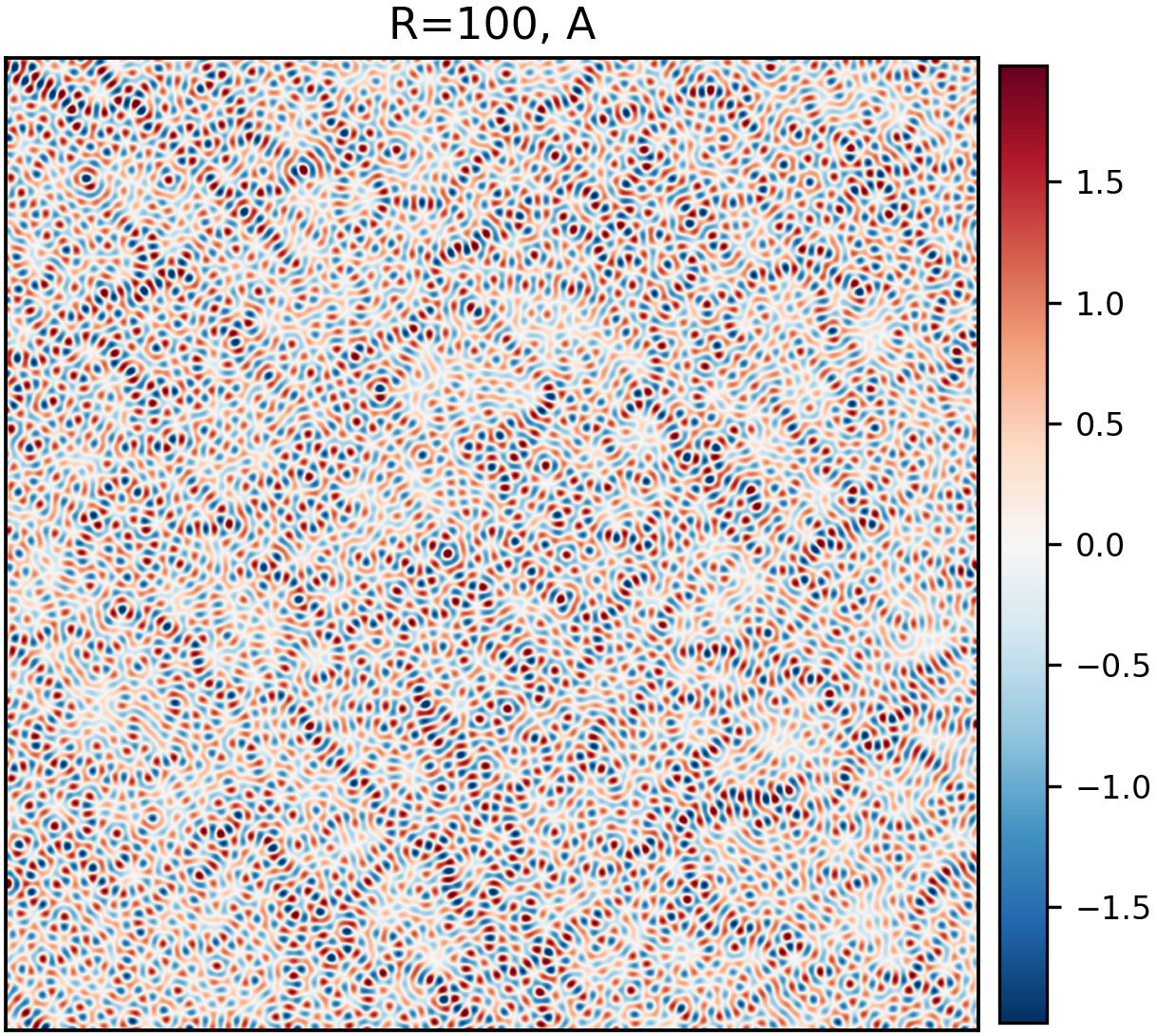}
\hspace{0.4cm}
\includegraphics[width=0.24\textwidth]{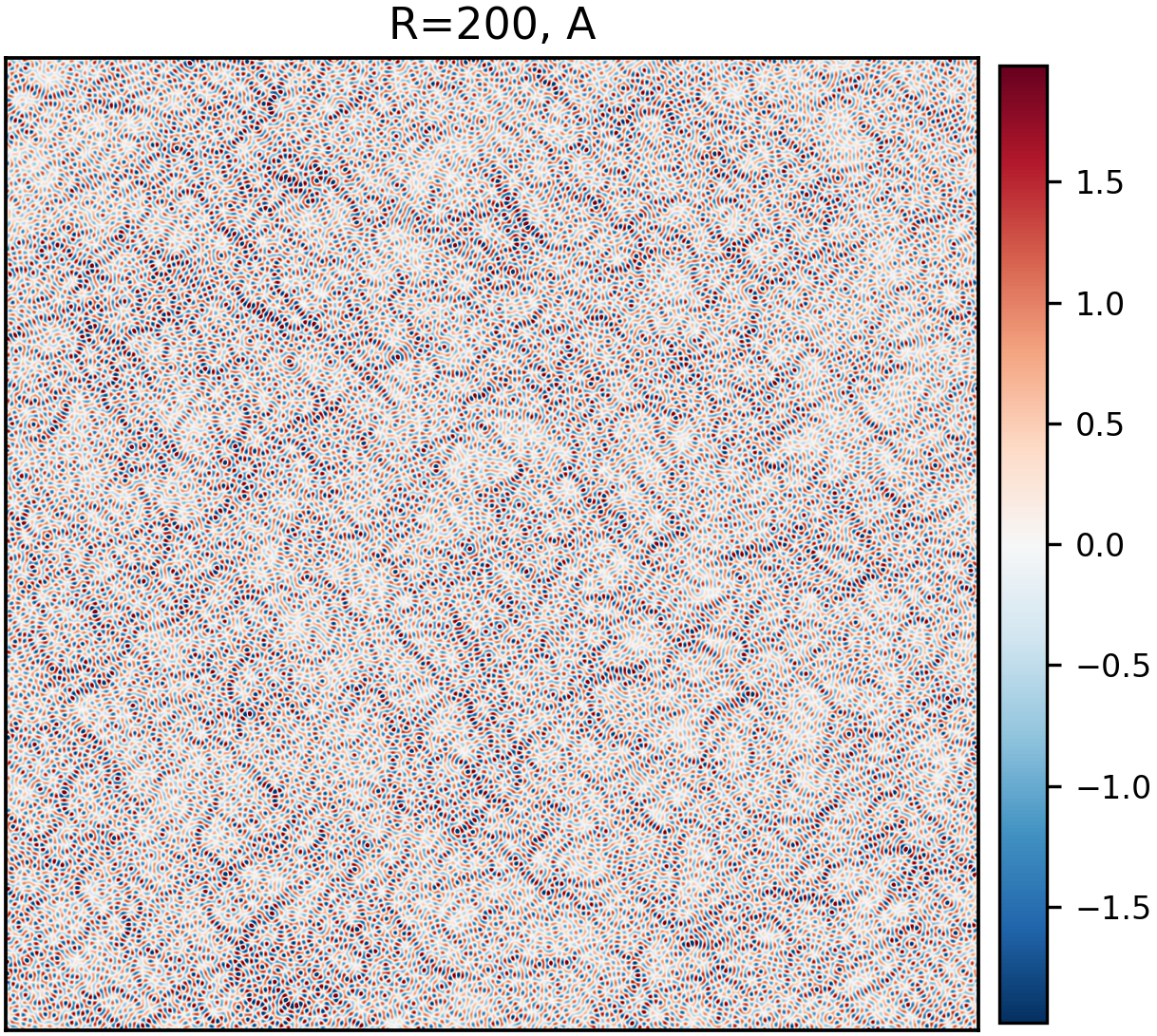}
\hspace{0.4cm}
\includegraphics[width=0.24\textwidth]{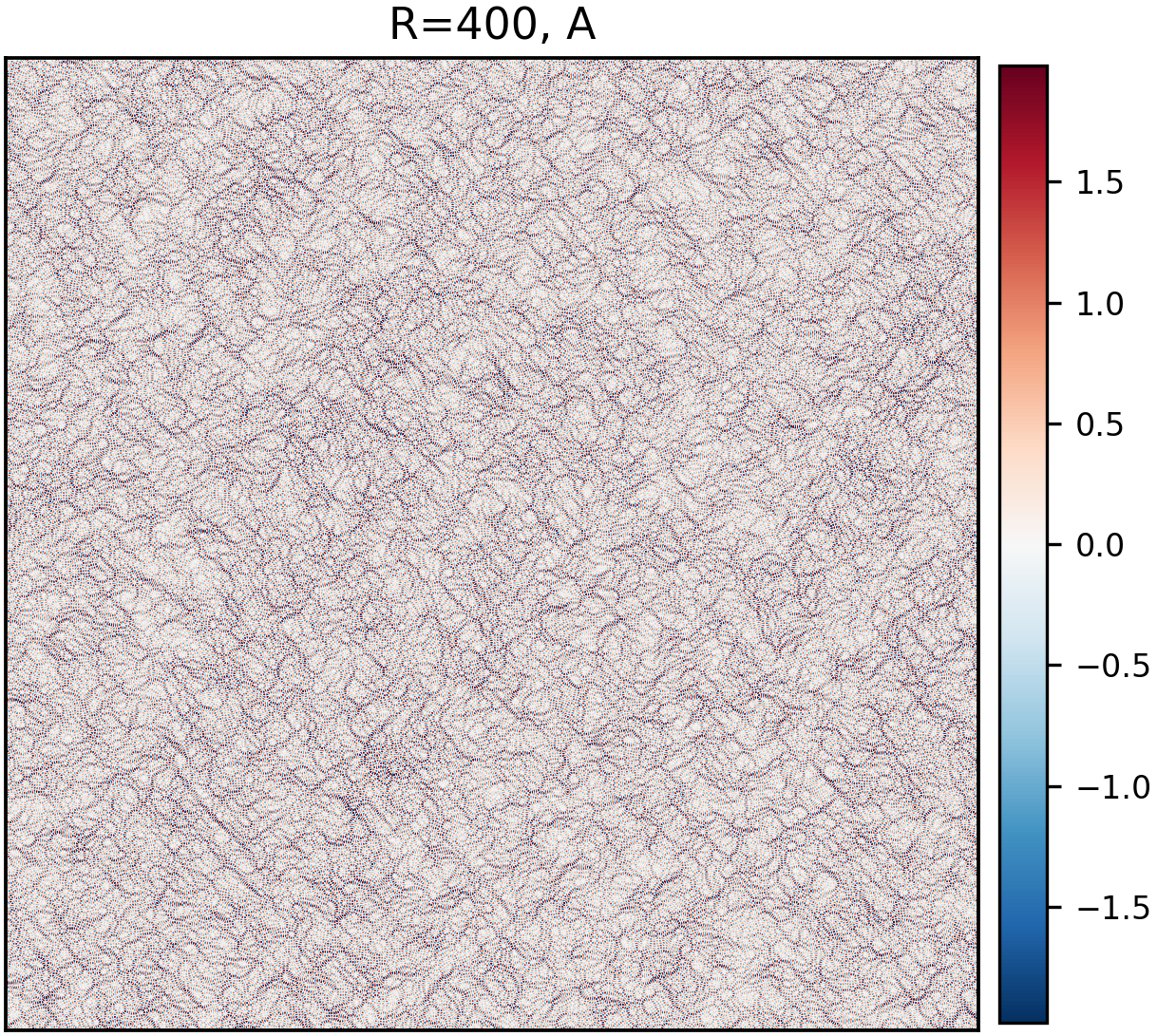}
\caption{\tiny Heatmap on the window $[-R,R]^2$ of a stationary Gaussian random field $A = \{A(x) : x\in \R^2 \}$ with a smooth spectral measure supported on the annulus $\{ x\in \R^2 : 1.9\leq \|x\| \leq 2.1\}$, for $R=100, 200, 400$.}
\label{fig:annulus}
\end{figure}

\begin{theorem}\label{t:integrablefields} Let $d\geq 2$ and $f$ be a stationary Gaussian field on $\R^d$ centered with unit-variance, such that $\nabla^j f\in L^1(\R^d)$ for $0\leq j\leq 4$. Then for $\ell=1,\ldots,4$ 
\begin{itemize}
\item[\rm (i)] As $R\to \infty$ and for fixed $u$, one has that, for some constants $C_\ell, C'_\ell>0$ depending on $u$, the following asymptotic relations:
\[\E[X_\ell(f,u,\varphi_R)] \simeq C_\ell R^d \int_{\R^d}\varphi(x)\dd x \quand \Var\left(X_\ell(f,u,\varphi_R)\right) \simeq C'_\ell R^d \int_{\R^d}\varphi(x)\dd x.\]
\item[\rm (ii)] As $R\to \infty$ and for fixed $u$, the random tempered distribution
\[\varphi\mapsto \frac{X_\ell(f,u,\varphi) -\mathbb{E}[X_\ell(f,u,\varphi)] }{\sqrt{C'_\ell} R^{d/2}}\]
converges in law towards a standard white noise on $\R^d$.

\item[\rm (iii)] If $R,u\to \infty$ and \eqref{e:lowboundc} is in order, then
\[\Var\left(X_1(f,u,\varphi_R)\right) \simeq \mathbb{E}[X_1(f,u,\varphi_R^2)] \simeq \frac{1}{d^{d/2}(2\pi)^{\frac{d+1}{2}}}R^du^{d-1}\int \varphi^2(x)\dd x.\]
\end{itemize}
\end{theorem}

\begin{remark}
In the setting of Theorem \ref{t:integrablefields}, that is, for
stationary Gaussian fields whose spectral measure satisfies the
non-degeneracy assumptions of that result, the limiting distribution
at fixed level $u$ is Gaussian white noise. This stands in contrast
with the Berry random wave model, where Proposition
\ref{p:squarecorrex} shows that the corresponding scaling limit is a
fractional Gaussian field $\FGF_{1/2}(\R^d)$. We further conjecture that the high-level regime exhibits the same
qualitative behaviour as in the monochromatic case. More precisely, in
view of the estimates provided by Lemmas \ref{lemm8}, \ref{lemm9} and
\ref{lemm10}, it seems plausible that the techniques developed in
\cite{Aza09,beliaevHegde2026} could be adapted to show that, as
$u\to\infty$ under Condition \eqref{e:lowboundc}, the point process of
critical points above level $u$ becomes asymptotically close in
distribution to a Poisson point process with intensity
\[
\frac{1}{d^{d/2}(2\pi)^{\frac{d+1}{2}}}
R^d u^{d-1}e^{-u^2/2}.
\]
\end{remark}

\subsubsection{Beyond white scars: scaling limits of random distributions on $\R^d\times \R$ }\label{ss:livingcoloursintro}
Fix $d\geq 2$. In what follows, we denote by $\mathcal A_d$ the set of all pairs
$(\ell,v)$ such that $\ell\in\AG_d$ and $v$ belongs to the
one-dimensional linear subspace parallel to $\ell$; see Definition
\ref{d:affgrass}. Equivalently, every element of $\mathcal A_d$ can be
written in the form $(\ell,\eta u)$, where $\ell=r+\R u\in\AG_d$,
$u\in S^{d-1}$ and $\eta\in\R$. As demonstrated in Section \ref{ss:xrayradon}, the space $\mathcal A_d$ provides the natural domain of the {\bf X-ray
Fourier transform}, as introduced in Definition \ref{d:xray}: a point $(\ell,\eta u)\in\mathcal A_d$ records both the
geometric component of the transform, through the affine line $\ell$,
and its oscillatory
component, through the frequency parameter $\eta$ along the direction
of $\ell$. We recall that the $X$-ray Fourier transform of a Schwartz function $\varphi\in \Sch(\R^d)$ is defined as 
\begin{equation}
\mathcal{A}_d \ni (\ell, \eta u) \mapsto \Xray\f(\ell, \eta u) := \int_\R \varphi(r+tu) e^{it\eta}\dd t;
\end{equation}
see \eqref{e:xraydef}.

Now let $f = \{f(x) : x\in \R^d\}$ be a centered stationary Gaussian field, and assume that the spectral measure $\omega(\dd \xi)$ of $f$ is
\begin{itemize}
\item[\bf (A1)] either absolutely continuous with respect to the uniform probability measure on the sphere $S^{d-1}$, with a bounded and continuous density $\rho$ (the case of BRW \eqref{e:berrycov} corresponds to $\rho = 1$), or
\item[\bf (A2)] or absolutely continuous with respect to the Lebesgue measure on $\R^d$, with a density $\rho$ such that $\sqrt{\rho}$ is a Schwartz function (for instance, the case of the Bargmann-Fock field in Figure \ref{fig:brw-bf-grid}, corresponds to $\rho(x) \asymp e^{-c\|x\|^2}$).
\end{itemize}
As anticipated in Section \ref{ss:overviewintro}, thanks to Proposition \ref{prop:Radoninv_Xrayf_to_fsquare}, the random tempered distribution $\wick{f(R\cdot)^2}$ defined by \eqref{e:sqintrointro} admits the a.s. representation 
\be\label{e:idle}
\wick{f(R\cdot)^2} \,\, = \int \dd\eta \,\, \scrd, 
\ee
where $\scrd$ is the random element of $\mathcal{S}'(\R^d\times \R)$ defined in \eqref{e:tierre}, with
\begin{itemize}
\item[--] $\mathcal{R}^*$ the adjoint of the {one-dimensional Radon transform} $\mathcal{R}\varphi(\ell) := \Xray \varphi (\ell, 0) $, canonically extended to tempered distributions on $\mathcal{A}_d$ (see Definition \ref{d:xray});
\item[--] $\sqrt{-\Delta_r}$ is a fractional Laplacian (introduced in Definition \ref{d:UC}), acting on tempered distributions on $\mathcal{A}_d$;
\item[--] $\sqrt{-\Delta_r} \wick{| \mathcal{X}_R f|^2}$ is a random tempered distribution on $\mathcal{A}_d$, obtained as the limit in probability of the (pointwise defined) random distribution $\sqrt{-\Delta_r} \wick{| \mathcal{X}_R f_\epsilon|^2}$, where $f_\epsilon$ is a suitable mollification of $f$ (see Theorem \ref{thm:Wicked_wp});
\item[--] $\int \dd \eta $ indicates the restriction of $\scrd$ to test functions of the form $\Phi(x,\eta) = \varphi(x)$ ($(x,\eta)\in \R^d\times \R$).
\end{itemize}

\medskip 

As discussed in Section \ref{ss:Z2}, the random distribution $\scrd$ should be viewed as a
frequency-dependent version of the Wick square
$\wick{f(R\cdot)^2}$ . While the latter only retains geometric
information, the additional variable $\eta$ records the oscillatory
content of the field along the direction of affine lines. Heuristically, the projection
$\int \dd\eta$ corresponds to forgetting this frequency information,
and relation \eqref{e:idle} shows that $\wick{f(R\cdot)^2}$ is formally recovered as
the frequency average of $\scrd$. The next statement corresponds to a condensed (and less general) version of Theorem \ref{thm:ultralimit}, which is one of the main achievements of our work.

\begin{theorem}\label{t:maincolouredscars}\label{t:ultralimitintro} Let the above notation and assumptions prevail.
\begin{itemize}
\item[\rm {\bf (1)}] In the case {\bf (A1)}, as $R\to\infty$, the renormalized distribution $R^{d-1}\scrd$ converges in law to a random
distribution on $\R^d\times\R$ whose action on a test function
$\Phi(x,\eta)$ is obtained by first applying a deterministic
dimension-dependent operator $E_d$ in the frequency variable $\eta$,
and then evaluating the resulting function against the random
distribution $\Rad^*(\rho\,\Wnoise_{\AG_d})$ in the spatial variable
$x$, where $\rho \Wnoise_{\AG_d}$ is a white noise on $\AG_d$ (with an intensity depending on $\rho$).
\item[\bf (2)] In the case {\bf (A2)}, as $R\to\infty$, the renormalized distribution $R^d\scrd$ converges in law to a Gaussian
random distribution on $\R^d\times\R$ whose action on a test function
$\Phi(x,\eta)$ is obtained by first applying a deterministic
dimension-dependent operator $F_d$ to $\Phi$, where $F_d\Phi$ is a
function of two variables $(s,z)\in\R^d\times\R^d$, and then evaluating
the resulting function against a white noise on $\R^{2d}$ with
intensity $\rho(z)^2\,\dd s\,\dd z$.
\end{itemize}

\end{theorem}

\begin{remark}\label{r:meaculpa}{\rm
\begin{itemize}
\item[(i)]
{\rm
In the case {\bf (A1)}, the dimension dependence of the limiting distribution is
entirely encoded by the operator $E_d$. In the framework of the present
section (see Remark \ref{r:ed} for a more precise statement), $E_d$ acts on test functions
$\Phi\in\Sch(\R^d\times\R)$ according to
\[
(E_d\Phi)(x)
\approx
\begin{cases}
\displaystyle
\int_{-1}^{1}
\Phi(x,\eta)\,(1-\eta^2)^{\frac{d-4}{2}}\,\dd\eta,
& d\ge 3, \\[3mm]
\Phi(x,1)+\Phi(x,-1),
& d=2.
\end{cases}
\]
Thus, in dimension $d=2$, the limiting frequency distribution is
concentrated on the two monochromatic frequencies $\eta=\pm1$,
whereas in dimensions $d\ge3$ it is spread over the interval
$[-1,1]$ according to the density
$(1-\eta^2)^{\frac{d-4}{2}}$. As already discussed in Remark \ref{r:overview}-(i), no such
dimensional dependence appears in observable-based results such as
Theorem \ref{t:squarecorr}. We refer the reader to Remark \ref{r:esempiodelsecolo} and \cref{ss:interpretationofultralimit} for further interpretations and discussions of this result.
}
\item[(ii)] We will see in Section \ref{ss:ricovero} that Theorem \ref{t:ultralimitintro} (as well as the more general Theorem \ref{thm:ultralimit}) allows one to recover Theorem \ref{t:squarecorr}\footnote{The case $T_R=0$ is a direct corollary, see \cref{ss:ricovero}, and the same arguments used in the proof of \cref{t:squarecorr} allows one to deal with a general $T_R$.}, as well as a variation of Theorem \ref{t:integrablefields} focusing on quadratic functionals of smooth Gaussian fields with an absolutely continuous density.
\end{itemize}
}
\end{remark}

\subsubsection{Local $FGF_{\frac12}$ scaling limits for quadratic functionals of monochromatic random waves on manifolds.}\label{ss:pbintro} The aim of this section is to show that, by combining the results of Section \ref{ss:fixedobservablesintro} with the coupling techniques initiated in \cite{DNPR23} and the estimates from \cite{Kee19}, one can establish ${\rm FGF}_{1/2}$-type fluctuations for quadratic functionals of {\bf pullback monochromatic random waves} on Riemannian manifolds without conjugate points. We choose to focus exclusively on quadratic functionals ({involving derivatives up to an arbitrary order $k\geq 1$}), both to keep the notational complexity of our forthcoming Theorem \ref{t:FGFpullback} under control and in view of the central role these functionals play in Section \ref{ss:fixedobservablesintro}. Our approach is robust, and can be in principle used to infer ${\rm FGF}_{1/2}$ limits for more general geometric functionals, e.g. through the use of {\bf mixed Kac-Rice formulae} similar to those applied in \cite[Section 3]{DNPR23}. 

\smallskip 

For $d\geq 2$, we denote by $(\mathcal M, g)$ be a $d$-dimensional smooth compact Riemannian manifold. We write $\Delta_g$ to indicate the {\bf Laplace-Beltrami operator} on $\mathcal M$, and use the symbol $\lbrace f_j : j\in \N \rbrace$ to indicate an orthonormal basis of $L^2(\mathcal M)$, composed of real eigenfunctions of $\Delta_g$ such that
$$
\Delta_g f_j + \lambda^2_j f_j=0,
$$
where (without loss of generality) the eigenvalues are ordered in such a way that $0=\lambda_0<\lambda_1\le \lambda_2 \le \dots \uparrow \infty$. Following e.g. \cite{CH20, canzani_MRWsurvey, Zel09}, the {\bf monochromatic random wave} on $\mathcal M$ of parameter $\lambda$ is defined as the real-valued random field
\begin{equation}\label{phi}
\Psi_\lambda(x) := \frac{1}{\sqrt{\text{dim}(H_{\lambda})}} \sum_{\lambda_j\in [\lambda, \lambda + 1]} a_j f_j(x), \quad x\in \mathcal{M},
\end{equation}
where $\{a_j\}$ is a collection of i.i.d. standard Gaussian random variables, and
$$
H_{\lambda} := \bigoplus_{\lambda_j \in [\lambda, \lambda + 1]} \text{Ker}(\Delta_g + \lambda_j^2\, \text{Id}),
$$
with $\text{Id}$ the identity operator. By construction, the field $\Psi_\lambda$ is centered and Gaussian, and its covariance kernel is 
\begin{equation}\label{covRiem}
K_{\lambda}(x,y) := \Cov(\Psi_\lambda(x), \Psi_\lambda(y)) = \frac{1}{\text{dim}(H_{\lambda})} \sum_{\lambda_j\in [\lambda, \lambda +1]} f_j(x) f_j(y),\quad x,y\in \mathcal M.
\end{equation}
Monochromatic ``short window random waves'' such as $\Psi_\lambda$ were first introduced by Zelditch in \cite{Zel09} as approximate models of random Laplace eigenfunctions on manifolds that do not necessarily possess spectral multiplicities; see e.g. \cite{CH20, PV, DNPR23, SarnakWigman2019, Canzani_Sarnak_2019, canzani_hanin_2015_immersions, canzani_MRWsurvey, canzani_hanin_2018_cinfscasymp} for further references and details, as well as \cite[Section 2.2]{UniformRW} for a recent overview of this area. 

\smallskip

Following \cite{CH20}, we now fix $x\in \mathcal M$, and consider the tangent space $T_x\mathcal M \simeq \R^d$ to $\mathcal M$ at $x$. We define the {\bf pullback random wave} associated with $\Psi_\lambda$ at the point $x$ to be the Gaussian random field on $T_x\mathcal M$ given by 
$$
\Psi_\lambda^x(u) := \Psi_\lambda\left ( \exp_x \left ( \frac{u}{\lambda} \right )\right ),\qquad u\in T_x \mathcal M, 
$$
where $\exp_x : T_x\mathcal M \to \mathcal M$ is the usual {\bf exponential map} at $x$. The centered Gaussian (Euclidean) field $\Psi_\lambda^x$ is smooth with probability one, and its covariance kernel is trivially given by
$$
K_{\lambda}^x(u,v) = K_{\lambda}\left(\exp_x \left ( \frac{u}{\lambda}  \right) , \exp_x \left ( \frac{v}{\lambda}\right ) \right ),\qquad u,v\in T_x \mathcal M.  
$$ 

\begin{definition}[See Definition 1 in \cite{CH20}]{\rm The point $x \in \mathcal{M}$ is said to be of {\bf isotropic scaling} if, for every positive function $\lambda \mapsto r(\lambda)$ such that $r(\lambda) = o(\lambda)$, one has that
\begin{equation}\label{limit}
\sup_{u,v \in {B}(0, r(\lambda))  } \left| \partial^a\partial^b\left[  K_{\lambda}^x(u,v) - (2\pi)^{d/2} \frac{J_\nu(\| u-v\|_{g_x})}{\|u-v\|^{\nu}_{g_x}}\right]\right| \to 0,\quad \lambda\to \infty,
\end{equation}
where $\nu =(d-2)/2$ as before, $a,b\in \mathbb{N}^d$ are multi-indices describing partial derivatives with respect to $u$ and $v$, respectively, $\| \cdot \|_{g_x}$ is the norm on $T_x\mathcal M$ induced by $g$, and ${B}(0, r(\lambda))$ is the ball of radius $r(\lambda)$ centered at the origin.
}
\end{definition}

\begin{remark}\label{r:easy}{\rm 
\begin{itemize}

\item[(a)] Note that ${B}(0, r_\lambda)$ corresponds to a shrinking ball of radius $\frac{r_\lambda}{\lambda}$ on $\mathcal M$. 

\item[(b)] Sufficient conditions for a point $x_0$ to be of isotropic scaling are discussed e.g.\ in \cite[Section 2.5]{CH20}, where it is shown that a sufficient condition for $x_0\in \mathcal{M}$ to be of isotropic scaling is that the class
$$
\mathcal{L}_{x_0,x_0} := \{\xi \in S_{x_0} \mathcal{M} : \exists t>0 \mbox{ s.t. } \exp_{x_0}(t\xi) = x_0 \} 
$$
has volume 0 in $T_{x_0} \mathcal{M}$, with $S_{x_0} \mathcal{M}$ denoting the unit sphere in $T_{x_0} \mathcal{M}$ with respect to the norm $\|\cdot \|_{g_{x_0}}$. For every compact smooth manifold $\mathcal{M}$ and for every $x_0\in \mathcal{M}$, the property $| \mathcal{L}_{x_0,x_0} | = 0$ is generic in the space of all Riemannian metrics \cite[Lemma 6.1]{SoggeZelditch}. It is well-known that the condition $| \mathcal{L}_{x_0,x_0}| = 0 $ holds for every $x_0\in \mathcal{M}$ whenever $\mathcal{M}$ has no conjugate points (and, in particular, when $\mathcal{M}$ is negatively curved).

\item[(c)] The asymptotic relation \eqref{limit} implies that the Gaussian field $\{ \partial^a\Psi^{x_0}_\lambda (u) : |a|\geq 0, \,  u\in T_{x_0}\mathcal{M}\}$ converges in the sense of finite-dimensional distributions to 
$$
\left\{({2\pi})^{d/4} \partial^a\Psi^{x_0}_\infty(u) : |a|\geq 0, \,  u\in T_{x_0}\mathcal{M}\right\},
$$   
where $\Psi^{x_0}_\infty$ is the centered Gaussian field on $T_{x_0}\mathcal{M}$ with covariance 
$$
\E[\Psi_\infty^{x_0} (u)\Psi_\infty^{x_0} (v)] = \frac{J_\nu(\| u-v\|_{g_x})}{\|u-v\|^{\nu}_{g_x}}.
$$
It is known that, with probability one, $\Psi_\infty^{x_0}$ is an eigenfunction with eigenvalue 1 of the Laplace operator on $T_{x_0}\mathcal{M}$ associated with the metric $g_{x_0}$. 

\item[(d)] In this section, we are only interested in second-order results holding for a fixed $x_0\in \mathcal{M}$ of isotropic scaling. As a consequence {we will always implicitly choose coordinates around $x_0$ in such a way that $g_{x_0} = {\rm Id}$}, and we will write $\|\cdot \|_{g_{x_0}} = \|\cdot \|$. With this convention, the field $(2\pi)^{d/4}\Psi_\infty^{x_0}$ at Point (c) can be identified with $\{ \sigma(d) \cdot B(x) : x\in \R^d\}$, where $\sigma(d ) := \sqrt{\frac{(2\pi)^{d/2}}{2^\nu \Gamma(\nu,+1)}}$ (with $\nu =\frac{d-2}{2}$), and $B$ is Berry's field, as defined in \eqref{e:helmholtz}--\eqref{e:berrycov}. 

\end{itemize}

}
\end{remark}

The following statement, taken from \cite{Kee19}, allows one to explicitly quantify the speed of convergence to zero of the supremum featured in \eqref{limit}.

\begin{theorem}[Corollary 1.1 in \cite{Kee19}]\label{t:newberard}
Let $(\mathcal M, g)$ be a smooth, compact, Riemannian manifold of dimension $d$ without conjugate points. Then as $\lambda \to +\infty$, for any multi-indices $a,b\in \mathbb N^d$
\begin{equation}\label{e:keeler}
\sup_{u,v \in {B}(0, r(\lambda))  } \left| \partial^a\partial^b\left[  K_{\lambda}^x(u,v) - (2\pi)^{d/2} \frac{J_\nu(\| u-v\|_{g_x})}{\|u-v\|^{\nu}_{g_x}}\right]\right| = O\left(\frac{1}{\log \lambda}\right),
\end{equation}
whenever $r_\lambda = O\left( \sqrt{\frac{\lambda}{\log \lambda} } \right)$. The constant in the $O$-notation depends on the choice of $x_0\in \mathcal M$ and $\{r_\lambda\}$, and on the order of differentiation. 
\end{theorem}

\begin{remark}{\rm According to Remark \ref{r:easy}-(a), the ball ${B}(0, r_\lambda)$ appearing in Theorem \ref{t:newberard} corresponds to a shrinking ball of radius $ O\left ( \frac{1}{\sqrt{\lambda \log \lambda}}  \right )$ on $\mathcal M$. }
\end{remark}

\smallskip 

Keeler's estimate \eqref{e:keeler} was combined in \cite[Theorem 1.6]{DNPR23} with coupling techniques for Gaussian random fields to show that, in dimension $d=2$, if the sequence $\{r_\lambda\}$ grows at a rate significantly slower than $\lambda$ (more precisely, as $r_\lambda = o(\log \lambda)^{1/25}$), then the {\bf nodal length measure} associated with $\Psi^{x_0}_\lambda$ at a point of isotropic scaling satisfies a large-domain CLT, with the same normalization and centering as in the case of the nodal length of planar Berry's random wave model (see, e.g., \cite{NourdinPeccatiRossi2019, Berry2002a}). Combining \cite[Section 4.1]{PV} and \cite[Proposition 1.3]{NPV}, one also infers that this convergence can be interpreted as convergence to white noise in the sense of random distributions.

\smallskip 

 The next statement, which constitutes the main result of this section, shows that in any dimension $d \geq 2$, and under growth conditions on $r_\lambda$ comparable to those above, one can use the results established in Section \ref{ss:fixedobservablesintro} to derive large-domain ${\rm FGF}_{1/2}$ limits for quadratic functionals of pullback waves. The proof, which also relies on the Gaussian coupling techniques introduced in \cite{DNPR23}, is deferred to Section \ref{ss:proofPullback}.  

\begin{theorem}[${\rm FGF}_{1/2}$ limits for pullback random waves]\label{t:FGFpullback} For $d\geq 2$, let the above notation and framework prevail, and let $x_0\in \mathcal M$ be a point of isotropic scaling. For $k\geq 1$, let $J = J(d; k)$ denote the cardinality of the set $\mathcal{J} = \mathcal{J}(d; k)$ of all multi-indices $a$ such that $|a|\leq k$. Let 
$$
Q : \R^{J}\to \R  
$$
be a polynomial of degree $\leq 2$ such that the Berry's random wave observable
$$
\varphi\mapsto \int_{\R^d}Q\big(\partial^aB(x) : a\in \mathcal{J} \big)\, \varphi(x)\, dx,
$$
is in the class ${\bf UC}\left\{{\rm FGF}_{1/2}(\R^d)\right\}$ for some $\alpha\in\R$ and $\beta^2>0$ (see Definition \ref{d:UC}). Then, if
\begin{equation}\label{e:rlento}
{r_\lambda\to \infty \quad \mbox{and}\quad r_\lambda^{12d+4k+4}(\log r_\lambda)^2 = o\big(\log \lambda\big)}, \quad \lambda\to\infty,
\end{equation}
one has that, for every $\varphi\in C_c^\infty$,
$$
\frac{\int_{\R^d}Q\big(\partial^a\tilde{\Psi}_\lambda^{x_0}(x) : a\in \mathcal{J}\big)\, \varphi_{r_\lambda}(x)\, dx - \alpha\, r_\lambda^d \int_{\R^d}\varphi(x) \, dx}{\beta\, r_\lambda^{(d+1)/2}} \stackrel{\rm law}{\longrightarrow} h(\varphi),
$$
where $h\sim {\rm FGF}_{1/2}(\R^d)$, $\tilde{\Psi}_\lambda^{x_0}:= \frac{{\Psi}_\lambda^{x_0}}{\sigma(d)}$, and the constant $\sigma(d)$ has been defined in Remark \ref{r:easy}-(d).

\end{theorem}

\begin{remark}{\rm
The growth condition \eqref{e:rlento} is not expected to be optimal. In particular, as in \cite{DNPR23}, the exponents are likely artefacts of the use of Sobolev embedding techniques, whereas the logarithmic dependence originates from the sharp estimate of Keeler \eqref{e:keeler} and cannot be dispensed with.}
\end{remark}

\begin{example}
By virtue of Theorem \ref{t:squarecorr}, for generic $d\geq 2$ and $k\geq 1$, in Theorem \ref{t:FGFpullback} one can choose
$$
Q\big(\partial^a\tilde{\Psi}_\lambda^{x_0}(x) : a\in \mathcal{J}(d;k)\big) = \sum_{j=0}^k c_j \| \nabla^j \tilde{\Psi}_\lambda^{x_0}(x)\|^2,
$$
where the coefficients $c_j$ are such that $\sum_{j=0}^k c_j\neq 0$.
\end{example}

\section{Further preliminary notions}\label{s:preliminaries}
\subsection{Tempered distributions}\label{sec:temperando}
\subsubsection{Schwartz spaces and distributions}\label{sec:tempeschwa}
We adopt standard conventions in distributional language. In particular, as in Section \ref{ss:convintro} above and \cite[Sec. 2.1]{FGFSurvey}, we denote by $\Sch(\R^d)$ the real Schwartz space and by $\Sch'(\R^d)$ the space of real tempered distributions. In the following, we will need to extend such notation to product spaces $M\times \R^d$ with $M$ being a compact smooth Riemannian manifold. So, we define $\Sch(M\times \R^d)$ as the space of real-valued smooth functions on $M\times \R^d$ whose derivatives of
all orders exist and decay faster than any polynomial at infinity. Then, the space of real tempered distributions on $\Sch'(M\times \R^d)$ is defined as the dual of $\Sch(M\times \R^d)$ with respect to its natural topology. Instead of spelling out the intrinsic definition, we shall assume that $M\subset \R^k$ is isometrically embedded. Then, $\Sch(M\times \R^d)$ is canonically isomorphic to the space of  all $\f|_{M\times \R^d}$, where $\varphi\in \Sch(\R^{k}\times \R^d)$, endowed with the quotient topology; while $\Sch'(M\times \R^d)$ is the subset of tempered distributions $T\in \Sch'(\R^{k}\times \R^d)$ such that $\f|_{M\times \R^d}=\psi|_{M\times \R^d}\implies T(\f)=T(\psi)$. More generally, if $E$ is a vector bundle over $M$, as above, we define the spaces $\Sch(E)$ and $\Sch'(E)$ in the analogous way, by embedding of $E$ as a subbundle in a trivial vector bundle: $E\subset M\times \R^d$.
We denote as $\Sch_\C(E), \Sch'_\C(E)$ the corresponding complexifications. 
\subsubsection{Symbolic differential notation} We denote the canonical pairing of $f\in \Sch'_\C(M\times \R^d)$ and $\f\in \Sch_\C(M\times \R^d)$, as 
\be 
\langle f,\f\rangle =: \int_{M\times \R^d} f(\dd x) \f(x),
\ee
In particular, formally, $f(\dd x)=f(x)\dd x$ when $f$ is the distribution associated with a measurable function $f\colon M\times \R^d\to \C$. Note that the pairing coincides with the $L^2$ complex linear product with respect to the Riemannian density, for which we keep the same notation.

\subsubsection{Composition and adjoint}
Let $E\colon \Sch(\R^k)\to \Sch(\R^d)$ be a linear and continuous operator, then, for any $W\in \Sch'(\R^d)$, one can define $W\circ E\in \Sch'(\R^k)$, such that $\langle W\circ E,\f \rangle= \langle W,E(\f) \rangle$, for every $\f\in \Sch(\R^k)$. The definition extends naturally to $\TAG_d$ and $\AG_d$, as defined in Section \ref{ss:livingcoloursintro} and, in more detail, in \ref{s:linespace}. The \embf{adjoint operator} of $E$ is the operator $E^*\colon \Sch' (\R^d)\to \Sch'(\R^k)$ defined as $E^*W=W\circ E$.

\subsubsection{Changes of variables} \label{sec:changofvar}
We say that $Y\colon \R^d\to \R^d$ is a {\bf Schwartz diffeomorphism} if it is smooth and all of its derivatives decay faster than any polynomial at infinity. This property ensures that $|\det Y|^{-1},|\det Y|\in \Sch(\R^d)$, so that, for any test function $\f\in \Sch(\R^d)$, the function $|\det Y|^{-1}\f\circ Y^{-1}$ is still in $\Sch(\R^d)$. Given $f\in \Sch'(\R^d)$, we define $f\circ Y\in \Sch'(\R^d)$ as 
\be 
\int_{\R^d} f(Y(\dd x))\f(x):=\langle f\circ Y, \f \rangle := \langle f, |\det Y|^{-1}\f\circ Y^{-1}\rangle\!\!=\!\!\int_{\R^d}f(\dd y)|\det Y|^{-1}\f(Y^{-1}(y)).
\ee
We also define the symbol $f(\dd Y(x)):=f(Y(\dd x))|\det Y|$, as the differential $g(\dd x)$ associated with the distribution $g=(f\circ Y)|\det Y|$. As a consequence,
\be 
\int_{\R^d} f(\dd Y(x))\f(x):=\langle f, \f\circ Y^{-1} \rangle=\int_{\R^d} f(\dd y)\f(Y^{-1}(y)) .
\ee
One can easily check that this definition is coherent with standard changes of variables in the integral, when $f(\dd x)=f(x)\dd x$ is also a smooth function. For instance, given $R>0$, the scaling $f(R\cdot)$ of a distribution $f\in \Sch'(\R^d)$ reads as
\be 
\langle f(R\cdot),\f\rangle=\int   f(R\dd x)\f(x)=\int  f(\dd (Rx))R^{-d} \f(x)=\int f(\dd x)\f(R^{-1}x)R^{-d},
\ee
for any test function $\f\in \Sch(\R^d)$.

\subsubsection{Delta function}\label{sec:delta_formalism}
The delta distribution $\delta\in \Sch'(\R^d)$ admits more general changes of variables than just the Schwartz diffeomorphisms. Let $E\colon \R^d\to \R^k$ be smooth and such that $0$ is a regular value. Then, $\delta\circ E\in \Sch'(\R^d)$ is defined as 
\be 
\int_{\R^d}\delta(E(\dd x))\varphi(x)=\langle \delta\circ E, \varphi\rangle :=  \int_{E^{-1}(0)} \frac{\varphi(x)}{\Jac E(x)} \dd x.
\ee
This is a well-defined tempered distribution whenever all components of $E$ are of Schwartz class. When working with the distribution $\delta\circ E$, it is more convenient to write the corresponding symbol as $\delta(E(x))\dd x:= \delta(E(\dd x))$, and manipulate directly the \embf{delta function} symbol $\delta(E(x))$. When there is a natural splitting $\R^{d+d'}=\R^d \times \R^{d'}$ involved, the notation $\delta(E(x,y))\dd x\dd y$ does not lead to ambiguities, since the following Fubini formula holds:
\be\label{eq:magicJacob} 
\int_{E^{-1}(0)} \f(x,y) \frac{\Haus^{d'-k}(\dd (x,y))}{\Jac\tyu E\uyt(x,y)}= \int_{\R^d}\dd x\int_{E(x,\cdot)^{-1}(0)} \f(x,y) \frac{\Haus^{d-k}(\dd (x,y))}{\Jac\tyu E(x,\cdot)\uyt(y)};
\ee
this is due to the coarea formula \cite[Theorem 3.2.12]{federer2014}, together with the fact that $$\Jac(A|_{\ker B})\Jac(B)=\Jac(A,B)=\Jac(B|_{\ker(A)})\Jac(A),$$ for any two linear maps $A,B$ with the same domain. In particular, for $f\in \Sch'(\R^d)$ and $E$ as above, we have that $\int_{\R^d} f(\dd x)\delta(E(x,y))$ is a well defined distribution in the variable $y$, provided that $0$ is a regular value for $E$.\footnote{Then, by Sard's theorem, $0$ is a regular value for $E(x,\cdot)$ for almost every $x$.}
In this context, it is useful to remember the following formula. Consider the case $E(x,y)=x-h(y)$, then one can easily check that $\Fouvar{x}{\xi} \tyu \dd x\delta(x-h(y))\uyt \dd y=\eim{\xi,h(y)}\dd \xi \dd y$. Therefore, by the inversion formula
\be\label{eq:exp_fou_delta}
\int_{\R^d}\dd \xi \ei{\xi,x-h(y)}:=\Fouvar{\xi}{x}^{-1}\tyu \eim{\xi,h(y)}\uyt =\delta(x-h(y)),
\ee
as distributions in the variables $(x,y)\in \R^{2d}$; where the left-most integral has to be interpreted as defined by the above identity.

The multiplication of two deltas $\delta(E_1(x))\delta(E_2(x))\dd x:=\delta(E_1(x),E_2(x))\dd x$ is allowed if and only if all three are defined, i.e., if they all satisfy the regular value condition. Again, such a notation does not lead to ambiguities, thanks to the coarea formula, combined with \cref{eq:magicJacob}:
\be \label{eq:delta_product}
\int_{E_1^{-1}(0)\cap E_2^{-1}(0)} \f \frac{\dd \Haus^{d-k_1-k_2}}{\Jac\tyu E_1,E_2\uyt}= \int_{E_1^{-1}(0)\cap E_2^{-1}(0)} \f \frac{\dd \Haus^{d-k_1-k_2}}{\Jac\tyu E_2|_{E_1=0}\uyt\Jac(E_1)}.
\ee
In such context it is often useful to observe that $\delta(E_1)\delta(E_2)=\delta(E_1)\delta(E_2+\a E_1)$, $\forall\a\in \R$. 

Note that, within this framework, the square $\delta(E(x))^2$ is not well defined.
\subsection{Notations}\label{sec:convprelim}
For the rest of the paper, we will adopt the following conventions, analogous to those adopted in \cite[1.3]{UniformRW}.
\subsubsection{Random elements}\label{ss:randin} 
Given random elements, $A,B$, we write $A\sim B$ to indicate that $A,B$ have the same distribution. Given a topological space $\mathscr{X}$, we will write 
\be \label{e:randin}
Z\, \randin \, \mathscr{X}
\ee
to denote the fact that $Z$ is a measurable random element with values in $\mathscr{X}$. 

\subsubsection{Notation for integrals, volumes and spheres}\label{subsec:notaintegrals}
Given a Riemannian manifold $(M,g)$ of dimension $d$ and $k\in \N$, we
denote the associated $k$-dimensional Hausdorff measure by $\Haus^{k}(\cdot)$. Moreover, given a Borel function $F\colon M\to \R$, we will use the following shorthand notations: 
\begin{eqnarray}
\int_MF(x)\dd x:=\int_MF(x)\Haus^{\dim (M)}(\dd x), 
\end{eqnarray}

In particular, if $S\subset M$ is a submanifold of dimension $k$, then $\int_SF(x)\dd x$ stands for the integral with respect to the $k$-dimensional Hausdorff measure associated with $g$ (or, equivalently, to $g|_S$), unless otherwise specified. 
{Note that, although such notation is not standard, it will be adopted throughout the paper to make our formulas more readable while causing only minimal ambiguities.

\subsubsection{Spheres}
Many formulas in the paper concern the case in which $M=S(V)$ is the unit sphere in a metric vector space $V$, that is, $S(V):=\{x\in V : \|x\|=1\}$ (see, e.g., \cref{lem:hpsi_cont}). In particular, $S^d=S(\R^d)$ denotes the standard round sphere. Finally, we define the constants
\be\label{eq:spherevolume}
s_{d}:=\vol{}(S^d)=\frac{2\pi^{\frac{d+1}{2}}}{\Gamma\tyu \frac{d+1}{2}\uyt},
\ee
for any integer $d$.
\subsection{Random tempered distributions}
{In this subsection, we expand the discussion developed in \cref{ss:convintro}, including additional technical details.}
Let $(\Omega,\mathscr{E},\PP)$ be our reference probability space. A \embf{random tempered distribution} on a space $E$ (where $E=\R^d$ or $E=M\times\R^d$ as above) is a measurable map $\Xi\colon \Omega\to \Sch'_\C(E)$, where $\Sch'_\C(E)$ is endowed with the Borel $\sigma$-algebra $\mathscr{B}(\Sch'_\C(E))$ generated by the weak-$*$ topology; we write $\Xi\randin \Sch'_\C(E)$, leaving $\Omega$ as implicit. 
These objects are also called \embf{Generalized random fields} in \cite{GelfandVilenkin} and \cite{Bierme_CSA}.  
{Note that the topological properties of $\Sch'(\R^d)$, for arbitrary $d\in \N$, easily translates to their analogous version for $\Sch'_\C(M\times \R^d)$. The extension to arbitrary $M$ is done by choosing an arbitrary embedding $M\subset \R^k$, as discussed in \cref{sec:tempeschwa}. The extension to the complexified space is straightforward: a complex random tempered distribution $\Xi\randin \Sch'_\C(E)$ is equivalent to a pair $(\Re \Xi, \Im \Xi)\randin \Sch(E)\times \Sch(E)\cong \Sch(E\sqcup E)$. 
In particular, as proved in \cite[Proposition 3.7]{Bierme_CSA}, it should be noted that $\mathscr{B}(\Sch'(E))$ coincides with the $\sigma$-algebra generated by a stronger topology: the one generated by the seminorms $\Sch'(E)\ni T\mapsto \sup_{\f \in B}|\langle T,\f \rangle|$, with $B$ ranging over bounded subsets of $\Sch(E)$. Therefore, the two corresponding notions of convergence in law coincide.}

The \embf{characteristic functional} of $\Xi$ is the map $\chi\colon \Sch_\C(E)\to\C$ defined by
\be 
\chi(\varphi):=\E\kop e^{i\Re \langle \Xi,\varphi\rangle}\pok, \quad \varphi\in\Sch_\C(E).
\ee
Note that $\chi$ is defined on the complexified Schwartz space $\Sch_\C(E)$: letting $\varphi$ vary over all complex test functions is necessary to recover the full joint law of $(\Re\Xi,\Im\Xi)$, since for $\varphi=\varphi_1+i\varphi_2\in\Sch_\C(E)$ with $\varphi_1,\varphi_2\in\Sch(E)$, the real part $\Re\langle\Xi,\varphi\rangle$ depends on both components of $\Xi$.

The following classical result provides the existence and uniqueness of random tempered distributions through their characteristic functional.

\begin{theorem}[Bochner--Minlos]\label{thm:bochnerminlos}
A functional $\chi\colon \Sch_\C(E)\to \C$ is the characteristic functional of a random tempered distribution $\Xi\randin\Sch'_\C(E)$ if and only if $\chi(0)=1$, $\chi$ is positive definite, and $\chi$ is continuous in the topology of $\Sch_\C(E)$. Moreover, the law of $\Xi$ on $\Sch'_\C(E)$ is uniquely determined by $\chi$.
\end{theorem}

\begin{proof}
See \cite[Theorem~2.3]{FGFSurvey} and \cite[Theorem~2.1]{Bierme_CSA}. The complex case follows from the real one via the identification $\Sch'_\C(E)\cong \Sch'(E\sqcup E)$.
\end{proof}

We say that $\Xi$ is a \embf{centered Gaussian} random tempered distribution if $\langle\Xi,\f\rangle$ is a centered Gaussian random variable for every $\f\in\Sch_\C(E)$. In this case, the characteristic functional takes the form $\chi(\f)=\exp\!\Big({-\tfrac{1}{4}C_1(\f,\f)-\tfrac{1}{4}\Re\, C_2(\f,\f)}\Big)$, where
\be 
C(\varphi,\psi):=\tyu 
\E\kop \langle\Xi,\f\rangle\overline{\langle\Xi,\psi\rangle}\pok, \E\kop \langle\Xi,\f\rangle \langle\Xi,\psi\rangle\pok
\uyt, \quad \varphi,\psi\in\Sch(E),
\ee
is a pair $C=(C_1,C_2)$ of two forms on $\Sch_\C(E)$, called the \embf{covariance} of $\Xi$. 

\begin{remark}\label{rem:bochmil_Gaussian}
By \cref{thm:bochnerminlos}, the law of a centered Gaussian random tempered distribution is uniquely determined by its covariance pair $C=(C_1,C_2)$, where $C_1$ ranges over all continuous non-negative definite sesquilinear Hermitian forms on $\Sch_\C(E)$ and $C_2$ over all continuous symmetric bilinear forms on $\Sch_\C(E)$ satisfying the compatibility condition $|C_2(\f,\psi)|^2\le C_1(\f,\f)\,C_1(\psi,\psi)$ for all $\f,\psi\in\Sch_\C(E)$.
\end{remark}

The following convergence criterion for $\Sch'_\C(E)$-valued random variables is the natural counterpart of L\'evy's classical continuity theorem.

\begin{theorem}[Fernique-L\'evy continuity]\label{thm:levy_Sprime}
Let $\{\Xi_n\}_{n\ge 1}$ be random tempered distributions on $E$. Then $\Xi_n\xrightarrow{\;\rm law\;}\Xi$ in $\Sch'_\C(E)$ if and only if the characteristic functionals converge pointwise:
\be
\chi_n(\f):=\E\kop e^{i\Re\langle\Xi_n,\f\rangle}\pok\;\xrightarrow{n\to\infty}\;\chi(\f), \qquad \forall\,\f\in\Sch_\C(E),
\ee
and $\chi$ is continuous at $0$ in the topology of $\Sch_\C(E)$.
\end{theorem}
\begin{proof}
See \cite{Fernique1968} and 
\cite[Corollary~2.4]{Bierme_CSA}.
The complex case follows from the real one via $\Sch'_\C(E)\cong\Sch'(E\sqcup E)$. 
\end{proof}

\begin{remark}\label{rmk:gaussian_fdd_suffice}
For centered Gaussian random tempered distributions, the characteristic functional $\chi(\f)$ is always continuous at $0$ whenever $C$ is on $\Sch_\C(E)^2$. In particular, by \cref{thm:levy_Sprime}: if $\Xi_n\randin\Sch'_\C(E)$ are centered Gaussian with covariance $C_n$, and $C_n(\f,\f)\to C(\f,\f)$ for every $\f\in\Sch_\C(E)$ with $C$ continuous on $\Sch_\C(E)$, then $\Xi_n$ converges in law to $\Xi$ in $\Sch'_\C(E)$, where $\Xi$ is the centered Gaussian with covariance $C$ (whose existence is guaranteed by \cref{thm:bochnerminlos}).
\end{remark}

\subsection{White Noise and Fractional Gaussian Fields}\label{ss:fgfs}
In this section we will introduce the necessary definitions and notations to work with white noise and fractional Gaussian fields. We largely follow the surveys \cite{FGFSurvey,Sheffield2003}.
\subsubsection{Fourier transform convention}\label{fourier}
Let $\langle x,\xi\rangle$ denote the Euclidean scalar product of two vectors $x,\xi\in \R^d$. We define the Fourier transform of a function $\varphi\in \Sch_\C(\R^d)$, as the function 
\be 
\hat{\varphi}(\xi):=\Fouvar{x}{\xi}\tyu \varphi(x)\uyt:=\int_{\R^d} \varphi(x) \eim{x, \xi} \dd x.
\ee
We recall that $\varphi\mapsto \hat \varphi$ extends to an isometry of $L^2(\R^d;\C)$ to itself, such that $
\Fouvar{\xi}{x}^{-1}=(2\pi)^{-d}\Fouvar{\xi}{(-x)}.
$
Such a definition extends from Schwartz functions to tempered distributions, by the adjunction relation:
$ 
\langle \hat{f}, \varphi \rangle =  \langle f, \hat\varphi \rangle.
$
We refer to \cite[Section V.3]{sigurdurHelgason} for a list of the standard properties of the Fourier transform.

Note that if $v\colon \R^d\to \R^d$ is a linear isomorphism, then
\be \label{eq:fouchangofvar}
\Fouvar{x}{v^{-*}(\xi)}\tyu \varphi(x)\uyt=\Fouvar{x}{\xi}\tyu \varphi(v(x))\uyt |\det v|.
\ee
Such a property implies that the Fourier transform is invariant under orthogonal ($v=v^{-*}$ and $|\det v|=1$) changes of coordinates.
Given a metric finite-dimensional vector space $V$, and an isometry $v\colon \R^d\to V$. One defines through $v$ the spaces $\Sch(V),\Sch'(V)$ and the Fourier transform 
\be 
\Fouvar{v}{v(\xi)}^V\tyu \varphi(v) \uyt:=\Fouvar{x}{\xi}\tyu \varphi(v(x)) \uyt, \qquad \Fouvar{v}{v(\xi)}^{-V}\tyu \varphi(v) \uyt:=\Fouvar{x}{\xi}^{-1}\tyu \varphi(v(x)) \uyt.
\ee
The definition so obtained does not depend on the chosen isometry $v$, by \cref{eq:fouchangofvar}.

\subsubsection{The fractional Laplacian}\label{sec:laplac} Let $\Delta=\sum_i \de_i^2$ be the usual (negative definite) Laplacian operator. 
The fractional Laplacian $(-\Delta)^{\frac{s}{2}}$ is the operator acting on $\f\in \Sch_\C(\R^d)$ by:
\be 
(-\Delta)^{\frac{s}{2}} \f (x)= \Fouvar{\xi}{x}^{-1}\tyu |\xi|^s \hat{\f}(\xi)\uyt.
\ee
This is a well defined function for any $\f\in \Sch(\R^d)$, whenever $s>-d$, while for arbitrary $s$ the domain must be specified, see \cite[Section 2.1]{FGFSurvey}. We refer to \cite[Section 2.2]{FGFSurvey} for the extension of $(-\Delta)^s$ to distributions, by self-adjointness.
\begin{remark}
 The operator $(-\Delta)^{\frac s2}$ does not depend on the choice of Fourier transform convention; its Fourier multiplier $|\xi|^s$ does.
\end{remark}

\subsubsection{White noise and FGF definition}\label{sec:wn}

Given $M\times \R^d$ as above, {and a Borel measure $\omega$ on it,} we call \embf{white noise on $M\times \R^d$ {with respect to $\omega$}} any Gaussian random tempered distribution $\Wnoise\randin \Sch'(M\times \R^d) $ such that 
\be \label{eq:wnoisdef}
\E\kop \langle \Wnoise, \varphi\rangle \langle \Wnoise, \psi\rangle \pok= 
\int \f \psi \, \dd \omega, \quad \forall \varphi,\psi\in \Sch(M\times \R^d).
\ee
{Notably, in the following, we will always work in situations in which $\Sch(M\times \R^d)\subset L^{2}(\omega)$. When $\omega$ is the Riemannian density or another canonical measure, we will just say \embf{white noise on $M\times \R^d$}.}
Let $\Wnoise$ be a white noise on $\R^d$. Then, for any $s\in \R$, we call \embf{a Fractional Gaussian Field of order $s$} any random distribution $h\randin \Sch'(\R^d)$ that satisfies the distributional identity, almost surely: 
\be 
h =(-\Delta)^{-\frac{s}{2}}\mathbb{W}
\ee
In such a case, we write $h\sim \FGF_s$.
\begin{remark}
In this paper we will be concerned only with the case $s < \frac{d}{2}$. In this regime, the operator $(-\Delta)^{-s/2}$ is well-defined and injective on tempered distributions, and therefore $\FGF_s$ is a well-defined random element of $\Sch'(\R^d)$. 
In particular, the law of $\FGF_s$ is translation invariant: {with obvious notation}, for any $y \in \R^d$, the translated field $x \mapsto \FGF_s(x+y)$ has the same law as $\FGF_s$.
\end{remark}
\begin{remark}
In the case $s\ge \frac{d}{2}$, the field $\FGF_s$ is only defined modulo polynomials of degree at most $\lfloor H \rfloor$, where $H=s-d/2$ (see \cite[Section 1.1]{CaoSheffield}). 
Accordingly, the pairing $\langle \FGF_s,\varphi\rangle$ is intrinsically defined only for test functions $\varphi \in \mathcal{S}(\R^d)$ satisfying
\[
\int_{\R^d} x^\alpha \varphi(x)\,\dd x = 0 \quad \text{for all } |\alpha|\le \lfloor H \rfloor.
\]
In the literature (see \cite{FGFSurvey}), $\FGF_s$ is therefore defined as a random element of a suitable quotient $\FGF_s\randin \Sch'_H(\R^d)=\Sch'(\R^d)/\mathcal{T}_H(\R^d)$, in a way that is translation invariant in law (see \cite[Section 2,3]{FGFSurvey}).

In such case, all identities involving $\FGF_s$, including the translation invariance, are understood modulo polynomials of degree at most $\lfloor H \rfloor$.
\end{remark}

As discussed in \cite[Section 2.1]{FGFSurvey}, due to Gaussianity, the point of view of random distribution is equivalent to other common ones adopted in the literature, including those oan f isonormal process (see \cite{nourdinpeccatibook}) or Gaussian Hilbert spaces (see \cite{JansonBook}). Such a point of view is very convenient for us, because it allows one to apply the Fourier transform techniques in an almost sure manner. 
\subsubsection{Hurst parameter}
The Hurst parameter of the $\FGF_s$ on $\R^d$ is the real number
\be 
\Hurst:=s-\frac{d}{2}.
\ee
We recall that such number characterizes the behavior of the field under rescaling: the following equivalence in law holds:
\be 
\mathrm{FGF}_{s} (R\cdot)\sim R^H\mathrm{FGF}_{s} (\cdot),
\ee
where $\langle h(R\cdot),\f \rangle:=\langle h(\cdot),R^{-d}\f(R^{-1}\cdot) \rangle$. Moreover, the Hurst parameter identifies the fractional Sobolev (a.s.) regularity of the paths of the field (see \cite[Proposition 6.2 and section 8]{FGFSurvey}): 
\be 
\mathrm{FGF}_{s}\in \Hdot^{H-\e}(\R^d), \quad \text{a.s.} \iff \e>0,
\ee
where, when $r\le 0$, $\Hdot^{r}(\R^d)\subset \Sch
'(\R^d)$ denotes the space of distributions defined as (see \cite[Section 2.1]{FGFSurvey})
\be
\dot{H}^r(\R^d):=((-\Delta)^{\frac{r}{2}})^{-1}L^2(\R^d)=\left\{ f \in \mathcal{S}' :
|\xi|^{r}\hat{f}(\xi) \in L^2(\R^d) \right\}.
\ee
This space is endowed with the natural scalar product $\langle\cdot,\cdot\rangle_{\dot{H}^r}$ on $\dot{H}^r(X)$, that makes $(-\Delta)^{\frac{r}{2}}:H^r\to L^2$ an isometry of Hilbert spaces.\footnote{
In contrast with what happens for standard fractional Sobolev spaces, there is no inclusion relating  $\Hdot^{s_1}$ and $\Hdot^{s_2}$. 
The standard fractional Sobolev spaces are defined, for all $r\in \R$, as
\[
H^{r}(\R^d)
= \left\{ f \in \mathcal{S}' :
(1+|\xi|^2)^{r/2}\hat{f}(\xi) \in L^2 \right\}.
\]
Then, $H^{r}(\R^d)=\dot{H}^{r}(\R^d)$ for all $r\ge 0$. Instead, if $r>0$, then $\dot{H}^{-r}(X)\subsetneq H^{-r}(X)$. 
Precisely, we have that
\bega
H^{-s}(\R^d)
=\{f\in \Hdot^{-s}(X) \colon \de^\a\hat{f}(0)=0 \ \forall |\a|< s\}.
\eega
See \cite[Section 2.1]{FGFSurvey} for details.
}
\subsection{Poisson point process}\label{ss:poissonstuff}
The reader is referred e.g. to \cite{LastPenroseBook, PeccatiReitznerBook, HugSchnieder} for a detailed presentation of the material discussed below.

Consider a measurable space $(\mathcal{Z},\mathscr{Z})$, and write $\nu$ to indicate a $\sigma$-finite measure on it. Denote by $\mathbf{N}_\sigma=\mathbf{N}_\sigma(\mathcal{Z})$ the class of all $\sigma$-finite point measures $\chi$ on $(\mathcal{Z},\mathscr{Z})$ that satisfy $\chi(B)\in\N_0\cup\{+\infty\}$ for all $B\in\mathscr{Z}$. The set $\mathbf{N}_\sigma=\mathbf{N}_\sigma(\mathcal{Z})$ is equipped with the smallest 
$\sigma$-field $\mathscr{N}_\sigma:=\mathscr{N}_\sigma(\mathcal{Z})$ enjoying the property that, for every $B\in\mathscr{Z}$, the mapping $\mathbf{N}_\sigma\ni\chi\mapsto\chi(B)\in[0,+\infty]$ is measurable. A \textbf{Poisson point process} on $(\mathcal{Z},\mathscr{Z})$ with \textbf{control measure} $\nu$, written
\begin{equation*}
\eta=\{\eta(B)\,:\,B\in{ \mathscr{Z}}\},
\end{equation*}
is a random element 
taking values in the measurable space $(\mathbf{N}_\sigma,\mathscr{N}_\sigma)$ and satisfying the following two properties: (i) for any choice $B_1,\dotsc,B_m\in\mathscr{Z}$ of pairwise disjoint sets, the random variables $\eta(B_1),\dotsc,\eta(B_m)$ are stochastically independent, and (ii) for every $B\in\mathscr{Z}$, the random variable $\eta(B)$ has a Poisson law with mean $\nu(B)$,
where we have implicitly extended the family of Poisson distributions to the completed half-line $[0,+\infty]$ in the usual way. The existence of Poisson point processes associated with generic $\sigma$-finite measure spaces $(\mathcal{Z},\mathscr{Z}, \nu)$ is discussed e.g. in \cite[Section 3.2]{LastPenroseBook}.

\begin{remark}\label{r:hompoisson}{\rm For $d\geq 1$, the homogeneous Poisson point process $\eta = \Pi_t$ on $\R^d$ with intensity $t>0$, as introduced in Section \ref{ss:convintro}, corresponds to the following specifications:
\begin{itemize}
\item[--] $Z = \R^d$;
\item[--] $\mathcal{Z} = \mathscr{B}(\R^d)$;
\item[--] $\nu = t\cdot \dd x$, where $\dd x$ stands for the Lebesgue measure.
\end{itemize}
}
\end{remark}

\begin{remark}\label{r:poissonline}{\rm The Poisson line process $\eta = \eta_t\sim {\rm PLP}_t(\R^d)$ ($d\geq 2$) defined in Section \ref{ss:ucintro} corresponds to the following configuration:
\begin{itemize}
\item[--] $Z = \AG_d$, the Grassmannian of affine lines in $\R^d$;
\item[--] $\mathcal{Z}$ is the $\sigma$-field generated by the standard topology on $\AG_d$, as described e.g. in \cite[p. 4]{HugSchnieder};\footnote{Fix an element $\ell$ of the Grassmannian $G(d,1)$ of $\R^d$; the standard topology of $\AG_d$ is the finest topology with respect to which the mapping $\ell^\perp \times SO(d): (x, \theta)\mapsto \theta(x+\ell)$ is continuous. Such a topology is independent of the choice of $\ell$.}
\item[--] $\nu = t\cdot \mu$, where $t>0$ and $\mu$ is the ($\sigma$-finite) invariant measure on $\AG_d$, that has been fixed through relation \eqref{e:fgfcov2}.
\end{itemize}
}
\end{remark}

{
\begin{remark}\label{r:poissdist}
Let $\Pi_t$ be as in Remark \ref{r:hompoisson}. By Campbell's formula
\cite[Proposition 2.7]{LastPenroseBook}, one has that
\[
\mathbb{E}\left[\int_{\R^d} (1+\|x\|)^{-N}\Pi_t(\dd x)\right] <\infty,
\]
for every $N>d$. It follows that
$\Pi_t\in \mathcal{S}'(\R^d)$ almost surely. An analogous argument
shows that, if $\eta_t\sim {\rm PLP}_t(\R^d)$, then
$\eta_t\in \Sch'(\AG_d)$ almost surely, where $\Sch'(\AG_d)$ is
the class of tempered distributions introduced in Definition \ref{def:sch_TAG}.
\end{remark}
}
}

\section{Observables of random fields: proofs of main results.}\label{s:observableproofs}
 The aim of this section is to prove the main results stated in Section \ref{ss:fixedobservablesintro}. 

 \begin{remark}[Convention on test functions]\label{r:convention}{\rm From now on, we fix an integer $d\geq 2$ and use the symbol $\varphi$ to denote either an element of the Schwartz class $\mathcal{S} =\mathcal{S}(\R^d)$, or a finite linear combination of indicators of bounded convex sets having a $C^\infty$ boundary with non-vanishing Gauss curvature (as in the statement of Proposition \ref{p:squarecorrex}).}\end{remark}
 
 Given $\varphi$ as above, we define the function $\varphi_R$ according to \eqref{e:rescaled}; note that, for every $p\geq 1$, one has that $\|\varphi_R\|^p_p = R^d\|\varphi\|^p_p$,  where $\|\cdot \|_p$ stands for the usual norm in $L^p(\R^d)$. We write $\varphi^{-}(x) := \varphi(-x)$. Defining 
 \begin{equation}\label{e:psi}
\psi = \psi_\varphi := \varphi*\varphi^{-},
 \end{equation}
 one has therefore that
\[\varphi_R*\varphi^{-}_R = R^d\psi_R .\]
Note that, for every $z\in \R^d$,
\[\lim_{R\rightarrow +\infty} \psi_R(z) = \varphi*\varphi^-(0) = \|\varphi\|_2^2,\]
and that our conventions allow one to rewrite \eqref{e:lnormINTRO} in convolutional form, namely:
\begin{equation}\label{e:lnorm}
\|\varphi\|_L^2 = \int_{\R^d}\int_{\R^d}\frac{\varphi(x)\varphi(y)}{\|x-y\|^{d-1}}\dd x\dd y = \int_{\R^d}\frac{\psi(x)}{\|x\|^{d-1}}\dd x.\end{equation}

\subsection{Three ancillary lemmas} In this subsection (consistently with \eqref{e:berrycov}), we will denote by $\sigma$ the uniform probability measure on the sphere with unit radius $S^{d-1}$. If $s$ is a smooth complex function on $S^{d-1}$, we let $r_s$ be the Fourier transform of $s\sigma$, that is: $r_s(x) = \int_{S^{d-1}} e^{i\langle x, \theta\rangle } s(\theta) \sigma(\dd\theta)$. We also recall the definition of the constant $K_d$ given in \eqref{e:kappamaiuscolo}.

\smallskip 

Our first result provides an exact characterization of the behaviour of $r_s(x)$ for large values of $\|x\|$.
\begin{lemma}
\label{lemm15a}
 Let the above notation prevail and assume that $s$ is a smooth symmetric function (that is $s(-\theta) = s(\theta)$). Then, as $\|x\|$ goes to $+\infty$,
\begin{equation}\label{e:xinftylouis}
r_s(x) = \frac{K_d}{\|x\|^{\frac{d-1}{2}}}s\left(\frac{x}{\|x\|}\right)\cos\left(\|x\| -\pi\frac{d-1}{4}\right) + O\left(\frac{1}{\|x\|^{\frac{d+1}{2}}}\right).
\end{equation}
If $s$ is antisymmetric (that is $s(-\theta) = - s(\theta)$) then \eqref{e:xinftylouis} continues to hold, with $\cos$ replaced by $i\sin$, with $i = \sqrt{-1}$.
\end{lemma}
\begin{proof}
The proof is a consequence of the stationary phase method (see e.g. \cite{HorBook}) applied to the measure $s \, \dd \sigma$. Let $x\in \R^d$. The sphere $S^{d-1}$ can be considered as the union of the graphs of the two functions $u\mapsto \pm \sqrt{1-\|u\|^2}$, with $u\in x^\perp\cap B(0,1)$. Then,
\begin{align*}
r_s(x) &= \frac{1}{|S^{d-1}|}\int_{x^\perp\cap B(0,1)} e^{i\|x\|\sqrt{1-\|u\|^2}}s\left(\sqrt{1-\|u\|^2}\frac{x}{\|x\|} + u\right)\frac{\dd u}{\sqrt{1-\|u\|^2}}\\
&+ \frac{1}{|S^{d-1}|}\int_{x^\perp\cap B(0,1)} e^{-i\|x\|\sqrt{1-\|u\|^2}}s\left(-\sqrt{1-\|u\|^2}\frac{x}{\|x\|} + u\right)\frac{\dd u}{\sqrt{1-\|u\|^2}}.
\end{align*}
As $\|x\|$ goes to infinity, one can apply the stationary phase method to each of these two terms to estimate the asymptotics of $r_s(x)$. In both cases, the phase is stationary when $u=0$. At this point, the Hessian is $-\Id$ for the first term, and $\Id$ for the second. We deduce that as $\|x\|$ goes to infinity, one the asymptotic relations
\[r_s(x)= \frac{1}{|S^{d-1}|}\left(\frac{2\pi}{\|x\|}\right)^{\frac{d-1}{2}}\left[s\left(\frac{x}{\|x\|}\right)e^{i\|x\|-(d-1)\frac{i\pi}{4}} + s\left(-\frac{x}{\|x\|}\right)e^{-i\|x\|+(d-1)\frac{i\pi}{4}}\right] + O\left(\frac{1}{\|x\|^{\frac{d+1}{2}}}\right)\]
The conclusion follows by gathering the complex exponentials in the case where $s$ is even (resp. odd), and the definition of the constant $K_d$.
\end{proof}

\begin{lemma}
\label{lemm13a} Let the test function $\varphi$ be as in Remark \ref{r:convention}. Then, for some constant $C$ depending on $\varphi$,
\[\lim_{R\rightarrow +\infty} \int_{\R^{d}}\psi_R(x) r_s(x)\dd x = 0,\quad \int_{\R^{d}}\left|\psi_R(x) r_s^2(x)\right|\dd x\leq CR,\quad \int_{\R^{d}}\left|\psi_R(x) r_s^3(x)\right|\dd x\leq C \sqrt{R}.\]
\end{lemma}
\begin{proof}
For the first point, one has that, by Plancherel's identity,
\[\int_{\R^{d}}\psi_R(x) r_s(x)\dd x = \int_{\R^{d}}\widehat{\psi_R(\theta)}\dd\sigma(\theta).\]
The function $\widehat{\psi_R}$ converges to a Dirac mass at zero. In particular, it converges to zero on the sphere, and the conclusion of the first point follows by dominated convergence. The second point follows from the bound
\[|r_s(x)|\leq \frac{K}{|x|^{\frac{d-1}{2}}},\]
given by the previous Lemma \ref{lemm15a}, which in turn implies that 
\[\int_{\R^{d}}\left|\psi_R(x) r_s^2(x)\right|\dd x\leq K^2 R\int_{\R^{d}}\frac{|\psi(x)|}{\|x\|^{d-1}}\dd x \leq CR.\]
The third point follows from the decay of $|r_s(x)|^3$ : if $d\geq 4$ the integral converges, if $d=3$ then there is a factor $\log(R)$ appearing, and if $d=2$ there is a factor $\sqrt{R}$ appearing. The given bound is thus valid in any dimension.
\end{proof}

\begin{lemma}
\label{lemm16a}
For $j\geq 0$,
\[\int_{\R^{d}}\psi_R(x) \|\nabla^j r(x)\|^2\dd x \underset{R\rightarrow+\infty}{\simeq} R\frac{K_d^2}{2} \|\varphi\|_L^2.\]
\end{lemma}
\begin{proof}
Note that 
\[
\nabla^j r(x)=\int_{S^{d-1}}(i\theta)^{\otimes j}e^{i\langle x,\theta\rangle}\sigma(d\theta),
\]
By Lemma \ref{lemm15a} applied with $s(\theta) = (i\theta)^{\otimes j}$, one has the asymptotic relation
\[
\nabla^j r(x)
=
\frac{K_d}{\|x\|^{\frac{d-1}{2}}}
\left(\frac{x}{\|x\|}\right)^{\otimes j}
\cos\left(\|x\|-\frac{\pi(d-1)}4+\frac{j\pi}{2}\right)
+
O\left(\frac{1}{\|x\|^{\frac{d+1}{2}}}\right),
\]
which implies that
\[
\|\nabla^j r(x)\|^2
=
\frac{K_d^2}{\|x\|^{d-1}}
\cos\left(\|x\|-\frac{\pi(d-1)}4+\frac{j\pi}{2}\right)^2
+
O\left(\frac{1}{\|x\|^{d+1}}\right),
\]
yielding that
\[\int_{\R^{d}}\psi_R(x) \|\nabla^j r(x)\|^2\dd x = R^d\int_{\R^{d}}\psi(x) r^2(Rx)\dd x \simeq R\frac{K_d^2}{2}\int_{\R^{d}}\frac{\psi(x)}{\|x\|^{d-1}}\dd x,\]
where the factor $1/2$ comes from the mean of $\cos^2$ and the Riemann--Lebesgue lemma.
\end{proof}

\subsection{Proof of Theorem \ref{t:squarecorr}}\label{e:proofsquarecorr} By standard Gaussian computations and stationarity, one has that
\[\Var\left(\sum_{j=0}^k c_j\int_{\R^d} \varphi_R(x) \left(\|\nabla^j B(x)\|^2 - 1\right)\dd x\right) = 2R^d\sum_{i,j=0}^k c_ic_j\int_{\R^d} \psi_R(x) \|\nabla^{i+j} r(x)\|^2\dd x .\]
Using Lemma \ref{lemm16a} together with assumption \eqref{e:trpiccolo}, one eventually deduces the asymptotic relation
\[\Var(X(\varphi_R)) \underset{R\rightarrow+\infty}{\simeq}  R^{d+1}{K_d^2}\sum_{i,j=0}^k c_ic_j \, \|\varphi\|^2_L = R^{d+1}\left(\sum_{j=0}^k c_j\right)^2{K_d^2}\, \|\varphi\|^2_L.\]
Now define
$$
Y_R := \sum_{j=0}^k c_j\int_{\R^d} \varphi_R(x) \left(\|\nabla^j B(x)\|^2 - 1\right) \dd x,
$$
so that $X(\varphi_R) - \mathbb{E}[X(\varphi_R)] = Y_R + T_R$. By virtue of \eqref{e:trpiccolo} it is now sufficient to show that, as $R\to\infty$, the sequence $$\widetilde{Y}_R:= \frac{Y_R}{R^{{(d+1)}/2}}, \quad R>0,$$
converges in distribution to a Gaussian random variable. To prove such a result, it is convenient to represent Berry's random field in terms of an adequate isonormal process over a space of Hermitian functions. To this end, consider the {\it real} Hilbert space $\mathfrak H$ generated by the class
\[
\left\{
h\in L_{\mathbb C}^2(S^{d-1},\sigma):
h(-\theta)=\overline{h(\theta)}
\right\},
\]
endowed with the scalar product $
\langle h,g\rangle_{\mathfrak H}
=
\int_{S^{d-1}}h(\theta)\overline{g(\theta)}\,\sigma(d\theta)$ (which is real-valued on $\mathfrak H$), and let $W=\{W(h):h\in\mathfrak H\}$ be an {\it isonormal Gaussian process} over $\mathfrak H$, as defined in \cite[Section 2.1]{nourdinpeccatibook}. For $x\in\mathbb R^d$, we set
\[
h_x(\theta)=e^{i\langle x,\theta\rangle};
\]
then, $h_x$ belongs trivially to $\mathfrak H$ and --- for the rest of the proof and without loss of generality --- we may assume that $B$ is realized as
\[
B(x)=W(h_x),
\]
where we use the fact that (consistently with \eqref{e:berrycov})
\[
\mathbb E[B(x)B(y)]
=\mathbb{E}[W(h_x)W(h_y)]= 
\langle h_x,h_y\rangle_{\mathfrak H}
=
\int_{S^{d-1}}e^{i\langle x-y,\theta\rangle}\,\sigma(d\theta).
\]
Elementary considerations also show that, for an ordered multi-index $\alpha =(\alpha_1,...,\alpha_j)\in \{1,...,d\}^j$, one has
\[
\partial^\alpha B(x) =\partial_{x_{\alpha_1}}\cdots \partial_{x_{\alpha_j}}B(x)
=
W(h_{x,\alpha}),
\qquad
h_{x,\alpha}(\theta)
=
i^j\theta^\alpha e^{i\langle x,\theta\rangle},
\]
where $\theta^\alpha = \theta_{\alpha_1}\cdots \theta_{\alpha_j} $, and $h_{x,\alpha}\in\mathfrak H$. We can now use the standard product formula stated in \cite[Theorem 2.7.10]{nourdinpeccatibook} to infer that
\[
\|\nabla^jB(x)\|^2-\mathbb E\|\nabla^jB(x)\|^2
=
I_2\left(
\sum_{\alpha \in \{1,...,d\}^j}h_{x,\alpha}\otimes h_{x,\alpha}
\right),
\]
where $I_2$ indicates a second order multiple Wiener-It\^o integral with respect to $W$, the sum is over all ordered multiindices $\alpha =(\alpha_1,...,\alpha_j)\in \{1,...,d\}^j$, and one has the explicit representation
\[
 \sum_{\alpha \in \{1,...,d\}^j} (h_{x,\alpha}\otimes h_{x,\alpha})(\theta, \eta) = e^{i\langle x, \theta+\eta\rangle} \sum_{\alpha \in \{1,...,d\}^j}
i^j\theta^\alpha\, i^j\eta^\alpha
=
e^{i\langle x, \theta+\eta\rangle} (-1)^j\langle\theta,\eta\rangle^j.
\]
It follows that
\[
Y_R=I_2(g_R),
\]
where
\[
g_R(\theta,\eta)
=
\int_{\mathbb R^d}
\varphi_R(x)e^{i\langle x,\theta+\eta\rangle}\,dx\,
P(\theta,\eta),
\]
with
\[
P(\theta,\eta)
=
\sum_{j=0}^k c_j(-1)^j\langle\theta,\eta\rangle^j.
\]
Equivalently,
\[
g_R(\theta,\eta)
=
R^d\widehat\varphi\big(R(\theta+\eta)\big)P(\theta,\eta).
\]
We observe that $g_R$ is symmetric and Hermitian (that is, $g_R(-\theta, -\eta) = \overline{g_R(\theta, \eta)}$), and also that $P$ is bounded and such that $P(\theta,-\theta)
=\sum_{j=0}^k c_j,$ and the first part of the forthcoming Lemma \ref{l:venerdi} yields that
\begin{equation}\label{e:eazy}
\| g_R\|^2_{\mathfrak{H}^{\otimes 2}}= \frac12 {\bf Var}(Y_R) \asymp R^{d+1}, \quad R\to\infty.
\end{equation}
Set 
\[ M_R := \sup_{\theta\in S^{d-1}} \int_{S^{d-1}}|g_R(\theta,a)|\,\sigma(\dd a). \] 
According to the second part of Lemma \ref{l:venerdi}, it is now sufficient to show that $M_R = o(R^{(d+1)/2})$. Since $P$ is bounded,
\[
|g_R(\theta,a)|
\le
C R^d
\left|
\widehat\varphi(R(\theta+a))
\right|,
\]
where $C$ is constant depending on $d$, $k$ and the coefficients $c_j$. We estimate $M_R$ by using Lemma \ref{l:sphericalstuff}, exploiting the explicit distribution of
\[
t=\langle\theta,a\rangle,
\]
where $a$ is a random element with law $\sigma$, and yielding the estimate
\begin{equation}\label{e:estimate}
\int_{S^{d-1}}F(\|\theta+a\|)\,\sigma(da)
\le
C\int_0^2 F(s)s^{d-2}\,ds,
\end{equation}
for every $F$ nonnegative and nonincreasing, and some absolute constant only depending on the dimension $d$. If $\varphi\in\mathcal S(\mathbb R^d)$, then, for all integers $N$, one has that
\[
|\widehat\varphi(\xi)|\le C_N(1+\|\xi\|)^{-N}.
\]
Choosing $N>d-1$ and $F(r) = (1+Rr)^{-N}$ in \eqref{e:estimate}, one obtains
\[
M_R\le CR,
\qquad
\|T_R\|_{\mathrm{op}}=O(R).
\]
If $\varphi=\mathbf 1_K$ with $K$ bounded, convex, having a $C^2$ boundary and with non-vanishing Gauss curvature, then
\[
|\widehat{\mathbf 1_K}(\xi)|
\le
C(1+\|\xi\|)^{-\frac{d+1}{2}}
\]
(see e.g. \cite{BHI} and the references therein), and analogous computations yield
\[
\|T_R\|_{\mathrm{op}}
\le
\begin{cases}
CR, & d=2,\\
CR\log R, & d=3,\\
CR^{\frac{d-1}{2}}, & d\ge4.
\end{cases}
\]
The proof of Theorem \ref{t:squarecorr} is concluded.

\medskip

\begin{lemma}\label{l:venerdi} Fix $d\geq 2$. Let $\mu$ be a symmetric\footnote{That is, $\mu$ is such that $\mu(A) = \mu(-A)$ for all measurable $A\subset \R^{d}$} Borel finite measure on $\R^{d}$, and consider the real Hilbert space $\mathfrak H$ generated by the class
\[
\left\{
h\in L_{\mathbb C}^2(\R^{d},\mu):
h(-y)=\overline{h(y)}
\right\},
\]
endowed with the real-valued scalar product
\[
\langle h,g\rangle_{\mathfrak H}
=
\int_{R^{d}}h(y)\overline{g(y)}\,\mu(\dd y).
\]
Let $W=\{W(h):h\in\mathfrak H\}$ be an {\it isonormal Gaussian process} over $\mathfrak H$, as defined in {\rm \cite[Section 2.1]{nourdinpeccatibook}}. Let $\{g_R : R>0\}\subset \mathfrak{H}^{\odot 2}$ and consider the corresponding sequence of double stochastic Wiener-It\^o integrals (with respect to $W$) $\{I_2(g_R) : R>0\}$. Then,
\begin{equation}\label{e:trivialvariance} {\bf Var}(I_2(g_R)) = 2\|g_R\|^2_{\mathfrak{H}^{\otimes 2}}.
\end{equation}
Moreover, if 
\begin{equation}\label{e:condizione}
M_R := \sup_{y\in \R^{d} \,: \, y\in {\rm supp}(\mu)} \int_{\R^{d}}|g_R(y,a)|\,\mu(\dd a) = o\big(\|g_R\|_{\mathfrak{H}^{\otimes 2}}\big),\quad R\to\infty,
\end{equation}
one has that $G_R:= (\sqrt{2}\|g_R\|_{\mathfrak{H}^{\otimes 2}})^{-1}\, I_2(g_R)$ converges in distribution to a standard Gaussian random variable.
\end{lemma}
\begin{proof}
Relation \eqref{e:trivialvariance} is a direct consequence of the usual isometry properties of multiple stochastic integrals (see \cite[Proposition 2.7.5]{nourdinpeccatibook}).
Now introduce the (symmetric and Hermitian) {\it contraction kernel} (see \cite[Section B.4]{nourdinpeccatibook})
$$
\R^{d}\times \R^{d}\ni (\theta, \eta)\mapsto g_R\otimes_1 g_R (x, y) := \int_{\R^{d}} g_R(x, a)g_R(y, -a)\mu(\dd a). 
$$
According to the Fourth Moment Theorem for multiple integrals stated in \cite[Theorem 5.2.7]{nourdinpeccatibook}, one has that the asymptotic Gaussianity of $G_R$ is equivalent to the relation
\[ \lim_{R\to\infty}\frac{\|g_R\otimes_1 g_R\|_{\mathfrak{H}^{\otimes 2}}}{\| g_R\|^2_{\mathfrak{H}^{\otimes 2}}} = 0.
\]
Now let $T_R$ be the Hilbert-Schmidt operator on $\mathfrak H$ defined by
\[
(T_R h)(y)
=
\int_{\R^{d}}g_R(y,a)h(-a)\,\mu(da),
\]
and denote by $\{\lambda_j : j\geq 1\}$ the sequence of its eigenvalues (note that $T_R$ is self-adjoint on $\mathfrak H$, since $g_R$ is symmetric). Then, $\|T_R\|^2_{\mathrm{HS}}=\|g_R\|^2_{\mathfrak{H}^{\otimes 2}} = \sum_{j} \lambda_j^2$, and
\[
\|g_R\otimes_1g_R\|_{\mathfrak{H}^{\otimes 2}}^2
=
\sum_{j} \lambda_j^4
\le
\|T_R\|^2_{\mathrm{op}}\, \|g_R\|_{\mathfrak{H}^{\otimes 2}}^2
\]
(see \cite[Section 2.7.4]{nourdinpeccatibook} for details), that is,
\[
\frac{\|g_R\otimes_1g_R\|_{\mathfrak{H}^{\otimes 2}}}{\|g_R\|_{\mathfrak{H}^{\otimes 2}}^2}
\le
\frac{\|T_R\|_{\mathrm{op}}}{\|g_R\|_{\mathfrak{H}^{\otimes 2}}}.
\]
We claim that \(\|T_R\|_{\mathrm{op}}\le M_R\). To see this, we use a self-contained argument that corresponds to a version of the well-known {\it Schur's test} for integral operators (see e.g. \cite{halmosbook}). For every \(h\in\mathfrak H\), one has
\[
|(T_Rh)(y)|
\le
\int_{\R^{d}} |g_R(y,a)|\,|h(-a)|\,\mu(\dd a).
\]
By Cauchy--Schwarz, for every $y\in {\rm supp}(\mu)$,
\[
\begin{aligned}
|(T_Rh)(y)|^2
&\le
\left(\int_{\R^{d}} |g_R(y,a)|\,\mu(\dd a)\right)
\left(\int_{\R^{d}} |g_R(y,a)|\,|h(-a)|^2\,\mu(\dd a)\right)\\
&\le
M_R
\int_{\R^{d}} |g_R(y,a)|\,|h(-a)|^2\,\mu(\dd a).
\end{aligned}
\]
Integrating with respect to \(y\), and using Fubini, yields
\[
\begin{aligned}
\|T_Rh\|_{\mathfrak H}^2
&\le
M_R
\int_{\R^{d}}\int_{\R^{d}}
|g_R(y,a)|\,|h(-a)|^2\,\mu(\dd a)\mu(\dd y)\\
&=
M_R
\int_{\R^{d}} |h(-a)|^2
\left(\int_{\R^{d}}|g_R(y,a)|\,\mu(\dd y)\right)\mu(\dd a).
\end{aligned}
\]
Since \(g_R\) is symmetric in its two variables, the inner integral is also bounded by \(M_R\). Hence,
\[
\|T_Rh\|_{\mathfrak H}^2
\le
M_R^2
\int_{\R^{d}}|h(-a)|^2\,\mu(\dd a)
=
M_R^2\|h\|_{\mathfrak H}^2,
\]
where we used the invariance of \(\mu\) under \(a\mapsto -a\). Taking the supremum over \(h\neq0\), we obtain $\|T_R\|_{\mathrm{op}}\le M_R,$
    
\end{proof}

\begin{lemma}\label{l:sphericalstuff}
Let $d\geq 2$, and let \(a\) be a random element distributed according to the uniform probability measure \(\sigma\) on
\(S^{d-1}\), and fix \(\theta\in S^{d-1}\). Then the law of
\[
t=\langle \theta,a\rangle
\]
does not depend on \(\theta\), and has density
\[
c_d(1-t^2)^{\frac{d-3}{2}}\mathbf 1_{(-1,1)}(t),
\]
where \(c_d\) is a normalizing constant. Moreover, if
\[
s=\|\theta+a\|,
\]
then for every nonnegative nonincreasing function \(F\),
\[
\int_{S^{d-1}}F(\|\theta+a\|)\,\sigma(da)
\le
C_d\int_0^2 F(s)s^{d-2}\,ds,
\]
for some constant \(C_d>0\) depending only on the dimension.
\end{lemma}

\begin{proof}
The fact that the law of \(t=\langle\theta,a\rangle\) does not depend on \(\theta\)
follows from the rotation invariance of \(\sigma\). Hence we may assume without loss of generality that \(\theta=e_d = (0,...,0,1)\), so that
\(t=a_d\). It is well known that \( U = a_d^2\) has a beta distribution with parameters
\(\frac12\) and \(\frac{d-1}{2}\). Therefore \(U\) has density
\[
\frac{1}{B\left(\frac12,\frac{d-1}{2}\right)}
u^{-1/2}(1-u)^{\frac{d-3}{2}},
\qquad 0<u<1.
\]
Since \(a_d\) is symmetric and \(U=t^2\), a standard change of variables shows
that \(t=a_d\) has density
\[
c_d(1-t^2)^{\frac{d-3}{2}},
\qquad -1<t<1
\]
(see \cite[formula (1.26)]{SymMultiBook} for more details). Now set \(s=\|\theta+a\|\). Since
\[
\|\theta+a\|^2=2(1+\langle\theta,a\rangle)=2(1+t),
\]
it follows that the law of \(s\) has density proportional to
\[
s^{d-2}\left(1-\frac{s^2}{4}\right)^{\frac{d-3}{2}},
\qquad 0<s<2.
\]
Let \(F\) be nonnegative and nonincreasing. Then
\[
\int_{S^{d-1}}F(\|\theta+a\|)\,\sigma(da)
=
\int_0^2 F(s)\,\rho_d(s)\,ds,
\]
where \(\rho_d(s)\) is the density of $s$ above. Since
\[
\rho_d(s)
=
C_d s^{d-2}\left(1-\frac{s^2}{4}\right)^{\frac{d-3}{2}},
\]
and the latter factor is bounded on \((0,2)\) for \(d\ge 3\), the desired estimate
follows immediately in that case. For \(d=2\), one has
\[
\rho_2(s)\asymp \left(1-\frac{s^2}{4}\right)^{-1/2},
\]
which has an integrable singularity at \(s=2\). Splitting the integral e.g. into
\((0,1)\) and \((1,2)\), and using that \(F\) is nonincreasing, one still obtains
\[
\int_0^2 F(s)\rho_2(s)\,ds
\le
C\int_0^2 F(s)\,ds,
\]
which coincides with the desired bound.
\end{proof}

\subsection{General analysis of critical points} \label{s:gencrit}

We will now move towards the proofs of Proposition \ref{p:squarecorrex} (given in Sections \ref{ss:proofsquarecorrex}), Proposition \ref{lemm14} (see Section \ref{ss:proofvarexpansion}) and Theorem \ref{t:integrablefields} (see Section \ref{ss:proofintegrablefields}). Since the most technically demanding (and innovative) part of our proofs revolves around the asymptotic analysis of critical points of smooth Gaussian functions, we devote the present section and the subsequent Section \ref{ss:levedep} to the study of critical points of Gaussian fields satisfying the following set of generic assumptions.
\medskip

\begin{assumption}\label{ass:A}
Fix $d\geq 2$. In what follows we denote by 
$$
f = \{f(x) : x\in \R^d\}
$$
a real, centered stationary Gaussian field on $\R^d$, with covariance function
$$
r(x) := \mathbb{E}[f(x)f(0)] =\int_{\R^d} e^{i\langle \xi, x\rangle} \omega(\dd \xi), \quad x\in\R^d,
$$
(where $\omega$ is the spectral measure of $f$), and satisfying the following assumptions:
\begin{itemize}
\item[(a)] $f$ is a.s. of class $\mathcal{C}^3$;
\item[(b)] $\lim_{\|x\|\to \infty} r(x) = 0$;
\item[(c)] the random element $(f(x),\sqrt{d}\nabla f(x))$ is a standard Gaussian vector so that, in particular,
\[r(0) = 1\quand \nabla^2 r(0) = -\frac{1}{d}\Id.\]
\item[(d)] the measure $\omega$ either has a non-zero part that is absolutely continuous with respect to the Lebesgue measure in $\R^d$, or it has a smooth density with respect to the surface measure of a centered sphere in $\R^d$. 
\end{itemize}

\end{assumption}

\medskip 

As in the statement of Proposition \ref{p:squarecorrex}, we write
\[{\bf CP}(u) = {\bf CP}(u,f) := \enstq{x\in \R^d}{\nabla f(x) = 0\;\text{and}\; f(x)\geq u}, \quad u\in \R,\]
to indicate the set of critical points of $f$ above $u$ and we observe that, under
\cref{ass:A} this set forms a.s. a locally finite simple point process (see \cite[Sec. 11.2]{AdlerTaylor}). Given a nonnegative test function $\varphi$ as in Remark \ref{r:convention}, we focus on the observable
\[X_1(f,u,\varphi) = \sum_{x\in {\bf CP}(u,f)}\varphi(x),\]
and observe that such a definition extends trivially to a generic $\varphi\in \mathcal{S}$ whenever the above sum is a.s. absolutely converging (which is always the case under \cref{ass:A}). In this section, we are interested in formulas for the expectation $\E[X_1(f,u,\varphi)]$ and for the second factorial moment \begin{equation}
\label{e:facmom}
\E[X_1(f,u,\varphi)^{[2]}] = \E[X_1(f,u,\varphi)^2] - \E[X_1(f,u,\varphi^2)]
\end{equation}
(which is always well-defined under \cref{ass:A}. To this end, we first recall the {\bf co-area formula} (see e.g. \cite[Theorem 3.2.12]{federer2014}, as well as \cite[Sec. 11.2]{AdlerTaylor}), yielding that one has the formal identity
\begin{equation}\label{e:coa}
X_1(f,u,\varphi) = \int_{\R^d}\varphi(x) \delta_0(\nabla f(x))|\det \nabla^2f(x)|\one_{[u,+\infty[}(f(x))\dd x,\end{equation}
that one has to regard as an a.s. convergence result, obtained by replacing $\delta_0$ by a sequence of appropriate mollificators.
\subsubsection{Kac--Rice formulae}
Roughly speaking, Kac--Rice formulae (see \cite{AdlerTaylor, Aza09, GassStecconi2024, ancona2025zeroscriticalpointsgaussian}) are analytical devices allowing one to take expectations on both sides of relations such as \eqref{e:coa}, and then to apply a formal Fubini-type argument to the right-hand side, thereby representing expectations as integrals of suitable conditional kernels. The next statement consists in an application of Kac-Rice techniques to evaluate the quantities $\mathbb{E}[X_1(f,u,\varphi)]$ and $\mathbb{E}[X_1(f,u,\varphi)^{[2]}]$ introduced above (see, in particular, \eqref{e:facmom}). Given a test function $\varphi$, we recall the notation $\psi = \varphi*\varphi^{-}$ introduced in \eqref{e:psi}.
\begin{lemma}
\label{lemm1} Let {\rm \cref{ass:A}} prevail and let $\varphi$ be a test function as in Remark \ref{r:convention}. Then, one has that
\begin{equation}\label{e:ekr}\E[X_1(f,u,\varphi)] = \frac{d^{d/2}}{(2\pi)^{d/2}}\E[|\det \nabla^2f(0)|\one_{[u,+\infty[}(f(0))]\int_{\R^d}\varphi(x)\dd x,\end{equation}
and
\begin{equation}\label{e:vkr}\E[X_1(f,u,\varphi)^{[2]}] = \int_{\R^d}\psi(x)\rho_u(x)\dd x,\end{equation}
with
\begin{equation}\label{e:rhou}\rho_u(x) = \frac{\E\left[|\det \nabla^2f(0)||\det \nabla^2f(x)|\one_{[u,+\infty[}(f(0))\one_{[u,+\infty[}(f(x))\,\middle|\,\nabla f(0)=\nabla f(x) = 0\right]}{\sqrt{\det 2\pi\Var(\nabla f(0),\nabla f(x))}}.\end{equation}
\end{lemma}
\begin{proof}
When the spectral measure $\omega$ has a non-zero part that is absolutely continuous with respect to the Lebesgue measure in $\R^d$, then the theorem is a direct consequence of \cite[Thm 11.2.1, Cor. 11.5.2]{AdlerTaylor}, the non-degeneracy condition can be inferred for instance from \cite[Ex. 3.4]{Aza09}. In the case where the spectral measure $\omega$ has a smooth density with respect to the surface measure of a centered sphere in $\R^d$ of radius $\lambda$, note that the vector $(f(x), \nabla^2 f(x))$ is degenerate, from the relation $\Delta f(x) + \lambda^2 f(x) = 0$. Nevertheless, \cite[Thm 11.2.1]{AdlerTaylor} and \cite[Cor. 11.5.1]{AdlerTaylor} can still be applied since at a fixed $x$, the vector $\nabla f(x)$ is independent of $(f(x),\nabla^2 f(x))$.
\end{proof}

The function $\rho_u$ has been proved to be integrable near the origin (despite the obvious singularity at $x=0$) under \cref{ass:A} (see \cite[Thm. 1.1]{GassStecconi2024}). Such a result will also follow as a corollary of the quantitative estimates derived in the forthcoming Section \ref{ss:diagKR}.
\subsubsection{Chaotic expansion}\label{ss:chaosProof1}
Given a square-integrable random variable $F = F(f)$ (that is, $F$ is measurable with respect to the $\sigma$-field generated by the Gaussian field $f$), one always has that $F$ admits a {\bf Wiener-It\^o chaos expansion}, meaning that $F$ has a unique representation of the form
\begin{equation}\label{e:wienerchaosisthereasonwearehere}
F = \mathbb{E}[F] +\sum_{q=1}^\infty F^{(q)},
\end{equation}
where the series converges in square-mean, and $F^{(q)} := {\rm proj}(F\, |\, C_q)$ is the projection of $F$ onto the $q$th {\bf Wiener chaos} $C_q$ of $f$, as defined e.g. in \cite[Section 2.2]{nourdinpeccatibook}.\footnote{Observe that, if $f$ is represented as a subset of an isonormal Gaussian process $W$ over some Hilbert space $\mathfrak H$ (as we did for $f= B$ in the proof of Theorem \ref{t:squarecorr}), then ${\bf F}^{(q)}$ admits a unique representation of the form ${\bf F}^{(q)} = I_q(f_q)$, where $I_q$ is a multiple Wiener-It\^o integral of order $q$ with respect to $W$, and $f_q\in \mathfrak{H}^{\odot q}$. See again \cite[Chapter 2]{nourdinpeccatibook}.} A rigorous introduction to Wiener chaos expansions can be found in \cite[Chapter 2]{nourdinpeccatibook}, as well as in \cite[Chapter 1]{nualartbook} or \cite[Chapter 2]{JansonBook}. Since the functionals of the Gaussian field $f$ considered here are local functionals, the analysis of their chaotic decompositions can be made considerably simpler than in the general case. The following discussion roughly follows the approach developed in \cite{gass2025spectralcriteriaasymptoticslocal}. Before we derive the chaotic decomposition of $X_1(f,u,\varphi)$, we need to introduce some further notation.
\jump

It is easily checked that, under \cref{ass:A} and for fixed $x\in \R^d$, the Gaussian vectors $f(x)$, $\nabla f(x)$ and $\nabla^2 f(x) + \frac{1}{d}f(x)\Id$ are mutually independent Gaussian vectors. The vector $\nabla^2 f(x) + \frac{1}{d}f(x)\Id$ is not necessarily a standard Gaussian element on the space of symmetric matrices (endowed with the Frobenius norm) on $\R^d$. Denoting by $T$ the square root of the covariance matrix of the vector $\nabla^2 f(x) + \frac{1}{d}f(x)\Id$ (which is independent of $x$), one can define $D(x)$ as the standard Gaussian vector taking values in the support of $\nabla^2 f(x) + \frac{1}{d}f(x)\Id$, satisfying the relation
\begin{equation}\label{e:TD}
TD(x) = \nabla^2 f(x) + \frac{1}{d}f(x)\Id
\end{equation}
The stationary Gaussian field 
\begin{equation}\label{e:W}
W(x) = (f(x), \sqrt{d} \nabla f(x), D(x))
\end{equation}
can then be identified, at fixed $x$, with a standard Gaussian vector of dimension $m$, with 
\begin{equation}
\label{e:m}
m = 1 + d + \dim\left(\nabla^2 f(x) + \frac{1}{d}f(x)\Id\right).
\end{equation}
In the isotropic case, the transformation $T$ can be made explicit, see for instance \cite[Lemma 4.10]{gass2025spectralcriteriaasymptoticslocal}. In particular, for Berry's random wave, one has the following lemma.
\begin{lemma}
\label{lemm2d}
Let $\tau = \frac{1}{d(d+2)}$, and let $f=B$ be defined as in \eqref{e:helmholtz}---\eqref{e:berrycov}. Then,
\[D(x) := \frac{1}{\sqrt{2\tau}}\left(\nabla^2 f(x) + \frac{1}{d}f(x)\Id\right)\]
is a standard Gaussian element over the space of traceless symmetric matrices on $\R^d$ (equipped with the Frobenius scalar product).
\end{lemma}
\begin{proof}
Note first that 
\[\Tr(\nabla^2f(x) + \frac{1}{d}f(x)\Id) = \Delta f(x) + f(x)=0,\]
so that this random matrix lives in the space of traceless symmetric matrices. An elementary covariance computation yields that for a traceless symmetric matrix $A$,
\[\E\left[\left\langle \nabla^2 f(x) + \frac{1}{d}f(x)\Id,A\right\rangle^2\right] = 2\tau\|A\|^2,\]
and the conclusion follows.
\end{proof}
\jump
Now let $f$ be as in \cref{ass:A}. For $u\in \R$, we denote by $F_u$ the generalized function defined by
\begin{align}\notag
F_u(W(x)) &:= \delta_0(\nabla f(x))|\det \nabla^2f(x)|\one_{[u,+\infty[}(f(x))\\ \label{e:gloom}
&= \delta_0(\nabla f(x))\left|\det \left(TD(x) - \frac{1}{d} f(x)\Id\right)\right|\one_{[u,+\infty[}(f(x)).
\end{align}
We also recall the following standard fact (see \cite[Chapter 2]{nourdinpeccatibook}): given a square-integrable random variable $F = F(W(x))$, its Wiener-It\^o chaos expansion \eqref{e:wienerchaosisthereasonwearehere} is given by
\begin{equation}\label{e:chaosisgod}F(W(x)) = \sum_{q=0}^{+\infty}  F^{(q)}(W(x)) = \sum_{q=0}^{+\infty} \frac{1}{q!}\langle F^{(q)}, H_{W(x)}^{(q)}\rangle,\end{equation}
where $\langle\cdot, \cdot\rangle$ stands for the Frobenius scalar product, $y\mapsto H_y^{(q)}\in (\R^m)^{\otimes q} $ is the (symmetric) tensor-valued mapping defined by
\[H_y^{(q)} := (-1)^qe^{\frac{\|y\|^2}{2}}\nabla^q \left(y\mapsto e^{-\frac{\|y\|^2}{2}}\right),\quad \mbox{and}\quad F^{(q)} := \E[F(W(x))H_{W(x)}^{(q)}].\]

The next statement exploits representation \eqref{e:chaosisgod} to derive the Wiener chaos expansion of the random variable $X_1(f,u,\varphi)$

\begin{lemma}
\label{lemm2c} Let {\rm \cref{ass:A}} prevail, and let $\varphi$ be as in Remark \ref{r:convention}. Then, for every $u\in \R$, the quantity
\begin{equation}\label{e:fuq}
F_u^{(q)} := \E[F_u(W(x))H_{W(x)}^{(q)}],
\end{equation}
where $F_u(W(x))$ is defined in \eqref{e:gloom}, is a well-defined symmetric element of $(\R^m)^{\otimes q}$, where $m$ is defined in \eqref{e:m} . Moreover, for $q\geq 1$, the projection of $X_1(f,u,\varphi)$ onto the $q$-th Wiener chaos $ C_q$ is given by
\begin{equation}\label{e:localchaos}X_1^{(q)}(f,u,\varphi) = \frac{1}{q!}\int_{\R^d}\varphi(x)\langle F_u^{(q)}, H_{W(x)}^{(q)}\rangle\dd x=\frac{1}{q!}\int_{\R^d}\varphi(x)F_u^{(q)}(W(x))\rangle\dd x, \end{equation}
where we used the notation
\begin{equation}\label{e:nottola}
F_u^{(q)}(W(x)):= \frac{1}{q!}\langle F_u^{(q)}, H_{W(x)}^{(q)}\rangle, \quad q\geq 1.
\end{equation}

\end{lemma}
\begin{proof}
It directly follows from the co-area formula \eqref{e:coa} and the linearity of chaotic decomposition, as done in \cite[Sec. 2.3]{gass2025spectralcriteriaasymptoticslocal}).
\end{proof}

\smallskip

The next two statements provide an explicit expression for the quantities $F^{(1)}(W(x))$ and $F^{(2)}(W(x))$ defined in \eqref{e:nottola}. They are both valid under the assumptions of Lemma \ref{lemm2c}.

\begin{lemma}
\label{lemm2b}
One has that
\begin{align*}
F_u^{(1)}(W(x)) &= \frac{d^{d/2}}{(2\pi)^{d/2}}\E\left[f(x)|\det \nabla^2 f(x)|\one_{[u,+\infty[}(f(x))\right]f(x)\\
&+ \frac{d^{d/2}}{(2\pi)^{d/2}}\left\langle \E\left[D(x)|\det \nabla^2 f(x)|\one_{[u,+\infty[}(f(x))\right],D(x)\right\rangle.
\end{align*}
If $f = B$, as defined in \eqref{e:helmholtz}---\eqref{e:berrycov}, the second term in the previous sum vanishes.
\end{lemma}
\begin{lemma}
\label{lemm2}
One has that
\begin{align*}
F_u^{(2)}(W(x)) &= \frac{d^{d/2}}{2(2\pi)^{d/2}}\E\left[(f(x)^2-1)|\det \nabla^2 f(x)|\one_{[u,+\infty[}(f(x))\right](f(x)^2-1)\\
&- \frac{d^{d/2+1}}{2(2\pi)^{d/2}}\E\left[|\det \nabla^2 f(x)|\one_{[u,+\infty[}(f(x))\right](\|\nabla f(x)\|^2-1)\\
&+ \frac{d^{d/2}}{2(2\pi)^{d/2}}\left\langle \E\left[(D(x)^{\otimes 2}-\Id)|\det \nabla^2 f(x)|\one_{[u,+\infty[}(f(x))\right],(D(x)^{\otimes 2}-\Id)\right\rangle\\
&+ \frac{d^{d/2}}{2(2\pi)^{d/2}}\left\langle \E\left[f(x)D(x)|\det \nabla^2 f(x)|\one_{[u,+\infty[}(f(x))\right],f(x)D(x)\right\rangle.
\end{align*}
If $f = B$, as defined in \eqref{e:helmholtz}---\eqref{e:berrycov}, then  the fourth term in the previous sum vanishes, and the third term is equal to
\[\frac{d^{d/2}}{2(2\pi)^{d/2}}\E\left[\left(\frac{\|D(x)\|^2}{c}-1\right)|\det \nabla^2 f(x)|\one_{[u,+\infty[}(f(x))\right](\|D(x)\|^2-c),\]
with $D(x)$ given by Lemma \ref{lemm2d} and $c = \frac{(d-1)(d+2)}{2}$.
\end{lemma}
\begin{proof}[Unified proofs of Lemmas \ref{lemm2b} and \ref{lemm2}]
The proof of formulas of the above two lemmas is straightforward from the definition, using for a continuous function $M$ with polynomial growth the identity:
\[\E[M(W(x))\delta_0(\nabla f(x))] = \frac{d^{d/2}}{(2\pi)^{d/2}}\E[M(W(x))|\nabla f(x) = 0].\]
For the isotropic random wave model, recall the expression of $D(x)$ in \ref{lemm2d}. The second chaotic projection must be invariant to isometric transformations of the Hessian. Up to a constant, the only linear form on symmetric matrices that is invariant by orthogonal transformation is the trace, which is zero for the hesHessian random waves. From this fact th,e cancellations above are deduced. Similarly, up to a constant, the only bilinear form in the space of traceless symmetric matrices invariant by isometry is the Frobenius norm, which yields the last formula in the Lemma \ref{lemm2}.
\end{proof}
\subsubsection{Variance bounds}\label{ss:variancebounds}
Let $\Omega$ be the covariance function of the $m$-dimensional Gaussian field $W$ introduced in \eqref{e:W}. The Kac density $\rho_u$ (defined in formula \eqref{e:rhou} of Lemma \ref{lemm1}) evaluated at point $x$ is actually a functional of the covariance function $\Omega(x)$ so that one can write $\rho_u(x) = \tilde{\rho}_u(\Omega(x))$. It is straightforward to observe that the function $\tilde{\rho}_u$ is analytic as soon as the Gaussian vector $(W(0),W(x))$ is non-degenerate, which is satisfied for all $x\neq 0$, as a consequence of \cref{ass:A}. The Taylor expansion of $\tilde{\rho}_u$ is given by the formula (see \cite[Sec. 2.2] {gass2025spectralcriteriaasymptoticslocal} for details)
\[\tilde{\rho}_u(\Omega(x)) = \E[F_u(W(0))F_u(W(x))] = \sum_{q=0}^{\infty}\frac{1}{q!}\Omega(x)^{\otimes q}(F_u^{(q)},F_u^{(q)}),\]
which implies, after integration and thanks to \eqref{e:vkr}, the equality
\[\E[X_1(f,u,\varphi)^{[2]}] = \int_{\R^d}\psi(x)\sum_{q=0}^{\infty}\frac{1}{q!}\Omega(x)^{\otimes q}(F_u^{(q)},F_u^{(q)})\dd x.\]

\begin{remark}{\rm it is important to remark that one cannot exchange the sum and the integral in the above formula. Doing so will account to the usual second moment of $X_1(f,u,\varphi)$ instead of the second factorial moment, since the second moment of $X_1(f,u,\varphi)$ is the sum of the second moments of its chaotic projections, given by Lemma \ref{lemm2c}.}
\end{remark}

The following useful result is a direct consequence of the bound of the remainder in the Taylor expansion near the origin of the function $\tilde{\rho}_u$.

\begin{lemma}
\label{lemm3}
Let $x\neq 0$. Then
\[\left|\rho_u(x) - \sum_{k=0}^{q-1}\frac{1}{k!}\Omega(x)^{\otimes k}(F_u^{(k)},F_u^{(k)})\right|\leq \frac{1}{q!}\sup_{t\in[0,1]}\|\nabla^q\tilde{\rho}_u(t\Omega(x))\| \times \|\Omega(x)\|^q,\]
where the norms on the right-hand side of the previous expression are considered in appropriate spaces of tensors and matrices.
\end{lemma}

\medskip

\subsection{Level dependence}\label{ss:levedep}
In this section, we make the dependence on the level $u$ of the various quantities introduced in the previous section more explicit. Our main objective is to investigate their asymptotic behavior in the regime $u\to\infty$.
\subsubsection{Chaotic decompositions}
We next state a lemma describing the general mechanism by which large-$u$ asymptotics can be obtained through the Laplace method, which is the object of the next Lemma. Variants of this statement will be used repeatedly throughout the remainder of the section.
\begin{lemma}
\label{lemm4}
Let {\rm \cref{ass:A}}, as well as the notation introduced in Section \ref{ss:chaosProof1}, prevail. Let $\alpha\in \R$, and let $h$ be a differentiable function on $\R$, such that, for large $u$, one has $h(u)\simeq C u^\alpha$ and $h'(u) = o(h(u))$. Then 
\[\int_u^{+\infty}h(v)e^{-\frac{v^2}{2}}\dd v=\frac{h(u)}{u}e^{-\frac{u^2}{2}}\left(1+O\left(\frac{1}{u}\right)\right)\]
As a consequence, if $G$ is a function with polynomial growth, one has that
\[\E[f(x)^\alpha G(D(x))|\det \nabla^2 f(x)|\one_{[u,+\infty[}(f(x))] = \E[G(D(x))]u^{\alpha+d-1}\frac{1}{d^d\sqrt{2\pi}}e^{-\frac{x^2}{2}}\left(1+O\left(\frac{1}{u}\right)\right).\]
\end{lemma}
\begin{proof}
Observe the equality
\begin{align*}
\int_u^{+\infty}h(v)e^{-\frac{v^2}{2}}\dd v &= \frac{e^{\frac{u^2}{2}}}{u}\int_0^{+\infty} h\left(u,+\frac{t}{u}\right)e^{-t}e^{-\frac{t^2}{2u^2}}\dd t
\end{align*}
the difference between $h(u)$ and $h(u,+t/u)$ is a $o(h(u)/u)$ uniformly on compacts of $t$, and the conclusion follows from dominated convergence. For the second statement, observe that one has
\begin{eqnarray*}&&\E[f(x)^\alpha G(D(x))|\det \nabla^2 f(x)|\one_{[u,+\infty[}(f(x))] \\ &&= \frac{1}{\sqrt{2\pi}}\int_u^{+\infty} v^\alpha \E\left[G(D(x))\left|\det TD(x) - \frac{v}{d}\Id\right|\right]e^{-\frac{v^2}{2}}\dd v.\end{eqnarray*}
As $u$ grows large, $\E[G(D(x))|\det TD(x) -\frac{u}{d}\Id|]\simeq \frac{u^d}{d^d}\E[G(D(x))]$ and the conclusion follows from the previous statement.
\end{proof}

The following statement is a consequence of the previous lemma.

\begin{lemma}
\label{lemm5}
Under {\rm \cref{ass:A}}, and for $\varphi$ as in Remark \ref{r:convention}, one has that, as $u\to \infty$,
\begin{equation}\label{e:vaughan}\E[X_1(f,u,\varphi)] = \left(\int_{\R^d} \varphi(x)\, \dd x\right)\frac{1}{d^{d/2}(2\pi)^{(d+1)/2}}u^{d-1}e^{-\frac{u^2}{2}}\left(1+O\left(\frac{1}{u}\right)\right),\end{equation}
where the constants involved in the $O(\cdot)$ notation are independent of $\varphi$.
\end{lemma}
\begin{proof}
The conclusion follows by combining \eqref{e:ekr} with Lemma \ref{lemm4}.
\end{proof}
We define the constant
\[\gamma_u := \frac{u^{d+1}}{2d^{d/2}(2\pi)^{\frac{d+1}{2}}}e^{-\frac{u^2}{2}}.\]
The above Lemma \ref{lemm5} implies the following lemma about large $u$ asymptotics of the projection on the second chaos of critical points.
\begin{lemma}
\label{lemm6}
Let {\rm \cref{ass:A}} prevail and recall the notation introduced in \eqref{e:nottola}. Then, there is a constant $C$ such that, for $u>0$,
\[\left|\E[F_u^{(2)}(W(0))F_u^{(2)}(W(x))]-\gamma_u^2\E[(f(0)^2-1)(f(x)^2-1)]\right|\leq C\frac{\gamma_u^2}{u^2}\|\Omega(x)\|^2,\]
where the covariance function $\Omega$ was defined in Section \ref{ss:variancebounds}. In particular, for $\varphi$ as in Remark \ref{r:convention},
\[\Var\left(X_1^{(2)}(f,u,\varphi)-\gamma_u\int_{\R^d}\varphi(x)(f(x)^2-1)\dd x\right)\leq C\frac{\gamma_u^2}{u^2}\int_{\R^d}|\psi(x)|\times \|\Omega(x)\|^2\dd x,\]
where the norm is Frobenius, and we have used \eqref{e:psi}.
\end{lemma}
\begin{proof}
The first statement directly follows by observing that the mapping
\[(v,w)\mapsto \E[F_u^{(2)}(W(0))F_v^{(2)}(W(x))|f(0)=v, f(x)=w]\]
is analytic in $v,w$. From Lemma \ref{lemm4}, one deduces that, for the first coefficient in the expression of $F^{(2)}(W(x))$ given in Lemma \ref{lemm2}, one has the asymptotic relations
\[\frac{d^{d/2}}{2(2\pi)^{d/2}}\E\left[(f(x)^2-1)|\det( TD(x) + f(x)\Id)|\one_{[u,+\infty[}(f(x))\right] =  \gamma_u + O(u^d).\]
The remaining three coefficients appearing in the expression of $F^{(2)}(W(x))$ given in Lemma \ref{lemm2} are all $O(u^d)$, as one can deduce from a further application of Lemma \ref{lemm4}. The first part of the statement now follows from a direct covariance computation. To prove the second part of the statement, we use Lemma \ref{lemm2c} and exploit the identity
\begin{align*}
&\Var\left(X_1^{(2)}(f,u,\varphi)-\gamma_u\int_{\R^d}\varphi(x)(f(x)^2-1)\dd x\right) \\
&=\int_{\R^d}\psi(x)\left(\E[F_u^{(2)}(W(0))F_u^{(2)}(W(x))]-\gamma_u^2\E[(f(0)^2-1)(f(x)^2-1)]\right)\dd x.
\end{align*}
The conclusion follows from the first point.
\end{proof}
\subsubsection{Near-diagonal estimates in the Kac density}\label{ss:diagKR}
The goal of this subsection is to prove the integrability of the Kac density given in \eqref{e:rhou} near $0$, by explicitly tracking the dependence in $u$. We start with a useful lemma, and recall the general notation $\psi = \varphi*\varphi^-$ introduced in \eqref{e:psi}.
\begin{lemma}
\label{lemm7}
Let {\rm \cref{ass:A}} prevail, and let $e$ be a unit vector in $\R^d$. Then, the vector $\nabla^2f\cdot e$ is non-degenerate.
\end{lemma}
\begin{proof}
For a non-zero vector $v\in \R^d$,
\[\Var(\nabla^2f\cdot e)\cdot(v,v) = \int_{\R^d}(\xi\cdot e)^2(\xi\cdot v)^2\omega(\dd\xi).\]
This quantity never cancels as soon as $\omega$ is not supported on a union of two hyperplanes, which is not the case from our assumptions on the field $f$.
\end{proof}
\begin{lemma}
\label{lemm8}
Under {\rm \cref{ass:A}} and for $\eta$ small enough, there is a constant $C$ and a positive constant $\varepsilon$, such that $\|x\|\leq \eta$,
\[\rho_u(x)\leq \frac{C}{\|x\|^{d-2}}e^{-\frac{u^2}{2}(1+\varepsilon)}.\]
In particular, for $\varphi$ as in Remark \ref{r:convention}
\[\int_{B(0,\eta)}\psi(x)\rho_u(x)\dd x\leq C\eta^2\|\varphi\|_2^2e^{-\frac{u^2}{2}(1+\varepsilon)}.\]
\end{lemma}
\begin{proof}
We define the quantities
\[h(x) = \frac{f(0)+f(x)}{2},\quad g(x) = \frac{f(x)-f(0)}{\|x\|}\]
\[\quad H(x) = \frac{\nabla f(0)+\nabla f(x)}{2}, \quad G(x) = \frac{\nabla f(x)-\nabla f(0)}{\|x\|}\]
\[P(x) = \left(\nabla^2f_{e_x^\perp}(0),\frac{\partial_{e_x}^2f(0) - G(x)}{\|x\|}\right)\quand Q(x) = \left(\nabla^2f_{e_x^\perp}(x),\frac{\partial_{e_x}^2f(x) - G(x)}{\|x\|}\right)\]
Notice that all these quantities have a.s. defined limits as $\|x\|$ goes to zero and $x/\|x\|$ converges to a limit vector $e$. One has the equality
\begin{align*}
\rho_u(x) = \frac{1}{\|x\|^{d-2}}\frac{\E\left[|\det P(x)||\det Q(x)|\one_{\{h(x)-u\geq |g(x)|/2\}}\,\middle | \,H(x)=0,\;G(x)=0\right]}{\sqrt{\det \left[2\pi \Var(G(x),H(x))\right]}}.
\end{align*}
Notice that the Gaussian vector $(h(x),G(x))$ is independent from $(g(x),H(x))$. Their variances are given by
\[\Var(h(x),G(x)) = \begin{pmatrix}
\frac{1+r(x)}{2} & \frac{\nabla r(x)^T}{\|x\|} \\ 
\frac{\nabla r(x)^T}{\|x\|} & -2\frac{\nabla^2 r(0)-\nabla^2r(x)}{\|x\|^2}
\end{pmatrix},\; \Var(g(x),H(x)) = \begin{pmatrix}
\frac{1-r(x)}{\|x\|^2} & -\frac{\nabla r(x)^T}{\|x\|} \\ 
-\frac{\nabla r(x)}{\|x\|} & -\frac{\nabla^2 r(0)+\nabla^2r(x)}{2}
\end{pmatrix}\]
Gaussian regression implies that
\[h(x) = \underset{\tilde{h}(x)}{\underbrace{h(x) - \langle \Var(G(x))^{-1}\nabla r(x), G(x)\rangle}} + \langle \Var(G(x))^{-1}\nabla r(x), G(x)\rangle,\]
where both terms are independent. In particular the distribution of $h(x)$ conditionally to $G(x) = 0$ is $\tilde{h}(x)$. Moreover,
\[\Var(\tilde{h}(x)) = \Var(h(x)) - \nabla r(x)^T\Var(G(x))^{-1}\nabla r(x).\]
As $\|x\|$ goes to 0 and $x/\|x\|$ converges to $e$
\begin{align*}
\Var(\tilde{h}(x)) &= \frac{1+r(x)}{2} - \nabla r(x)^T\Var(G(x))^{-1}\nabla r(x)\\
&\simeq 1 - e^T\Var((\nabla^2 f\cdot e)^{-1})e \\
&<1,
\end{align*}
where the last inequality follows from the previous Lemma \ref{lemm7}. One can then choose positive constant $\eta,\delta$ such that
\[\forall \|x\|\leq \eta,\quad \Var(\tilde{h}(x))\leq 1-\delta\qsothat \PP(\tilde{h}(x)\geq u)\leq e^{-\frac{u^2}{2(1-\delta)}}.\]
One has
\begin{align*}
\rho_u(x) \leq \frac{e^{-\frac{u^2}{2(1-\delta)}}}{\|x\|^{d-2}}\frac{\E\left[|\det P(x)||\det Q(x)|\,\middle | \,H(x)=0,\;G(x)=0,\;\tilde{h}(x)\geq u\right]}{\det \left[2\pi \Cov(H(x),G(x))\right]}.
\end{align*}
Notice that the covariance of $(P(x),Q(x))$ conditionally to $H(x),G(x),\tilde{h}(x)$ is uniformly bounded in $x$, since they all have a definite limit as $\|x\|$ goes to $0$ and $x/\|x\|$ converges to a unit vector $e$. Applying Laplace method in Lemma \ref{lemm4} to this quantity, it follows that one can bound the expectation by a polynomial in $u$, yielding the bound
\[\rho_u(x)\leq P(u)\frac{e^{-\frac{u^2}{2(1-\delta)}}}{\|x\|^{d-2}}.\]
The first part of the lemma follows by choosing $\varepsilon$ accordingly. The second part of the lemma follows by integration of the first bound and using the inequality
\[\sup_{x\in \R^d} \psi(x)\leq \|\varphi\|_2^2,\]
yielding the desired conclusion.
\end{proof}
\subsubsection{Off-diagonal estimates in the Kac density}

The following result allows one to estimate the Kac density \eqref{e:rhou} for values of $x$ that are far from the origin.

\begin{lemma}
\label{lemm9}
Under {\rm \cref{ass:A}}, let $\eta$ be a positive constant. Then, there is a positive constant $\varepsilon$ such that, for $\|x\|\geq \eta$, one has
\[\rho_u(x)\leq C_\eta e^{-\frac{u^2}{2}(1+\varepsilon)}.\]
In particular, for $\varphi$ as in Remark \ref{r:convention},
\[\int_{B(0,A)\setminus B(0,\eta)}\psi(x)\rho_u(x)\dd x\leq CA^d\|\varphi\|_2^2e^{-\frac{u^2}{2}(1+\varepsilon)}\]
\end{lemma}
\begin{proof}
The idea of proof is very similar to the one of Lemma \ref{lemm8}. The non-degeneracy condition of the field $f$ implies that one can find $\delta>0$ such that for all $x$ with $\|x\|\geq \eta$, one has $\Cov(f(0),f(x)|\nabla f(0) = \nabla f(x)= 0)<1-\delta$. It follows that
\[\PP(f(0)\geq u, f(x)\geq u | \nabla f(0) = \nabla f(x)= 0) \leq \frac{C}{u^2}e^{-\frac{u^2}{1+\delta}}.\]
Applying Laplace method in Lemma \ref{lemm4} to $\rho_u$ we deduce that for some polynomial $P$, the following bound holds
\[\rho_u(x)\leq P(u)e^{-\frac{u^2}{1+\delta}}.\]
The first part of the lemma follows from an adequate choice of $\varepsilon$, and the second part by integration on $B(0,A)\setminus B(0,\eta)$.
\end{proof}

The following statement uses once again the notation \eqref{e:psi}.

\begin{lemma}
\label{lemm10}
Let the assumptions of Lemma \ref{lemm8} prevail. For all $\varepsilon$ positive, there is a positive constant $\eta$ and a constant $C$, such that if $\|\Omega\|\leq \eta$, then
\[\sup_{t\in[0,1]}\|\nabla^q\tilde{\rho}_u(t\Omega)\|\leq C e^{-u^2(1-\varepsilon)}.\]
In particular, there is a constant $A$ such that for $\|x\|\geq A$,
\[\int_{\R^d\setminus B(0,A)}\psi(x)\left|\rho_u(x)-\sum_{k=0}^{q-1}\frac{1}{k!}\Omega(x)^{\otimes k}(F_u^{(k)},F_u^{(k)})\right|\dd x\leq Ce^{-u^2(1-\varepsilon)}\int_{\R^d\setminus B(0,A)}\!\!\psi(x)\|\Omega(x)\|^q\dd x\]
\end{lemma}
\begin{proof}
We observe the quantity 
\[\nabla^q\tilde{\rho}_u(t\Omega) = \nabla^q\left(\Omega \rightarrow \E[F_u(V)F_u(W)]\right),\]
where $V,W$ are standard Gaussian vectors with covariance $\Omega$. Recall that $V$ can be decomposed as $V_1,V_2,V_3$, and similarly for $W$, which corresponds to the decomposition of $W(x) = (f(x),\sqrt{d}\nabla f(x), D(x))$. Taking derivatives with respect to $\Omega$ we still obtain a Gaussian integral with respect to the density of $(V,W)$. The successive derivatives yield powers of $V$ and $W$ with coefficients that depend continuously on the covariance matrix of $(V,W)$ in a neighborhood of the identity matrix, i.e., continuously on $\Omega$ in a neighborhood of $0$, say a ball $B(0,\eta)$. It follows from the expression of the Kac density that there is a polynomial $P_\eta$ such that
\begin{align*}
\|\nabla^q\tilde{\rho}_u(t\Omega)\| &\leq \E[P_\eta(V,W)F_u(V)F_u(W)]\\
&\leq \E[P_\eta(V,W)F_u(V)F_u(W)|V_1\geq u, W_1\geq u]\PP(V_1\geq u, W_1\geq u|V_2 = W_2 = 0).
\end{align*}
Now, by an application of the Laplace method in Lemma \ref{lemm4}, one finds that
\[\PP(V_1\geq u, W_1\geq u|V_2 = W_2 = 0)\leq \frac{C}{u^2}e^{-\frac{u^2}{1+\delta}},\]
where $\delta = \Cov(V_1,W_1|V_2=W_2=0)$. As for the conditional expectation, another application of the Laplace method yields that the expectation is bounded by a polynomial in $u$. Choosing eventually $\eta$ small enough so that $\delta$ is smaller than $\varepsilon$, one concludes the proof of the first part of the statement. The second part of the  statement follows from Lemma \ref{lemm7} and the fact that $\Omega(x)$ tends to zero as $\|x\|\rightarrow+\infty$, so that for some $A$ big enough and $\|x\|\geq A$, one has $\|\Omega(x)\|\leq \eta$.

\end{proof}
We now collect all the previous estimates to state the following general bound on the second factorial moment of $X_1(f,u,\varphi_R)$. In the next statement, \cref{ass:A} is in force and $\varphi$ is as in Remark \eqref{r:convention}.
\begin{lemma}
\label{lemm11}
For all $\varepsilon$ positive, there are constants $A,C$ independent of $\varphi$ such that
\begin{align*}
\left|\Var(X_1(f,u,\varphi)) - \E[X_1(f,u,\varphi^2)] - \sum_{k=0}^{q-1} \Var(X_1^{(k)}(f,u,\varphi))\right|\\
\leq CA^d\|\varphi\|_2^2e^{-\frac{u^2}{2}(1+\varepsilon)} + Ce^{-u^2(1-\varepsilon)}\int_{\R^d\setminus B(0,A)}\!\!\psi(x)\|\Omega(x)\|^q\dd x,
\end{align*}
where the norm in the last summand is Frobenius, and we have used the notation \eqref{e:psi}.
\end{lemma}
\begin{proof}
We use the equality
\begin{align*}
\Var(X_1(f,u,\varphi)) - \E[X_1(f,u,\varphi^2)] - \sum_{k=0}^{q-1} \Var(X_1^{(k)}(f,u,\varphi))\\
= \int_{\R^d}\psi(x)\left(\rho_u(x)-\sum_{k=0}^{q-1}\frac{1}{k!}\Omega(x)^{\otimes k}(F_u^{(k)},F_u^{(k)})\right)\dd x
\end{align*}
We split the domain of integration in three subdomains : $B(0,\eta)$, $B(0,A)\setminus B(0,\eta)$, and $\R^d\setminus B(0,A)$. The constants $\eta,A$ are chosen accordingly to Lemmas \ref{lemm8} and \ref{lemm10}. On each of these domains, the bound follows from the second statements in Lemmas \ref{lemm8}, \ref{lemm9} and \ref{lemm10}, and the fact that Lemma \ref{lemm4} implies that $\|F_u^{(q)}\|$ is bounded by $Cu^{d+q-1}e^{-\frac{u^2}{2}}$.
\end{proof}

\subsection{End of the three remaining proofs}
We will now focus on the proofs of Theorem \ref{t:integrablefields}, Proposition \ref{p:squarecorrex} and Proposition \ref{lemm14}. These three results deal with the limit in distribution of $X_1(f,u,\varphi_R)$, where the level $u$ is possibly dependent on $R$. Observe that, by virtue of Lemma \ref{lemm5}, one has that, for every $\varphi$ as in Remark \ref{r:convention}.
\[\E[X_1(f,u,\varphi_R)] = \frac{R^d \int_{\R^d} \varphi(x) \, \dd x}{d^{d/2}(2\pi)^{\frac{d+1}{2}}}u^{d-1}e^{-\frac{u^2}{2}}\left(1+O\left(\frac{1}{u}\right)\right).\]

\subsubsection{Asymptotic results for fields with integrable covariance functions: Proof of Theorem \ref{t:integrablefields} }\label{ss:proofintegrablefields}

The proof is a consequence of the next two lemmas.

\begin{lemma}
\label{lemm12b}
If $\nabla ^j r\in L^1(\R^d)$ for $j=0,\ldots, 4$, then as $R$ grow to infinity and $u$ a fixed level, there exists a positive constant $C_u$ such that
\[\Var(X_1(f,u,\varphi_R)) \simeq C_uR^d,\]
and the observable $X_1(f,u,\cdot)$ converges in the weak sense towards a Gaussian white noise.
\end{lemma}
\begin{proof}
The proof is almost entirely contained in \cite{gass2025spectralcriteriaasymptoticslocal}, with the slight modification that we consider critical points above a level $u$ (instead of all critical points). The argument for the finiteness and positivity of $C$ through the analysis of the second chaos is in all points similar, as well as the Gaussian asymptotics via the fourth moment Theorem.
\end{proof}
\begin{lemma}
\label{lemm12}
If $\nabla ^j r\in L^1(\R^d)$ for $j=0\ldots 4$, then as $u,R$ grow to infinity,
\[\Var(X_1(f,u,\varphi_R)) \simeq \E[X_1(f,u,\varphi_R^2)]\]
\end{lemma}
\begin{proof}
According to Lemma \ref{lemm11} applied with $q=1$,
\begin{align*}
\left|\E[X_1(f,u,\varphi_R)^{[2]}] - \E[X_1(f,u,\varphi_R)]^2\right| &\leq CA^dR^d\|\varphi\|_2^2 e^{-\frac{u^2}{2}(1+\varepsilon)} \\
&+ CR^de^{-u^2(1-\varepsilon)}\int_{\R^d\setminus B(0,A)}\psi_R(x)\|\Omega(x)\|\dd x
\end{align*}
Since $\psi_R$ converges to $\|\varphi\|_2^2$ and $\Omega\in L^1(\R^d)$, the last integral is bounded by $C\|\varphi\|_2^2$. As $u,R$ grow to infinity, the two terms are negligible compared to $\E[X_1(f,u,\varphi_R^2)]$ and the conclusion follows.
\end{proof}
\subsubsection{Proof of Proposition \ref{p:squarecorrex}}
\label{ss:proofsquarecorrex}

The proof is a consequence of Theorem \ref{t:squarecorr}. We first prove the theorem for the observable $X_1(B,u,\cdot)$. Recall the expression of the chaotic decomposition of $X_1(B,u,\varphi_R)$ given by Lemma \ref{lemm2c}. According to Lemma \ref{lemm13a}, the first chaotic projection satisfies
\[\Var(X_1^{(1)}(B,u,\varphi_R)) = o(R^{d+1})\]
as $R$ grows to infinity. Similarly, by Lemma \ref{lemm11} with $q=2$ and the asymptotics given by Lemma \ref{lemm13a}
\[\Var\left(\sum_{k=3}^{+\infty}X_1^{(k)}(B,u,\varphi_R)\right) \leq CR^{d+\frac{1}{2}} = o(R^{d+1})\]
as $R$ grows to infinity. Now Lemma \ref{lemm2} yields the existence of constant $c_0^{(u)}, c_1^{(u)}, c_2^{(u)}$ such that
\[F_u^{(2)}(W(x)) = c_0^{(u)} (f(x)^2-1) + c_1^{(u)}(\|\nabla f(x)\|^2-1) + c_2^{(u)}\left(\|\nabla^2 f(x)\|^2 - 1\right),\]
which implies that
\[X_1(B,u,\varphi_R) = \sum_{j=0}^2c_j^{(u)}\int_{\R^d} \varphi_R(x) \left(\|\nabla^j B(x)\|^2 - 1\right)\dd x + T_R,\]
where $T_R$ is the sum of all chaotic components of $X_1(B,u,\varphi_R)$ except for the second, and whose variance is $o(R^{d+1})$ as $R$ grows to infinity. Note that the constants $c_u^{(i)}$ are analytic in $u$, and according to the asymptotics established in Lemma \ref{lemm6}, one has as $u$ goes to infinity the asymptotics
\[\sum_{j=0}^2 c_j^{(u)}\sim\gamma_u.\]
It follows that this sum is not constant in $u$, and can then cancel only for a finite number of exceptional values of $u$. Outside of this exceptional subset of $\R$, one can directly apply Theorem \ref{t:squarecorr} to establish that the critical points above $u$ belong the universality class.\\

We quickly sketch the proofs for the three other observables $X_\ell(B,u,\cdot)$, $\ell=2,3,4$, which follow the same lines, up to minor modifications that we detail here. For the point processes of maxima above $u$ given by $X_2(B,u,\cdot)$, Kac--Rice formula implies that the function $F_u$ must be modified as 
\[\widetilde{F}_u(W(x)) = \delta_0(\nabla f(x))|\det \nabla^2f(x)|\one_{\{-d\}}({\rm sign}\, \nabla^2 f(x))\one_{[u,+\infty[}(f(x)),\]
where $\rm{sign}$ denotes the signature of a symmetric matrix. The proof then follows the exact same lines for the point process of critical points above $u$. For the measure of the $u$-level given by the observable $X_3(B,u,\cdot)$, Kac--Rice formula implies that the function $F_u$ must be modified as 
\[\widetilde{F}_u(f(x), \nabla f(x)) = \delta_u(f(x))\|\nabla f(x)\|.\]
The rest of the proof follows the same lines: one has the domination of the second chaotic projection, and Theorem \ref{t:squarecorr} can be applied. For the volume of the excursion above $u$ given by the observable $X_4(B,u,\cdot)$, the function $F_u$ must be modified as 
\[\widetilde{F}_u(f(x)) = \one_{[u,+\infty[}(f(x)),\]
and the conclusion again follows from the domination of the second chaotic projection and Theorem \ref{t:squarecorr}.

\subsubsection{Proof of Proposition \ref{lemm14}}\label{ss:proofvarexpansion}
Gathering Lemma \ref{lemm6} and \ref{lemm13a}, the asymptotics of the variance of the second chaotic projection is given by
\[\Var(X_1^{(2)}(B,u,\varphi_R)) = \gamma_u^2K_d^2R^{d+1}\|\varphi\|_L^2\left(1+O\left(\frac{1}{u}\right)\right) \simeq  C_dR^{d+1}u^{2d+2}e^{-u^2}\|\varphi\|_L^2.\]
From Lemma \ref{lemm13a} and the bound on $F_u^{(1)}$ obtained by Lemma \ref{lemm4}, the following  bound follows
\[ \Var(X_1^{(1)}(B,u,\varphi_R)) = u^d e^{-u^2}\left(1+O\left(\frac{1}{u}\right)\right)o(R^d).\]
Using Lemma \ref{lemm11} with $q=3$, and again the asymptotics of Lemma \ref{lemm13a} we get
\begin{align*}
&\left|\Var(X_1(B,u,\varphi_R)) - \E[X_1(B,u,\varphi_R^2)] - \sum_{k=0}^{2} \Var(X_1^{(k)}(B,u,\varphi_R))\right|\\
&\leq CR^dA^d\|\varphi\|_2^2e^{-\frac{u^2}{2}(1+\varepsilon)} + CR^de^{-u^2(1-\varepsilon)}\int_{\R^d\setminus B(0,A)}\!\!\psi_R(x)\|\Omega(x)\|^3\dd x\\
&\leq CR^d\|\varphi\|_2^2e^{-\frac{u^2}{2}(1+\varepsilon)} + CR^de^{-u^2(1-\varepsilon)}\sqrt{R}.
\end{align*}
It follows that, independently of the regime of $u,R$ the first chaos and the tail are negligible with respect to $\E[X_1(B,u,\varphi_R^2)]+ \Var(X_1^{(2)}(B,u,\varphi_R))$, from which the conclusion follows.

\section{Fluctuations of Radon-Fourier coefficients}\label{s:linespace}
\subsection{The affine Grassmannian}\label{ss:affinegrass}
\begin{definition}\label{d:affgrass}
    We denote by $\AG_d$ the set of all affine lines in $\R^d$ and call it the \embf{affine Grassmannian of lines}. Given a unit vector $u\in S^{d-1}$ and $r\in \R^d$, such that $\scal{r,u}=0$ we identify the line
    \be
\ell:=\ell_{r,u}:=r+ \R u\in \AG_d, \
    \ee
    We set the \embf{origin} of such line $\ell$ to be the point of minimal norm, which we denote as
    \be 
\origin_\ell:=\mathrm{arg}\min_{x\in \ell}|x|=r
    \ee
    We denote as $\TAG_d$ the set of pairs $(\ell,v)\in \AG_d\times \R^d$, such that $v\in \ell-\origin_\ell$. Given $(r,u)$ and $\ell$ as above, for any $\eta\in \R$, we define
    \be 
(\ell,v)_{r,u,\eta}:=\tyu r+\R u, \eta u\uyt \in \TAG_d
    \ee
    We denote by $TS^{d-1}$ the space of pairs $(r,u)$ defined as above: such that $u\in S^{d-1}$ and $r\in u^\perp$.
    \footnote{$\TAG_d$ is the tautological bundle over $\AG_d$.}
\end{definition}
The space $TS^{d-1}$ is indeed the tangent bundle of the sphere, hence, it is a smooth  manifold of dimension $2d-2$.

Note that $(r,u)$, $(r,-u)$ represent the same element $\ell_{r,u}\in \AG_d$. Moreover, $(r,u,\eta)$, $(r,-u,-\eta)$ represent the same element in  $\TAG_d$.

The next two lemmas gather the geometric description of the spaces $\AG_d$ and $\TAG_d$. Their proof is straightforward differential geometric book-keeping.
\begin{lemma}[Differentiable structure]
$\TAG_d$, $\AG_d$ are smooth manifolds. There is a commutative diagram of smooth mappings:
\be 
\begin{array}{ccc}
\RUSR & \xrightarrow{(\ell,v)_{\cdot}} & \TAG_d \\
\downarrow \text{$\cancel\eta$} &  & \downarrow \text{$\cancel v$} \\
\RUS & \xrightarrow{\ell_{\cdot}} & \AG_d
\end{array},
\ee
where the horizontal arrows are smooth double coverings, with respect to the relations
\be\label{eq:symm} 
(r,u)\sim (r,-u), \quad (r,u,\eta)\sim (r,-u,-\eta).
\ee
In particular, the mapping $\psi_\AG\colon (\ell)_{r,u}\mapsto \psi(r,u)$ defines a one-to-one correspondence $\psi_\AG\leftrightarrow \psi$ (resp. $\psi_\TAG\leftrightarrow \psi$) between functions on $\AG_d$ (resp. $\TAG_d$) and functions on $\RUS$ (resp. $\RUSR$) that are invariant under the symmetry imposed by \cref{eq:symm}. Such a correspondence preserves measurability, smoothness and Sobolev regularity.\footnote{The lemma entails that $\AG_d$ is diffeomorphic to the tautological vector bundle $\ell\mapsto \ell^\perp$ over the real projective space $\ell\in \mathbb{RP}^{d-1}$. It should be noted that this is not the tangent bundle of $\mathbb{RP}^{d-1}$.} 
\end{lemma}
\begin{definition}\label{def:sch_TAG}{\rm
The real Schwartz space $\Sch(\AG_d)$ is the set of functions $\psi_\AG$ such that
\be 
\psi\in \Sch_{\AG}(\RUS):=\kop \psi \in \Sch(\RUS): \psi(r,-u)=\psi (r,u), \ \forall (r,u)\in \TS\pok,
\ee 
where $ \Sch(\RUS)$ is defined as in \cref{sec:tempeschwa}.
The tempered distributions on $\AG_d$ are the elements of the dual space $\Sch'(\AG_d)$. We define analogously the spaces $\Sch(\TAG_d)$, $\Sch_\TAG(\RUSR)$, $\Sch'(\TAG_d)$ and denote all the complexifications with the subscript $\C$.

\smallskip

\noindent\textbf{Notation:} In what follows, we will indiscriminately identify $\psi_\TAG$ and $\psi$ and write
\be 
\psi(r+u\R,u\eta)=:\psi(r,u,\eta), \quad \forall (r,u,\eta)\in \TSR, \ \forall \psi\colon \TAG_d\to \C.
\ee
Moreover, for $g\in \Sch'_\C(\TAG_d)$ and $\psi\in \Sch_\C(\TAG_d)$, we will write:
\be\label{eq:pairingonA}
 \langle g, \psi \rangle=: \int_{\TAG_d}g(\dd (\ell,v))\psi(\ell,v)=: \int_{\TSR } g(\dd (r,u,\eta))\psi(r,u,\eta).
\ee}
\end{definition}

\begin{lemma}[Invariant Measure]\label{lem:invame}
There exists a unique measure $\minv$ on $\AG_d$ such that for any $g\in L^1(\AG_d)$ the next formula holds:
\be\label{eq:nuint} 
\int_{\AG_d}{g}(\ell)\minv(\dd \ell)=\frac{1}{2\pi s_{d-2}}\int_{S^{d-1}}\int_{u^\perp}g(r,u) \dd r\dd u.
\ee
With such a choice, for any isometry $\iota$ of $\R^d$, the map $\ell\mapsto \iota(\ell)$ preserves $\minv$.
The analogous statement holds for any $g\in L^1(\TAG_d)$. 
\end{lemma}
\begin{remark}
The symbolic notation introduced in \cref{eq:pairingonA}, reduces to that of \cref{eq:nuint} in case the distributions $g$ is represented by a function $g$, that is, $g(\dd (\ell,v))=g(\ell,v)\minv(\dd \ell)\dd v$ and $g(\dd (r,u,\eta))=g(r,u,\eta)\dd r\dd u\dd \eta$. The symbol $\dd r\dd u$, in this context, denotes the integral in $r\in u^\perp$ first and then in $u\in S^{d-1}$, as in the r.h.s. of \cref{eq:nuint}. It should be noted that this is different from the integral with respect to the Riemannian density induced by the inclusion $\TS\subset \R^{2d}$ (which we never use in this paper). The precise relation is shown in \cref{lem:fubigrass} below. 
\end{remark}
\begin{remark}\label{rem:alpha}
    The normalizing constant of $\minv$, $\a_d:=\frac{1}{2\pi s_{d-2}}$ is such that \cref{e:fgfcov2} and (equivalently) \cref{eq:radon_inversion} hold without constants. 
\end{remark}
\begin{lemma}[Grassmannian-Fubini]\label{lem:fubigrass}
For any $\psi\in \Sch_\C(\AG_d)$, we have the following rule:
\begt 
\int_{S^{d-1}}\int_{u^\perp}\psi(r,u) \dd r\dd u= \int_{\R^d}\int_{S(r^\perp)}\psi(r,u) \frac{1}{|r|} \dd u\dd r
\\
=\int_{\kop (u,r)\in S^{d-1}\times \R^d : \scal{u,r}=0\pok} \psi(r,u) \frac{1}{\sqrt{1+|r|^2}} \Haus^{2d-2}(\dd (u,r)) .
\eegt
\end{lemma}
\begin{proof}
The coarea formula \cite[Theorem 3.2.12]{federer2014} (apply twice), gives
\begt 
\int_{S^{d-1}}\int_{u^\perp}\psi(r,u) \dd r\dd u= \int_{\R^d}\int_{S(r^\perp)}\psi(r,u) \frac{\Jac\tyu\tyu (u,r)\mapsto u \uyt|_{\scal{u,r}=0}\uyt}
{\Jac\tyu\tyu (u,r)\mapsto r \uyt|_{\scal{u,r}=0}\uyt}
\dd u\dd r
\\
=\int_{\kop (u,r)\in S^{d-1}\times \R^d : \scal{u,r}=0\pok} \psi(r,u) \Jac\tyu\tyu (u,r)\mapsto u \uyt|_{\scal{u,r}=0}\uyt \Haus^{2d-2}(\dd (u,r)).
\eegt
Now, let us observe that for any two linear map $(A,B)\colon \R^n\to \R^{a+b}$, one has that 
\be \label{eq:Jacobians}
\Jac(A,B)= \Jac(A)\Jac(B|_{\ker(A)}).
\ee
In particular, one can interchange the role of $A$ and $B$. Therefore, we can rewrite the above formula in terms of simpler Jacobians by applying \cref{eq:Jacobians} to $A,B,C$ being the differentials of the maps sending $(u,r)\in S^{d-1}\times \R^d$ to $u,r,\scal{u,r}$, respectively. 
\begt 
\int_{S^{d-1}}\int_{u^\perp}\psi(r,u) \dd r\dd u= \int_{\R^d}\int_{S(r^\perp)}\psi(r,u) \frac{\Jac\tyu r\mapsto \scal{u,r}\uyt\Jac\tyu  (u,r)\mapsto u\uyt}
{\Jac\tyu u\mapsto \scal{u,r}\uyt \Jac\tyu  (u,r)\mapsto r\uyt}
\dd u\dd r
\\
=\int_{\kop (u,r)\in S^{d-1}\times \R^d : \scal{u,r}=0\pok} \psi(r,u) \frac{\Jac\tyu r\mapsto \scal{u,r}\uyt\Jac\tyu  (u,r)\mapsto u\uyt}
{\Jac\tyu (u,r)\mapsto \scal{u,r}\uyt }\Haus^{2d-2}(\dd (u,r)) .
\eegt
All Jacobians in this formula are now easy to compute and give the desired result.
\end{proof}
\subsection{X-ray-Fourier and Radon transforms}\label{ss:xrayradon}
\begin{definition}
For $\f\in \Sch_\C(\R^d)$, we define its \embf{X-ray-Fourier transform} $\Xray \f \in \Sch_\C(\TAG_d)$ 
as
\be \label{e:xraydef}
\Xray \f(\ell,v):=\int_\ell \f(x)\eim{v,x} \dd x.
\ee
We extend such a definition to distributions as for the Fourier transform: given $f\in \Sch'_\C(\R^d)$, we define $\Xray f\in \Sch'_\C(\TAG_d)$ by setting 
\be 
\langle \Xray f, \psi \rangle :=\int_{S^{d-1}}\dd u\int_{\R^d} f(\dd x) \Fouvar{t}{\scal{x,u}}\tyu \psi(x+u\R,ut)\uyt.
\ee
for any test function $\psi\in \Sch_\C(\TAG_d)$.
\end{definition}
Concretely, the $\Xray$ transform is just the one-dimensional Fourier transform along the line. Precisely, the following identity holds, in a  distributional sense, for all $f$ as in the definition above.
    \be \label{eq:XwithFOU}
\Xray f(r+\R u,\eta u)=\Fouvar{t}{\eta}\tyu f(r+ut)\uyt.
    \ee
(see \cref{fourier}) In particular, the above identities hold pointwise if $f,g$ are of Schwartz class.

\begin{remark}
As a first rough heuristic, one can think that the concentration of $|\Xray f|^2$ near a pair $(\ell,v)$ corresponds to a strong oscillatory behavior of $f$ along the line $\ell$.
\end{remark}
\begin{definition}\label{d:xray}
    We call $\Rad=\Xray(\cdot,0)$ the \embf{X-ray (one-dimensional Radon) transform}, that is 
    \be 
\Rad \varphi (\ell)= \int_\ell \varphi(x)\dd x,
    \ee
For all $\varphi\in \Sch_\C(\R^d)$. The adjoint Radon transform is thus defined as the operator $\Rad^*\colon \Sch'_\C(\AG_d)\to \Sch'_\C(\R^d)$ such that $\langle \Rad^*g,\varphi\rangle =\langle g,\Rad\varphi\rangle $. We extend such a definition and keep the same notation for the operator $\Rad^*\colon \Sch'_\C(\TAG_d)\to \Sch'_\C(\R^d\times \R)$ acting only in the space variable:\footnote{In other words, the operator $\Rad^*$ is extended by identifying $\Rad^*\equiv (\Rad\otimes\Id_{\R})^*$.} \be 
\langle \Rad^* g, \Phi\rangle {:=}\int_{\TAG_d} g(d(r,u,\eta))\int_{r+u\R} \Phi(x,\eta)\dd x.
\ee 
\end{definition}

The above definition is well posed, since $\Rad \varphi \in \Sch_\C(\AG_d)$ for any $\varphi\in \Sch_\C(\R^d)$, see \cite{sigurdurHelgason}. Additionally, the Radon transform $\Rad$ extends to distributions $f\in \Sch'_\C(\R^d)$ with compact support as in \cite{sigurdurHelgason}. For $\psi\in \Sch_\C(\AG_d)\subset \Sch'_\C(\AG_d) $, one has 
\be 
\Rad^*\psi(x)=c_d\int_{x\in \ell} \psi(\ell) \dd \ell=c_d\int_{S^{d-1}} \psi(x-\scal{x,u}u,u) \sigma(\dd u),
\ee
for some dimensional constant $c_d>0$, where the integral is with respect to the canonical probability measure of the projective space $\mathbb{RP}^{d-1}\cong \{\ell\in \AG_d:x\in \ell\}$, see \cite[Section 6]{sigurdurHelgason}, or, equivalently, as a uniform integral over the sphere, as above.

\subsection{The slice and inversion formulas for $\Xray f$ }
The \embf{Fourier slice theorem}, a cornerstone formula of the Radon transform theory (see \cite{sigurdurHelgason}), takes the following form in our context. 
\begin{theorem}[Fourier slice for $\Xray$]\label{thm:slice}
Let $f\in \Sch'_\C(\R^d)\cap \mC^\infty(\R^d,\C)$,
\be 
\Xray f (r+u\R,u\eta)=\int_{\R^d}\hat{f}(\dd \xi)\ei{\xi,r}\delta(\eta-\scal{u,\xi}),
\ee
where the identity is meant in $\Sch'_\C(\TAG_d)$.
\end{theorem}
\begin{proof}
Recall the formalism established in \cref{sec:temperando}, in particular the part regarding the delta function \cref{sec:delta_formalism}.
    \bega
\Fouvar{t}{\eta}\tyu f(r+ut)\uyt
=\int_\R \dd t\int_{\R^d}\hat{f}(\dd \xi)\,\ei{r+ut,\xi}\,\eimR{t\eta} =\int_{\R^d}\hat{f}(\dd \xi)\,\ei{\xi,r}\underbrace{\int_\R \eiR{t(\scal{u,\xi}-\eta)}\,\dd t}_{=\,\delta(\scal{u,\xi}-\eta)}.
\eega
The exchange of integrals is justified since the resulting expression is evaluated against a Schwartz test function.

\end{proof}
\begin{corollary}[Fourier slice for test functions]\label{lem:fourier_slice_Rad}
For $\f\in\Sch_\C(\R^d)$ and $\rho\in u^\perp$:
\be\label{eq:fourier_slice_Rad}
\Fouvar{r}{\rho}^{u^\perp}(\Xray\f(r,u,\eta))=\hat\f(\rho+\eta u), \quad \text{and}\quad \Fouvar{r}{\rho}^{u^\perp}\tyu\Rad\varphi(r,u)\uyt=\hat\varphi(\rho).
\ee
\end{corollary}
\begin{proof}
The slice theorem \cref{thm:slice} gives the first expression after developing the delta, and Fourier inversion.
By the same formula, applied at $\eta=0$: $\Rad\varphi=\Xray\varphi(\cdot,\cdot,0)$. At $\eta=0$: $\hat\varphi(\rho+0\cdot u)=\hat\varphi(\rho)$.
\end{proof}

\begin{definition}\label{d:deltar}
We define $(-\Delta_r)^s$ as the \embf{fractional Laplacian} operator in the variable $r\in u^\perp$, acting on tempered distributions on $\TAG_d$, (see \cref{def:sch_TAG}). For $\psi\in \Sch_\C(\TAG_d)$ such that $(-\Delta_r)^s\psi (r,u,\eta):=\Fouvar{\rho}{r}^{-u^\perp}(|\rho|^{2s}\Fouvar{s}{\rho}^{u^\perp}(\psi(s,u,\eta)))$. For general distributions, it is defined as the self-adjoint extension, exactly as for $(-\Delta)^s$, see \cref{sec:laplac}.
\end{definition}

\begin{theorem}[Intertwining relations]
\be
\Rad (-\Delta)^s=(-\Delta_r)^s\Rad, \quad \Rad^* (-\Delta_r)^s=(-\Delta)^s\Rad^*,
\ee
For any $s\in \R$. 
\end{theorem}
\begin{proof}
See \cite[Lemma 5.3]{sigurdurHelgason}.
\end{proof}

\begin{theorem}[Inversion]\label{thm:xray_fourier_inv}
Let $f,\hat{f}\in\Sch'_\C(\R^d)\cap\mC^\infty(\R^d,\C)$. The following identity holds:
\be\label{eq:xf_inv}
\tyu \Rad^*(-\Delta_r)^{1/2}\Xray f\uyt \tyu x,\eta\uyt
=\Fouvar{\xi}{x}^{-1}\tyu \frac{1}{s_{d-2}} \int_{S(\xi^\perp)}\hat f(\xi+\eta u) \dd u\uyt,
\ee
where $S(\xi^\perp):=S^{d-2}\cap\xi^\perp$ is the unit $(d-2)$-sphere perpendicular to $\xi$.
\end{theorem}
\begin{proof}
Denote $\psi:=(-\Delta_r)^{1/2}\Rad \Phi$. Then, by the slice formula \cref{thm:slice},
\begt
\langle \Xray f, (-\Delta_r)^{1/2}\Rad \Phi \rangle = \frac{1}{2\pi s_{d-2}}\int_\R \dd \eta\int_{S^{d-1}}\dd u \ \hat{f}(r+\eta u)\Fouvar{\rho}{r}^{-u^\perp}\tyu \psi(u,\rho,\eta )\uyt
\\
= \frac{1}{2\pi s_{d-2}}\int_\R \dd \eta\int_{S^{d-1}}\dd u \int_{u^\perp} \dd r\ \hat{f}(r+\eta u)|r|\Fouvar{\rho}{r}^{-u^\perp}\tyu \Rad\Phi(u,\rho,\eta )\uyt
\\
=
\frac{1}{2\pi s_{d-2}}\int_\R \dd \eta\int_{\R^d}\dd r\int_{S(r^\perp)}\dd u\ \hat{f}(r+\eta u)\frac{\hat{\Phi}(-r,\eta)}{(2\pi)^{d-1}}
\\=
\frac{1}{ s_{d-2}}\int_\R \dd \eta\int_{\R^d}\dd r\int_{S(r^\perp)}\dd u\ \hat{f}(r+\eta u)\Fouvar{x}{r}^{-1}(\Phi(x,\eta));
\eegt
where in the last step, we exchanged the integrals with \cref{lem:fubigrass} and used \cref{eq:fourier_slice_Rad}.
\end{proof}

\begin{remark}[Recovery of $f$]\label{rem:radoninv}
At $\eta=0$, \eqref{eq:xf_inv} reduces to $f(x)$, consistent with $\Xray f(\cdot,\cdot,0)=\Rad f$ and the X-ray inversion formula (\cite[Theorem 6.2]{sigurdurHelgason}): \be \label{eq:radon_inversion}
\Rad^*\sqrt{-\Delta_r}\Rad f=f, \quad \text{or} \quad \Rad^*\Rad=(-\Delta)^{-\frac12}.
\ee 
Such a formula, however is valid only when $\Rad f$ is defined and the evaluation at $\eta=0$ makes sense: for $f$ compactly supported distribution, or $f$ a smooth function such that $f(x)=O(|x|^{-d-\e})$ for some $\e>0$, see \cite[Theorem 6.2]{sigurdurHelgason}.
\end{remark}

\subsection{The white noise on the line space}
Let $\Wnoise_{\AG_d}\randin \Sch'(\AG_d)$ be the white noise, defined as in \cref{sec:wn}, on the space $\AG_d =\TS/\sim$, with respect to the measure $\minv$ of \cref{lem:invame} that is,
\be \label{eq:AG_wnoisdef}
\E\kop \langle \Wnoise_{\AG_d}, \phi\rangle \langle \Wnoise_{\AG_d}, \psi\rangle \pok= \langle  \phi,\psi\rangle=\int_{\AG_d}\phi(\ell)\psi(\ell) \minv(\dd \ell), \quad \forall \phi,\psi\in \Sch(\AG_d).
\ee
The inversion theorem \cref{rem:radoninv}, yields the following relation between $\Wnoise_{\AG_d}$ and the fractional Gaussian field on $\R^d$, see \cref{ss:fgfs}. 
\begin{proposition} \label{p:spiegone}
The following have the same law:
\be \label{eq:FGFRW}
\mathrm{FGF}_{\frac12}\sim \mathcal{R}^* \Wnoise_{\AG_d}
\ee
In particular, Blaschke--Petkantschin formulae (see \cref{e:fgfcov2}) follow.
\end{proposition}
\begin{proof} Let $\f\in \Sch(\R^d)$, then. By \cref{eq:radon_inversion}, 
\be 
\Var\qwe \langle \mathcal{R}^* \Wnoise_{\AG_d},\f \rangle \ewq= \Var\qwe \langle  \Wnoise_{\AG_d},\Rad \f \rangle \ewq = \langle  \varphi,\Rad^*\Rad\varphi\rangle= \langle  \varphi,(-\Delta)^{-\frac12}\varphi\rangle.
\ee
This concludes the proof since $\FGF_{\frac{1}{2}}=(-\Delta)^{-\frac14}\Wnoise_{\R^d}$ is unique (up to law) Gaussian random tempered distribution, having the operator $(-\Delta)^{-\frac12}$ as covariance.
\end{proof}

\subsection{Stochastic integral} 
In this {part of the paper}, we consider smooth centered Gaussian random fields $f$ on $\R^d$ that are naturally described (see \cref{e:berrycov}) as the Fourier transform of a random tempered distribution $Z\randin \Sch'_\C(\R^d)$, through the expression:
\be \label{eq:defZinv}
f(x)=\hat{Z}(x)=\int_{\R^d} \eim{x, \xi} Z(\dd \xi)
\ee
where the \embf{spectral distribution} $Z$ ($=(2\pi)^{-d}\hat{f}(-\cdot)$ by Fourier inversion) is a centered complex Gaussian\footnote{Following \cite{JansonBook}, a complex random variable is a complex Gaussian if the real and imaginary part are jointly real Gaussian random variables; without any assumption on their covariance.} random tempered distribution, see \cref{sec:tempeschwa} for the definition of random tempered distributions via Bochner-Minlos theorem. In general, the above identity has to be understood in a distributional sense:
for any test function $\f\in \Sch(\R^d)$, we have
\be \label{eq:defZ}
\int_{\R^d}\f(x)f(\dd x)=\int_{\R^d}\hat{\f}(\xi)Z(\dd \xi)
\ee

In such a context, the field $f$ is real if $Z$ is Hermitian, that is: $\overline{Z}(\dd (-\xi))=Z(\dd \xi)$, and smooth if $Z$ has a sufficiently strong decay at infinity, in the sense of the next \cref{rem:pointwise_spectral}.

\begin{assumption}\label{ass:hyp_f}
From now on, in all statements of this section, we will make the following assumptions: $f$ is a real stationary smooth Gaussian random field on $\R^d$, with spectral measure $\spec$ and spectral distribution $Z$, i.e., $f=\hat{Z}$. We will additionally assume that $f\in \mC^\infty(\R^d)\cap \Sch'(\R^d)$ with probability one. 
\end{assumption}

The next remark implies that the requirements of \cref{ass:hyp_f} are automatically satisfied under the assumptions of Theorem \ref{t:ultralimitintro}

\begin{remark}[Pointwise evaluation of spectral integrals]\label{rem:pointwise_spectral}
We clarify when the distributional identity \eqref{eq:defZ} can be upgraded to the pointwise formula \eqref{eq:defZinv}. For distribution spaces, we take up the classical notation of \cite{SchwartzDistr}.
\begin{enumerate}[\rm(a)]
\item\textbf{(General criterion.)} A classical result of Schwartz \cite[Theorem~XV, \S VII]{SchwartzDistr} asserts that the Fourier transform $\hat Z$ of a tempered distribution $Z\in\Sch'_\C(\R^d)$ belongs to the space $\mathscr{O}_M(\R^d)=\mC^\infty(\R^d)\cap \Sch'(\R^d)$ of smooth functions with at most polynomial growth (together with all derivatives) if and only if $Z$ belongs to the space $\mathscr{O}'_C(\R^d)$ of \embf{rapidly decreasing distributions}, meaning that: $Z*\rho\in\Sch_\C(\R^d)$ for every $\rho\in\mC^\infty_c(\R^d)$.\footnote{Equivalently, a tempered distribution $Z$ is \embf{rapidly decreasing} if $(1+|\xi|^2)^{k/2}Z$ is a bounded distribution for every $k\geq 0$; see \cite[Section~VI.8]{SchwartzDistr}. The equivalence is explained in \cite[``Remarque'' after Theorem IX, \S VII]{SchwartzDistr}.}
When $Z\in\mathscr{O}'_C(\R^d)$, the pairing $\langle Z,\varphi\rangle$ extends from Schwartz test functions to a larger class, and in particular to $\xi\mapsto \eim{x,\xi}$. The pointwise formula $f(x)=\langle Z_\xi,\,\eim{x,\xi}\rangle$ then coincides with $\hat Z(x)\in\mathscr{O}_M(\R^d)$.
\item\textbf{(Compact support case.)} If $Z\in\mathscr{E}'(\R^d)$ (compactly supported), then $Z\in\mathscr{O}'_C(\R^d)$, and $\hat Z$ extends to an entire complex analytic function on $\C^d$ with polynomial growth on $\R^d$; this is the Paley-Wiener theorem, see \cite[Theorem XVI, \S VII]{SchwartzDistr} or \cite[Theorem~7.1.14]{HormanderALPDO}. This covers the \embf{Berry model}, where $Z(\omega)$ is supported on $S^{d-1}$ almost surely, as well as any model whose spectral measure $\spec$ has compact support.
\item\textbf{(Bargmann--Fock case.)} When $\spec(\dd \xi)=\rho(\xi)\,\dd \xi$ with $\sqrt{\rho}$ smooth and rapidly decreasing (e.g., $\rho(\xi)=c\,e^{-|\xi|^2/2\sigma}$), one can write $Z=\sqrt{\rho}\cdot\Wnoise_{\R^d}$, where $\Wnoise_{\R^d}\randin\Sch'(\R^d)$ is the standard real white noise from \cref{sec:wn}. For each $x\in\R^d$, the function $\xi\mapsto\sqrt{\rho(\xi)}\,\eim{x,\xi}$ belongs to $\Sch(\R^d)$, so $f(x)= \langle \Wnoise_{\R^d},\,\sqrt{\rho}\,\eim{x,\cdot}\rangle$ is a well-defined Gaussian random variable for every $x$, using only that $\Wnoise_{\R^d}(\omega)\in\Sch'(\R^d)$ a.s.\ (\cref{thm:bochnerminlos}).
\end{enumerate}
\end{remark}

Moreover, $f$ is real and stationary if and only if there exists a positive definite function $K$ such that $K(x):=\E\qwe f(y)f(y+x)\ewq$ for all $x,y\in \R^d$. By Bochner's theorem \cite{Bochner1932}  (see also \cite[Theorem 5.4.1]{AdlerTaylor}), this is equivalent to the existence of a positive even Borel measure $\spec$ on $\R^d$ such that $K=\hat{\spec}$. $\spec$ is called the \embf{spectral measure} of $f$. 
In terms of the spectral distribution $Z$, the stationarity condition is equivalent to: $\E[Z(\dd \xi)\overline{Z(\dd \xi')}]=\delta(\xi-\xi')\spec(\dd \xi)$; in other words, $Z$ is a \embf{complex Gaussian white noise with intensity $\spec$}: the family $\{\langle Z,\f\rangle : \f\in  \Sch_\C(\R^d)\}$ is a Hermitian complex Gaussian tempered distribution with covariance:
\be \label{e:isometry}
\E\qwe \langle Z,\f\rangle \overline{\langle Z,\psi\rangle}\ewq=\int_{\R^d}\f(\xi)\overline{\psi(\xi)} \spec(\dd \xi), \quad \E\qwe \langle Z,\f\rangle \langle Z,\psi\rangle
\ewq =\int_{\R^d}\f(\xi)\psi(-\xi) \spec(\dd \xi),
\ee
This can be easily verified via Fourier calculus rules, see \cite[Section V.3]{sigurdurHelgason}. 
The covariance form is clearly continuous on $\Sch_\C(\R^d)^2$, hence by \cref{thm:bochnerminlos}, these two identities characterize $Z$ as a Gaussian tempered distribution $Z\randin \Sch'_\C(\R^d)$. The second of the above identities follows from the first and the afore-mentioned Hermitian property of $Z$, reflecting the fact that the field is real. In terms of test functions such a condition reads: 
\be \label{e:herm}
\overline{\langle Z,\f \rangle}=\langle Z,\overline{\f} (-\cdot)\rangle.
\ee

In addition, $f$ is isotropic if and only if $\spec$ is radial, meaning that there exists a measure $\rho$ on $[0,+\infty)$, such that
\be 
\int \varphi(\xi)\spec(\dd \xi)=\int_0^\infty \rho(\dd t) \tyu \int_{tS^{d-1}}\f(\xi) \dd \xi \uyt,
\ee
for every $\f\in \Sch_\C(\R^d)$. 

Our main cases of interest are, in order of importance: the \embf{Berry} random field (see \cref{ex:berry} below), for which $\rho(\dd t)=c\delta(\lambda-t)\dd t$; the \embf{Bargmann-Fock} random field, for which $\rho(t)=ce^{-{t^2}/{2\sigma}}$ {(see Figure \ref{fig:brw-bf-grid})}; the \embf{Annulus} random field, for which $\rho(t)=c1_{[a,b]}$ {(see Figure \ref{fig:annulus})}; moreover, we will consider also the general non-isotropic case. The implied constants are specified elsewhere. In all these cases, the field $f$ satisfies \cref{ass:hyp_f}.

\begin{example}\label{ex:berry}
Berry's field $B=\hat{Z}$ is defined in \cref{e:berrycov}, as the stationary isotropic Gaussian random field with $\spec=\sigma$, the uniform probability measure on $S^{d-1}$. Since $\spec$ is supported on $S^{d-1}$, the spectral distribution $Z\randin \Sch'_\C(\R^d)$ is supported on $S^{d-1}$ as well and is a complex Gaussian white noise with intensity $\sigma$.
\end{example}

\begin{remark}\label{rem:isonormal}
For a continuous real stationary Gaussian field $f$, by standard extension arguments, {based e.g. on \eqref{e:isometry}}, the random variable $\langle Z, \f \rangle$ is defined for all complex-valued test functions $\f$ in $L^2(\spec)$. Such a family restricts to a real \embf{isonormal Gaussian process} $\Ito\colon \Hil \to L^2(\Omega;\R)$ over the real Hilbert subspace $\Hil=L^2_{\text{her}}(\spec)\subset L^2(\spec;\C)$ of Hermitian complex valued functions: $\overline{\f}=\f(-\cdot)$, in the sense of \cite[Section 2.1]{nourdinpeccatibook}. In this sense, \cref{eq:defZinv} is equivalent to $f(x)=\Ito(\eim{x,\cdot})$ and $Z|_{\Sch(\R^d)}=\Ito_\C$ is the $\C$-linear extension of $\Ito$ to $L^2(\spec;\C)=\Hil_\C:=\Hil\otimes \C$, restricted to Schwartz functions.
\end{remark}
\begin{remark}\label{rem:janson}
In the sense of \cite{JansonBook}, $Z$ can be equivalently described (up to the Fourier transform convention) in terms of the complex Gaussian stochastic measure $Z(A)=\langle Z, 1_A\rangle$ of \cite[Theorem 7.54]{JansonBook}. The corresponding stochastic integral, given by \cite[Theorem 7.52]{JansonBook}, is the operator $\Ito_\C$ described in \cref{rem:isonormal} above. This can be seen by comparing \cref{eq:defZinv} and \cref{eq:defZ} with \cite[Eq. (7.45) and Remark 7.55]{JansonBook}.
\end{remark}
\subsection{The Wick square $Z^{:2:}$ and $\wxrayf$}\label{ss:Z2}

Let $f$ be a real, smooth, stationary, centered Gaussian field on $\R^d$ with spectral measure $\spec$ and spectral distribution $Z=\Fou^{-1}( f)\randin\Sch'_\C(\R^d)$. By \cref{rem:isonormal}, $Z$ can be seen as the $\C$-linear extension of an isonormal Gaussian process $\Ito$ over $\Hil=L^2_{\rm her}(\spec)$, in the sense of \cite[Section~2.1]{Nourdin_Peccati_2012}. {For $n=1,2,...$}, we write $\Ito_n\colon L^2_\sym(\spec^n;\C)=\Hil_\C^{\odot n}\to L^2(\Omega;\C)$ for the associated {$n$th {\bf multiple Wiener--It\^o integral}} (this notion corresponds to \cite[Theorem 7.52]{JansonBook}, and to the complexification of \cite[Section 2.7]{Nourdin_Peccati_2012}). 

\begin{remark}
We adopt the convention according to which the domain of $\Ito_q$ is the whole space $L^2(\spec^q,\C)$, by writing $\Ito_q(h):=\Ito_q(h_\sym)$, where $h_\sym$ is the symmetrization of $h$, that is, the orthogonal projection of $h$ onto the subspace $L^2_\sym(\spec^q;\C)=\Hil_\C^{\odot q}\subset L^2(\spec^q;\C)$ of a.e. symmetric complex-valued functions on $(\R^d)^q$. The image of $\Ito_q$ coincides with the complexification of the $q$-th Wiener chaos associated to $f$ (i.e., to $Z$) and the mapping $\Ito_q$ satisfies: 
\be \label{eq:ItoIso}
\E[|\Ito_q(h)|^2]=q!\|h_\sym\|_{L^2(\spec^q;\C)}^2
\ee
by the isometry properties of multiple integrals, see \cite[Theorem 7.52]{JansonBook} or \cite[Theorem 2.7.5]{nourdinpeccatibook}. 
\end{remark}

\begin{definition}[Wick square]\label{def:Wicksquare}
For $h\in L^2(\spec^2;\C)$, the \embf{Wick square pairing} is the second-chaos random variable
\be\label{eq:Z2general}
\int_{(\R^d)^2} Z^{:2:}(\dd (\xi_1,\xi_2))\  h(\xi_1,\xi_2) = \langle Z^{:2:}, h\rangle := \Ito_2(h_\sym),
\ee
where $h_{\rm sym}(\xi_1,\xi_2):=\tfrac12(h(\xi_1,\xi_2)+h(\xi_2,\xi_1))$ denotes the symmetrization of $h$. In symbolic notation: $Z^{:2:}(\dd (\xi_1,\xi_2))=\ \wick{Z(\dd \xi_1)\,Z(\dd \xi_2)}$.
\end{definition}
Since {(e.g. by Jensen) $\|h\|^2_{L^2(\spec^2;\C)} \geq \|h_\sym\|^2_{L^2(\spec^2;\C)}$}, the random variable $\langle Z^{:2:},h\rangle$ is well-defined. On product test functions $\f\otimes\psi$ (with $\f,\psi\in L^2(\spec)$), the definition reduces to
\be\label{eq:Z2product}
\langle Z^{:2:},\f\otimes\psi\rangle
= \wick{\langle Z,\f\rangle\,\langle Z,\psi\rangle}
= \langle Z,\f\rangle\,\langle Z,\psi\rangle
  - \E\qwe\langle Z,\f\rangle\,\langle Z,\psi\rangle\ewq,
\ee
It should be noted that $Z^{:2:}$ is uniquely characterized by the above formula, given in terms of the \embf{Wick product}, see \cite[Theorem~3.9]{JansonBook}. 
\begin{lemma}\label{lem:Z2cov}
The covariance of $Z^{:2:}$ is:
\be\label{eq:Z2cov}
\E\kop\langle Z^{:2:},h_1\rangle
  \,\overline{\langle Z^{:2:},h_2\rangle}\pok
= 2\langle (h_1)_{\rm sym},(h_2)_{\rm sym}\rangle_{L^2(\spec^2;\C)}.
\ee
\end{lemma}
\begin{proof}
 Follows from \cref{eq:ItoIso}.
\end{proof}

\subsubsection{A.s.\ well-posedness of the Wick square distribution}\label{ss:mollif}
The Wick square $Z^{:2:}$ is a priori defined as a continuous linear map $L^2(\spec^2;\C)\to L^2(\Omega;\C)$. However, we can promote it to a genuine $\Sch'_\C(\R^{2d})$-valued random variable, after observing that the inclusion $\iota:\Sch_\C(\R^{2d})\hookrightarrow L^2(\spec^2;\C)$ is continuous. Then, by the isometry property of $\Ito_2$, the composition $Z^{:2:}=\Ito_2\circ \iota$
is a continuous linear map from
$\Sch_\C(\R^{2d})$ to $L^2(\Omega;\C)$.

\begin{proposition}[A.s.\ well-posedness of $Z^{:2:}$]\label{thm:Z2_RTD}
There exists a measurable $Z^{:2:}\colon\Omega\to\Sch'_\C(\R^{2d})$, unique up to a.s.\ equality, such that for every $h\in\Sch_\C(\R^{2d})$:
\be\label{eq:Z2_as}
\langle Z^{:2:}(\omega),h\rangle = \Ito_2(h_{\rm sym})(\omega), \qquad \text{for a.e. } \omega\in \Omega.
\ee
\end{proposition}

\begin{proof}
By the isometry property of $\Ito_2$, the map $h\mapsto\Ito_2(h_{\rm sym})$ is a continuous linear map $\Sch_\C(\R^{2d})\to L^2(\Omega;\C)$. By \cref{thm:bochnerminlos}, there exists a unique probability measure $\spec$ on $\Sch'_\C(\R^{2d})$ whose finite dimensional marginals match those of $\{\Ito_2(h_{\rm sym})\}_{h\in\Sch_\C(\R^{2d})}$. 
 
Let $\{h_j\}_{j\in \N}\subset\Sch_\C(\R^{2d})$ be a countable dense subset. Then, $\Omega\ni \omega\mapsto \Gamma(\omega):=(\Ito_2(h_{j, \sym}))_{j\in \N}\in \C^\N$ is a measurable map whose law is determined by a push-forward of $\spec$ under the bijection:
\be
e\colon \Sch_\C'(\R^{2d})\to E:=\kop (\langle T, h_j \rangle)_{j\in \N } \in \C^\N \colon T\in \Sch'_\C(\R^{2d})\pok \subset \C^\N,
\ee
Therefore $\PP\qwe \omega\in \Omega \colon \Gamma(\omega)\in E\ewq =\spec(\Sch_\C'(\R^{2d}))=1$. Since the finite-dimensional projections determine the $\Sch'$-topology, $e^{-1}$ is measurable (and $e$ too). It follows that setting $Z^{:2:}(\omega) :=e^{-1}\Gamma(\omega) $, defines a measurable mapping from a full measure subset $\Omega_0 \subset \Omega$ to $\Sch_\C'(\R^{2d})$.

Now, for arbitrary $h\in \Sch_\C'(\R^{2d})$ there is a converging subsequence $h_j\to h$, hence $\Ito_2(h_{j,\sym})\to \Ito_2(h_\sym)$ in $L^2(\Omega;\C)$, which implies almost sure convergence on a further subsequence. It follows that for $\omega$ in the intersection of $\Omega_0$ with such a.s. event one must have that $\Ito_2(h_\sym)(\omega)=Z^{:2:}(\omega)$.\footnote{$\Ito_2(\cdot)(\omega)$ does not need to be continuous, nor linear, in $h$ for fixed $\omega$. Indeed, here, the a.s. event depends on $h$.}
\end{proof}

\begin{remark}[Nuclearity]\label{rem:nuclearity}
$Z^{:2:}$ is \emph{not} Gaussian (it lives in the second chaos), so the Gaussian shortcut (\cref{rmk:gaussian_fdd_suffice}) does not apply directly, via the covariance formula \cref{eq:Z2cov}. \cref{thm:bochnerminlos} requires nuclearity of $\Sch_\C(\R^{2d})$. For instance, the same argument fails on $\Hil_\C^{\odot q}$, which is not nuclear (see \cite[Section~2.3]{FGFSurvey}) and indeed $\Ito_q$ cannot be seen as a random element of $\Hil_\C^{\odot q}$.
\end{remark}

\subsubsection{Mollification}\label{sec:molly}
For $\rho\in\mC^\infty_c(\R^d)$ with $\rho\ge 0$, $\int\rho=1$, set $\rho_\e:=\e^{-d}\rho(\cdot/\e)$ and $Z_\e:=Z*\rho_\e$. Each $Z_\e$ is a smooth Gaussian random function in $\Sch_\C(\R^d)$ (see \cref{rem:pointwise_spectral}(a)), so $Z_\e^{:2:}(\xi_1,\xi_2):=\,\,\wick{Z_\e(\xi_1)\,Z_\e(\xi_2)}$ is a smooth random function on $\R^{2d}$. In particular, $Z_\e^{:2:}$ is a genuine smooth (in the Schwartz class) random tempered distribution, defined on $(\Omega,\PP)$.

\begin{proposition}[Mollification convergence]\label{prop:mollif}
$Z_\e^{:2:} \to Z^{:2:}$ in $\Sch_\C'(\R^{2d})$ in probability.
\end{proposition}

\begin{proof}
By linearity and the fact that $\Sch_\C'(\R^{2d})$ is metrizable, the convergence in probability can be reduced to the convergence in law. By \cref{thm:levy_Sprime}, we only need to check the convergence for any fixed $h\in \Sch_\C(\R^{2d})$.
Setting $h_\e:=h*(\rho_\e\otimes\rho_\e)$, Fubini and the bilinearity of the Wick product (see \cite[Theorem 3.9]{JansonBook}) give $\langle Z_\e^{:2:},h\rangle=\langle Z^{:2:},h_\e\rangle$. Since $h_\e\to h$ in $\Sch_\C(\R^{2d})$, it follows: $\langle Z_\e^{:2:},h\rangle\to\langle Z^{:2:},h\rangle$ almost surely.
\end{proof}
\subsubsection{Well-posedness of $\wxrayf$}\label{ss:wick_xray}

\begin{lemma}[Continuity of $\psi\mapsto h_\psi$]\label{lem:hpsi_cont}
For $\psi\in\Sch_\C(\TAG_d)$, define $h_\psi\colon\R^d\times\R^d\to\C$ by
\bega\label{eq:kernel_h}
 h_\psi(\xi_1,\xi_2): = \a_d\int_{S((\xi_1+\xi_2)^{\perp})} \Fouvar{r}{\xi_1+\xi_2}^{u^\perp} \ 
\tyu \psi(r,u,\scal{\xi_2,u})\uyt \,\dd u
\eega
Then the map $\psi\mapsto h_\psi$ is continuous from $\Sch_\C(\TAG_d)$ to $L^2(\spec^2;\C)$. Moreover, 
\be \label{eq:kernel_h_sliced}
h_{\Rad \Phi}(\xi_1,\xi_2)= \a_d \int_{S((\xi_1+\xi_2)^\perp)}\hat\Phi(\xi_1+\xi_2,\scal{\xi_2,u}) \dd u
\ee
for any $\Phi\in \Sch_\C(\R^d\times \R)$. Here $\hat{\Phi}(\xi,\eta):=\Fouvar{x}{\xi}(\Phi(x,\eta))$.
\end{lemma}

\begin{proof}
Formula \eqref{eq:kernel_h} integrates a bounded function ($|\Phi_\psi|\le\|\Phi_\psi\|_\infty$) over the compact set $S((\xi_1+\xi_2)^\perp)$, where the mapping $\psi\mapsto\Phi_\psi$ is continuous. Hence $\|h_\psi\|_{L^2(\spec^2)}\le\|\Phi_\psi\|_\infty\cdot\spec(\R^{d})k_d<\infty$, for a constant $k_d>0$.

The second formula is an immediate consequence of \cref{thm:slice}.
\end{proof}

\begin{remark} $h_\psi\notin\Sch_\C(\R^{2d})$ in general. In coordinates $s=\tfrac{\xi_1+\xi_2}{2}$, $z=\tfrac{\xi_1-\xi_2}{2}$, when $z\parallel s$ one has $\scal{z,u}=0$ for all $u\in S(s^\perp)$, so $h_\psi$ does not decay in $|z|$. 
\end{remark}

\begin{theorem}[Well-posedness of $\wxrayf$]\label{thm:Wicked_wp}
Let $f$ be a real, smooth, stationary, centered Gaussian field with spectral measure $\spec$. 
Then there exists $\sqrt{-\Delta_r}\wxrayf\randin \Sch'_\C(\TAG_d)$, unique up to a.s.\ equality, such that for every $\psi\in\Sch_\C(\TAG_d)$:
\be\label{eq:wick_pairing}
\langle\sqrt{-\Delta_r}\wxrayf,\psi\rangle = \Ito_2(h_\psi) \qquad \text{a.s.}
\ee
Moreover, given smooth approximations $Z_\e=Z* \rho_\e$, $Z_\e^{:2:}$, as in \cref{sec:molly}, and $f_\e=\widehat{Z_\e}$, then $\Xray f_\e, \wick{|\Xray f_\e|^2}$ are random smooth functions on $\TAG_d$ satisfying:
\be\label{e:norman}
\wick{|\Xray f_\e|^2}
 (r,u,\eta)= 
\wick{\Xray f_\e(r,u,\eta)\overline{\Xray f_\e(r,u,\eta)}}, \quad \forall (r,u,\eta)\in \TAG_d
\ee
and $\sqrt{-\Delta_r}\wick{|\Xray f_\e|^2}\to\sqrt{-\Delta_r}\wxrayf$ in probability in $\Sch'_\C(\TAG_d)$ as $\e\to 0$.
\end{theorem}

\begin{proof}
The argument is the same as for \cref{thm:Z2_RTD}. By \cref{lem:hpsi_cont}, $\psi\mapsto\Ito_2(h_\psi)$ is a continuous linear map $\Sch_\C(\TAG_d)\to L^2(\Omega;\C)$. The proof of \cref{thm:Z2_RTD} (based on Bochner-Minlos \cref{thm:bochnerminlos} and the Polish structure of $\Sch'_\C(\TAG_d)$) gives the version $(\sqrt{-\Delta_r}\wxrayf)\colon\Omega\to\Sch'_\C(\TAG_d)$ satisfying \eqref{eq:wick_pairing}. For the convergence in probability: for each $\psi$, one has that $\langle(\sqrt{-\Delta_r}\wick{|\Xray f_\e|^2})-(\sqrt{-\Delta_r}\wxrayf),\psi\rangle\to 0$ in $L^2(\Omega)$ (since $(h_\psi)_\e\to h_\psi$ in $L^2(\spec^2;\C)$ and $\Ito_2$ is an isometry), hence in probability (convergence in law to a constant). Since $\Sch'_\C(\TAG_d)$ is metrizable, this gives convergence in probability in $\Sch'_\C(\TAG_d)$.

The only point that remains to show is the identity \eqref{e:norman}. A natural starting point is from the r.h.s. of the latter. Using the slice theorem and recalling that Fourier inversion gives $\hat{f}(\dd \xi)=(2\pi)^d Z(-\dd \xi)$ (\cref{thm:slice}), we deduce the identity (see \cref{eq:Xf_Wick2} for a more detailed explanation):
\begt 
\wick{\Xray f(r,u,\eta)\,\overline{\Xray f(r,u,\eta)}} 
\\
=\int_{\R^{2d}} Z^{:2:}(\dd (\xi_1,\xi_2))\,
\eim{r,\xi_1+\xi_2}\,\delta(\eta_1+\scal{u,\xi_1})\,
\delta(\eta_2+\scal{u,\xi_2}) \Big|_{\eta_1=-\eta_2=\eta}
\\
=\int_{\R^{2d}} Z^{:2:}(\dd (\xi_1,\xi_2))\,
\eim{r,\xi_1+\xi_2}\,
\delta(\scal{u,\xi_1+\xi_2})\delta(\eta+\scal{\xi_1,u}) .
\eegt
In case both $f$ and $Z$ are smooth fields, then $\Xray f$ is also smooth, by the slice theorem (\cref{thm:slice}), thus the above formula defines a smooth random field on $\TAG_d$. By developing the deltas and using that the Jacobian of $\scal{\xi_1+\xi_2,u}$ is $|\xi_1+\xi_2|$, we get
\begt 
\langle \wick{\Xray f(r,u,\eta)\,\overline{\Xray f(r,u,\eta)}} , \sqrt{-\Delta_r}\psi\rangle 
= 
\int_{\R^d\times \R^d} Z^{:2:}(\dd (\xi_1,\xi_2))\a_d\int_{S^{d-1}}\dd u\int_{u^\perp}\dd r \int_\R \dd \eta\ 
\\
\delta(\scal{\xi_1+\xi_2,u})\delta(\eta+\scal{\xi_1,u}) \eim{r,\xi_1+\xi_2}(\sqrt{-\Delta_r}\psi)(r,u,\eta) =
\\
=\a_d\int_{\R^d\times \R^d} \!\!\! Z^{:2:}(\dd (\xi_1,\xi_2))\int_{S((\xi_1+\xi_2)^{\perp})} \!\!\!\dd u\int_{u^\perp}\dd r \ 
|\xi_1+\xi_2|^{-1}\eim{r,\xi_1+\xi_2}(\sqrt{-\Delta_r}\psi)(r,u,-\scal{\xi_1,u})=
\\
=
\a_d\int_{\R^d\times \R^d} Z^{:2:}(\dd (\xi_1,\xi_2))\int_{S((\xi_1+\xi_2)^{\perp})} \dd u\  \Fouvar{r}{\xi_1+\xi_2}^{u^\perp} \ 
\tyu \psi(r,u,-\scal{\xi_1,u})\uyt
\\
=\langle Z^{:2:},h_\psi\rangle = \langle \sqrt{-\Delta_r}\wxrayf, \psi\rangle.
\eegt
In the last step, we used the intertwining property of Fourier transform.

\end{proof}

\begin{remark}
One can think that $\wxrayf$ is a random element of the dual space of $\sqrt{-\Delta_r}\Sch_\C(\TAG_d)$, of test functions which live in the range of the operator $\sqrt{-\Delta_r}$.
\end{remark}

\subsubsection{Interpretation: antidiagonal restriction}\label{ss:antidiag}

The object $\sqrt{-\Delta_r}\wxrayf$ admits a natural interpretation as a \embf{restricton to the antidiagonal} in frequency space, which clarifies the necessity of $\sqrt{-\Delta_r}$ in \cref{thm:Wicked_wp}.

For fixed $\ell=r+\R u\in\AG_d$, the X-ray--Fourier transform $\Xray f(\ell,\cdot)$ is the one-dimensional {Gaussian random distribution} in $\eta${, obtained as the Fourier transform of the $f(r+(\cdot)u)\randin \mC^\infty(\R)\cap\Sch(\R)$, a smooth real stationary Gaussian field on $\R$. Note that this statement holds pointwise for each $\ell$, since $f$ is assumed to be a smooth function (see \cref{ass:hyp_f})}. By the slice theorem (\cref{thm:slice}):
\be\label{eq:Xf_Z}
\Xray f(\ell,\eta)=\int_{\R^d}\eim{r,\xi}\,
\delta(\eta+\scal{u,\xi})\,Z(\dd \xi).
\ee
One can represent the collection of double integrals associated to such a family of one-dimensional Gaussian processes by writing:
\be\label{eq:Xf_Wick2}
\wick{\Xray f(\ell,\eta_1)\,\Xray f(\ell,\eta_2)}
=\int_{\R^{2d}} Z^{:2:}(\dd (\xi_1,\xi_2))\,
\eim{r,\xi_1+\xi_2}\,\delta(\eta_1+\scal{u,\xi_1})\,
\delta(\eta_2+\scal{u,\xi_2}),
\ee
which is a well-defined random tempered distribution on $\R^2_{\eta_1,\eta_2}$ (at fixed $\ell$). The squared modulus $\wick{|\Xray f(\ell,\eta)|^2}$ corresponds to the \embf{restriction to the anti-diagonal} $\eta_1=-\eta_2$ (since $\overline{\Xray f(\ell,\eta)}=\Xray f(\ell,-\eta)$ by Hermitian symmetry). However, restricting a distribution $T\in\Sch'_\C(\R^2)$ to a submanifold $\Sigma\subset\R^2$ is not automatically well-defined; it requires additional regularity assumptions on $T$.

\begin{remark}[On the role of $\sqrt{-\Delta_r}$ in \cref{thm:Wicked_wp}]\label{rmk:role_sqrtDelta}
The filter $\sqrt{-\Delta_r}$ is a \emph{sufficient}
regularization that makes the antidiagonal restriction well-defined as a random tempered distribution. This is the reason why \cref{thm:Wicked_wp} defines $\sqrt{-\Delta_r}\wxrayf$ rather than $\wxrayf$ itself.
\end{remark}

\begin{remark}[Mollification validates the construction]\label{rmk:mollif_validates}
\cref{thm:Wicked_wp} confirms that the formal construction agrees with the ``honest'' pointwise Wick square: for the mollified $Z_\e=Z*\rho_\e$ (a smooth Gaussian function), the object $\wick{|\Xray f_\e(\ell,\eta)|^2}\,\, =|\Xray f_\e(\ell,\eta)|^2-\E[|\Xray f_\e(\ell,\eta)|^2]$ is defined pointwise, and $\sqrt{-\Delta_r}\wick{|\Xray f_\e|^2}\to\sqrt{-\Delta_r}\wxrayf$ in probability in $\Sch'_\C(\TAG_d)$ as $\e\to 0$.
\end{remark}
\subsection{Scaling limit}
\subsubsection{Scaling}
Let us consider the rescaled field $f_R(\cdot):=f(R(\cdot))$. Then, we shall consider the following rescaling of $\Xray f$:
\be 
(\Xray_R f)(r,u,\eta):=\Xray f_R(r,u,R\eta)R=\Fouvar{t}{\eta}\tyu f(Rr+ut)\uyt,
\ee
where the latter identity is to be interpreted in a distributional sense, as elements of $\Sch'(\TAG_d)$. Then, for any $\psi\in \Sch(\R^d)$, we have that 
\be \label{eq:scaledXray}
\langle \Xray_R f, \psi \rangle
=
\langle \Xray f_R, \psi(\cdot,\cdot,R^{-1}\cdot) \rangle, 
\quad 
\langle \wxrayfR, \psi \rangle
=
\langle \wick{|\Xray f_R |^2}, \psi(\cdot,\cdot,R^{-1}\cdot) \rangle.
\ee
We apply such a rescaling to capture the fact that $\Xray f_R$ concentrates around $\eta\sim R$, hence $\Xray_R f$ concentrates near $\eta \sim 1$.
\subsubsection{Limit} 
In order to understand the limit as $R\to+\infty$, one should study the behavior of the measure $$ \spec^2\tyu d\tyu {z+R^{-1}s},{-z+R^{-1}s}\uyt\uyt.$$
This can be seen from \cref{lem:covaraince} (in the appendix), a key step in the proof of \cref{thm:ultralimit} below.
In other words, the problem reduces to analyzing the (Euclidean) blow-up of the measure $\spec^2$ along the anti-diagonal $\{(z,-z)\in \R^{d}\times \R^d\}$.\footnote{In fact, since $\spec$ is symmetric in our setting of interest ($f$ is real), this is equivalent as studying $\spec^2$ around the diagonal.} This is where we need to distinguish between the case where $\spec$ admits a density with respect to the Lebesgue measure, from the case where it is supported on lower-dimensional manifold. Indeed, the blow-up behaviour is fundamentally different in the two situations. {We observe that Point (1) in the theorem below} recovers the case of Berry's random field, as defined in \cref{e:berrycov}, by taking the spectral measure to be $\omega=\sigma$, i.e., $\omega$ is the probability measure with constant density $\rho=s_{d-1}^{-1}$ with respect to the volume measure of the sphere $S^{d-1}$. {As already observed, the forthcoming statement is a more general version of Theorem \ref{t:ultralimitintro}.}

\begin{theorem}\label{thm:ultralimit}
Let $\Phi \in \Sch(\R^d\times \R)$. Then, under {\rm \cref{ass:hyp_f}},
\begin{enumerate}
\item[\rm (1)]\label{itm:limsph} If $\spec(\dd \xi)=\rho(\xi)s_{d-1}\sigma(\dd \xi)$ is supported on $S^{d-1}$ and is absolutely continuous with respect to the uniform spherical probability measure $\sigma$, and assume $\rho$ is bounded and almost everywhere continuous. Then,
\begt 
\lim_{R\to +\infty}R^{d-1}\E\qwe |\langle \Rad^* \sqrt{-\Delta_r}\wxrayfR\ , \Phi\rangle|^2 \ewq =\E\qwe |\langle \Rad^*\tyu \rho\Wnoise_{\AG_d} \circ E_{d}\uyt,\Phi\rangle|^2\ewq
\\
=\frac{2\a_d\pi^{d-1}}{2^{d-1}}\int_{\AG_d}\dd (u,r)\tyu \rho(u) \int_{S^{d-2}}   (\Rad \Phi)(r,u,\theta_1)\dd \theta\uyt^2.
    \eegt
The Gaussian process $\rho\Wnoise_{\AG_d}\circ  E_{d}$ is the composition of: a White noise $\rho\Wnoise_{\AG_d}$ on $\AG_d$ with respect to the measure $\rho(u)^2\minv(\dd (u,r))$, and the deterministic operator: $$E_{d}(\psi)(r,u):=\sqrt{\frac{2\a_d\pi^{d-1}}{2^{d-1}}} \int_{S^{d-2}} \psi(r,u,\theta_1)
\dd \theta$$.
    \item[\rm (2)] \label{itm:limleb}
If $\spec(\dd \xi)=\rho(\xi)\dd \xi$ is absolutely continuous with respect to the Lebesgue measure, and assume $\rho$ is bounded and almost everywhere continuous. Then,
 \begt 
\lim_{R\to +\infty}R^d\E\qwe |\langle \Rad^*\sqrt{-\Delta_r}\wxrayfR\ , \Phi\rangle|^2 \ewq  
= \E\qwe |\langle  \rho\Wnoise_{\R^{2d}}\circ F_d,\Phi\rangle|^2\ewq
\\
=\frac{2\a_d^2}{2^{2d}}\int_{\R^d}\dd s \int_{\R^{d}}
\dd z\  \left|
 \rho(z)\int_{S^{d-2}}   \Fouvar{r}{s}\tyu \Phi(r,{|\Pi_{s^{\perp}}(z)|\theta_1})\uyt\
\dd \theta \right|^2.
    \eegt
The Gaussian process $\rho\Wnoise_{\R^{2d}}\circ F_d$ is the composition of: a White noise $\rho\Wnoise_{\R^{2d}}$ on $\R^{2d}$, with respect to the measure $\rho(z)^2\dd s\dd z$, 
and the deterministic operator $F_{d}(\Phi) (s,z):=\sqrt{\frac{2\a_d^2}{2^{2d}}}\int_{S^{d-2}}   \Fouvar{r}{s}\tyu \Phi(r,{|\Pi_{s^{\perp}}(z)|\theta_1})\uyt\
\dd \theta $.
\item[\rm (3)] Both limits at Points {\rm (1)} and {\rm (2)} correspond to limits in law in the sense of random tempered distributions in $\Sch'(\R^d\times \R)$: 
\begt\label{eq:XR-law}
R^{\frac{d_i}2} \Rad^* \sqrt{-\Delta_r}\wxrayfR
\lawR{\Sch'(\R^d\times \R)}
\begin{cases} \Rad^*\tyu \rho\Wnoise_{\AG_d} \circ E_{d}\uyt, \quad &\text{in case (1)},
\\
\rho\Wnoise_{\R^{2d}}\circ F_d, \quad &\text{in case (2)},
\end{cases}
\eegt
where $d_i$ depends on the case (i) (i=1,2) under consideration: $d_1=d-1$ and $d_2=d$. In particular, in both cases the limit exists and is Gaussian, with covariance given as above.
    \end{enumerate}
\end{theorem}

The proof of \cref{thm:ultralimit} above is deferred to \cref{sec:ultraproof}.

\subsubsection{White scars: recovering \cref{t:squarecorr} and \cref{t:integrablefields}(ii)}\label{ss:ricovero}
Let $A\in \kop \R^d,\AG_d\pok$. For a distribution $g\in \Sch'_\C(A\times \R)$, we denote as $\int_\R \dd \eta g\in \Sch'_\C(A)$, the restriction of $g$ to test functions $\Phi$ of the form $\Phi(a,\eta)=\varphi(a)$. Note that since the operators $\Rad^*$ and $\sqrt{-\Delta_r}$ act only on the first variable $a$, they both commute with $\int \dd \eta$, that is, $\Rad^*\sqrt{-\Delta_r}\int \dd \eta=\int \dd \eta \Rad^*\sqrt{-\Delta_r}$.

The operation $\int_\R \dd \eta$ is well-defined for $\sqrt{-\Delta_r}\wxrayfR$, although $\varphi(r,u)\cdot 1\notin\Sch(\TAG_d)$. The reason is that the kernel $h_{\varphi\cdot 1}$ defined by \eqref{eq:kernel_h} is uniformly bounded (since the Fourier transform in $r$ is Schwartz and the integration is over a compact sphere), hence it belongs to $L^2(\spec^2;\C)$. The pairing $\langle\sqrt{-\Delta_r}\wxrayfR,\varphi\cdot 1\rangle:=\Ito_2(h_{\varphi\cdot 1})$ extends the definition in \cref{thm:Wicked_wp} beyond Schwartz test functions.

\begin{proposition}\label{prop:Radoninv_Xrayf_to_fsquare}
The following identity holds almost surely in $\Sch'(\R^d)$:
\be 
\Rad^* \sqrt{-\Delta_r}\int \dd \eta\wxrayfR=\wick{f(R\cdot)^2},
\ee
Where $f$ is any real smooth and stationary Gaussian random field on $\R^d$.
\end{proposition}
\begin{proof}
Recall that $\Xray f$ is the Fourier transform of $f$ in the variable $\eta$, therefore, by Plancherel $\eta$, $\int \dd \eta\wxrayfR=\Rad \wick{f(R\cdot)^2}$. The conclusion now follows from the Radon inversion formula $\Rad^* \sqrt{-\Delta_r}\Rad = \Id$, see \cref{rem:radoninv}. This argument, however, requires a clarification, since $\Rad f^2$ is not defined unless $f$ has sufficient decay (see \cite{sigurdurHelgason}). We can bypass this issue with a mollifier.  
For the mollified field $f_\varepsilon = \widehat{Z_\varepsilon}$ 
of \cref{sec:molly}, both $f_\varepsilon$ and $:\!f_\e^2\!:$ 
are smooth functions of rapid decay (since $Z_\e$ has 
compactly supported spectral density), so the identity
\be
\Rad^* \sqrt{-\Delta_r}\int \dd \eta\,
\wick{|\Xray_R f_\varepsilon|^2} = \wick{f_\varepsilon(R\cdot)^2}
\ee
holds pointwise, by Plancherel in $\eta$ followed by 
the Radon inversion formula \cref{eq:radon_inversion} 
applied to the Schwartz function $:\!f_\varepsilon(R\cdot)^2\!:$. 
The conclusion follows by passing to the limit 
$\varepsilon\to 0$: the left-hand side converges in 
$\Sch'(\R^d)$ by \cref{thm:Wicked_wp} (mollification 
convergence), and the right-hand side converges because 
$:\!f_\varepsilon^2\!: \to :\!f^2\!:$ in $\Sch'(\R^d)$ 
by \cref{prop:mollif}.
\end{proof}
\begin{corollary}\label{cor:rado}
Let the setting of \cref{thm:ultralimit} prevail. Assume in addition that $\rho(\xi)=\rho(|\xi|)$ in both cases. The following limits hold in law in the sense of random tempered distributions in $\Sch'(\R^d)$: 
\begt
R^{\frac{d_i}2} 
\wick{f(R\cdot)^2}
\lawR{\Sch'(\R^d)}
\sqrt{\spec(\rho)}c_{i,d}
\begin{cases} \Rad^* \Wnoise_{\AG_d} = \FGF_{\frac12} , \quad &\text{in case (1)},
\\
\Wnoise_{\R^{d}}, \quad &\text{in case (2)},
\end{cases}
\eegt
where $d_i$ depends on the case (i) under consideration: $d_1=d-1$ and $d_2=d$, referring to \cref{thm:ultralimit}. The constants $c_{1,d}$ and $c_{2,d}$ can be deduced explicitly from \cref{thm:ultralimit} and from the constant in \cref{thm:xray_fourier_inv}.
\end{corollary}
\begin{proof}
It is sufficient to compute the covariance in (1) and (2) in the case when $\Phi$ does not depend on $\eta$.
\end{proof}
\begin{remark}\label{r:ed}
From \cref{thm:ultralimit} one can immediately deduce other more general variants of \cref{cor:rado}, including: the case of non-isotropic fields, i.e., more general $\rho$; as well as the analogous theorem for derivatives:  Let $P$ denote a real polynomial $P(x)=\sum_{\a\in \N^d}c_\a x^\a$ (the sum being finite) and define $f':=P(\nabla) f:=\sum_{\a\in \N^d}c_\a \de_\a f$. Then,  $f'$ satisfies the hypotheses of \cref{thm:ultralimit} with $\rho'(\xi)=P(- i\xi)\rho(\xi)$. Consider now a finite family $(P_k)_k$ of polynomials such that $\sum_k P_k(- i \xi)^2\rho(\xi)=q(|\xi|)^2$ and assume $\rho(\xi)=\rho(|\xi|)$; then, by summing over $k$, one can show that
\begt\label{eq:derivultra}
R^{\frac{d_i}2} 
\wick{\sum_{k}\tyu (P_k(\nabla)f)(R\cdot)\uyt ^2}
\lawR{\Sch'(\R^d)}
\sqrt{\spec\tyu \frac{q^2}{\rho}\uyt} c_{i,d}
\begin{cases} \Rad^* \Wnoise_{\AG_d} = \FGF_{\frac12} , \quad &\text{in case (1)},
\\
\Wnoise_{\R^{d}}, \quad &\text{in case (2)},
\end{cases}
\eegt
Such a description includes any finite linear combination of the random fields appearing in \cref{t:squarecorr}: indeed, there exist polynomials $P_{k,j} $ (homogeneous, of degree $j$) such that
\be \label{eq:derivatives}
\|\nabla^j f\|^2-\E\qwe \|\nabla^j f\|^2\ewq=\wick{\sum_{i}\tyu P_{k,j}(\nabla)f\uyt ^2},
\ee 
for $j\in \N$. The condition $\sum_{k}c_k\neq 0$ in \cref{t:squarecorr} corresponds to $q\neq 0$. The conclusion follows from the same proof as that of \cref{thm:ultralimit} applied with spectral density $\tilde{\rho}(\xi)=\sum_k |P_k(i\xi)|^2 \cdot \rho(\xi)=q(|\xi|)$, after an additional covariance computation.
\end{remark}

\subsubsection{Line space convergence}
\begin{remark}
    The operator $E_{d}(\psi)$ evoked in the statement of \cref{thm:ultralimit} shows how the frequency concentrates in the limit. It admits the following alternative formula:
    \be 
E_{d}(\psi)(r,u)=e_d\rho(u)\begin{cases} 
\int_{-1}^1 \dd \eta\ \psi(r,u,\eta) \tyu 1-\eta^2\uyt^{\frac{d-4}{2}}   
, \quad &\text{if $d\ge 3$,}
\\
\psi(r,u,1)+\psi(r,u,-1), \quad &\text{if $d=2$.}
\end{cases}
    \ee
    Here $e_d\in \R$ is a dimensional constant. We can observe that if $d=2,3$, the limit exhibits a singularity at $|\eta|=1$; for $d= 4$, the frequency is uniformly distributed in the interval $[-1,1]$; for $d\ge 5$ the function $\eta\mapsto (1-\eta^2)^{\frac{d-4}{2}}1_{[-1,1]}(\eta)$ is continuous and reach its maximum value at $\eta=0$.
\end{remark}
\begin{remark}\label{r:esempiodelsecolo}
One can see the operator $E_d$ appear naturally, when computing the X-ray-Fourier transform of an eigenfunction of $\Delta$ of the form $h=\widehat{s\sigma}_{d}$, where $\sigma_d$ denotes the uniform measure on $S^{d-1}$ and $s$ a smooth symmetric function on $S^{d-1}$, as in the setting of \cref{lemm15a}. Let $\ell=r+\R u\in \AG_d$ be a line, where $r\in u^\perp$, $r\neq 0$ and $u\in S^{d-1}$. The slice theorem \cref{thm:slice} gives:
\begt
\Xray_R h(r,u,\eta)=\Fouvar{t}{\eta}\tyu h(Rr+tu)\uyt = c_d\int_{S^{d-1}}\sigma_d(\dd \xi) s(\xi) e^{i\scal{\xi,rR}} \delta(\eta-\scal{u,\xi})\\
=
c_d1_{[-1,1]}(\eta)(1-\eta^2)^{\frac{d-3}{2}}\int_{S(u^\perp)}\sigma_{d-1}(\dd v) s(v) e^{i\scal{\sqrt{1-\eta^2}v,rR}} 
\\
=c_d1_{[-1,1]}(\eta)(1-\eta^2)^{\frac{d-3}{2}}\widehat{s\sigma_{d-1}}(R\sqrt{1-\eta^2}r)
\\
\sim 
c_d(1-\eta^2)^{\frac{d-4}{4}}1_{[-1,1]}(\eta)(\|r\|R)^{-\frac{d-2}{2}}\cos\tyu R\sqrt{1-\eta^2}\|r\|-\pi\frac{d-2}{4}\uyt s\tyu \frac{r}{\|r\|}\uyt
\eegt
In the first step, we used that $\Jac\tyu \xi\mapsto \eta-\scal{u,\xi}\uyt=\sqrt{1-\eta^2}$; in the approximation step, as $R\to +\infty$, we used \cref{lemm15a}, in the space $u^\perp\cong \R^{d-1}$. Indeed, the formula depends on $u$, through the constraint $r\in u^\perp$. The constant $c_d$ changes from line to line. Therefore, 
\be \label{eq:appero}
|\Xray_R h(r,u,\eta)|^2 \sim  \begin{cases}
c_d(\|r\|R)^{-d+2}s\tyu \frac{r}{\|r\|}\uyt\cdot (1-\eta^2)^{\frac{d-4}{2}}1_{[-1,1]}(\eta), &\text{ if $d=2$;}
\\
c_2\log(R\|r\|)s\tyu \frac{r}{\|r\|}\uyt\cdot \delta(|\eta|-1) , &\text{ if $d\ge 3$;}
\end{cases}
\ee
as $R\to +\infty$, in the sense of $\Sch'(\R)$, as a distribution in $\eta$, for any fixed line $r,u$. Here, we used that, if $d\ge 3$ then $\cos^2(R\sqrt{1-\eta^2}\|r\|)\to 1/2$ in a distributional sense, by Riemann-Lebesgue Lemma. In dimension $d=2$, such a limit does not hold anymore because the function $(1-\eta^2)^{-1}$ is not integrable near $\eta=\pm 1$ (since $(d-4)/2 >-1$ if and only if $d>2$). Such a non-integrable singularity at $\pm 1$ is responsible for the localization, happening in the two-dimensional case. The delta limit in \cref{eq:appero}, easily follows from a direct analysis of the integral $\int_{-1}^1 {\cos^2(R\|r\| \sqrt{1-\eta^2})}{(1-\eta^2)}^{-1} \tau(\eta) \dd \eta$, for an arbitrary Schwartz test function $\tau$, via the change of variable $t=R\|r\| \sqrt{1-\eta^2}$. We notice that the product form of the quantity appearing on the right-hand side of \cref{eq:appero}, separating the roles of the variables $r$ and $\eta$, is consistent with the factorized structure of the noise emerging in \cref{eq:XR-law}.
\end{remark}

\begin{remark}\label{prop:dense_TAG}
The convergence in law at \cref{eq:XR-law} descends from a convergence in $\Sch'(\TAG_d)$:
\be
R^{\frac{d_i}{2}}\sqrt{-\Delta_r}\wxrayfR\;\xrightarrow{\;\rm law\;}\;\Xi \qquad \text{in }\Sch'_\C(\TAG_d),
\ee
where, in case i=(1), $\Xi\randin\Sch'_\C(\TAG_d)$ is the centered Gaussian with covariance
\be
\E[|\Xi(\psi)|^2]=\frac{2}{2^{d-1}}\int_{S^{d-1}}\dd v\int_{v^\perp}\dd s\;\rho(v)^2\left|h_\psi(s+v,s-v)\right|^2.
\ee
The Gaussianity and tightness follow from the same Schur test argument as in the proof of \cref{thm:ultralimit}(3), which works for any $\psi\in\Sch_\C(\TAG_d)$ (not only $\psi=\Rad\Phi$), see \cref{rem:end_of_tightgaus}. On the subclass $\psi=\Rad\Phi$, the covariance simplifies via \cref{lem:planchelike} to the expression in \cref{thm:ultralimit}.
\end{remark}

The range of $\Rad:\Sch(\R^d\times\R)\to\Sch(\TAG_d)$ is \emph{not} dense in $\Sch(\TAG_d)$ (it is constrained by the Helgason--John range conditions, see \cite[Ch.~I, \S6]{sigurdurHelgason}). In particular, $\Rad^*$ is not injective on $\Sch'(\TAG_d)$, so the lift from $\Sch'(\R^d\times\R)$ to $\Sch'(\TAG_d)$ does not follow formally from the $\Sch'(\R^d\times\R)$ result. This issue is present in all dimension $d\ge 2$.

We stress that indeed the covariance of $\Xi$ has a more involved formula than that of \cref{thm:ultralimit}, for general test functions $\psi\in \Sch_\C(\TAG_d)$, for all $d\ge 3$. However, in the two-dimensional case it is no longer the case and we have the following strengthening of \cref{thm:ultralimit}.
\begin{theorem}[Two-dimensional case]\label{thm:ultralimit_2}
Let $\psi \in \Sch_\C(\TAG_2)$. Assume that $\spec(\dd \xi)=\rho(\xi)\sigma(\dd \xi)$ is supported on $S^{1}$ and is absolutely continuous with respect to the uniform circular measure $\sigma$, and assume $\rho$ is bounded and almost everywhere continuous. Then,
\begt \label{eq:var_limit_2}
\lim_{R\to +\infty}R\E\qwe |\langle  \sqrt{-\Delta_r}\wxrayfR\ , \psi\rangle|^2 \ewq =\E\qwe |\langle  \rho\Wnoise_{\AG_2} \delta(|\eta|-1), \psi\rangle|^2\ewq
\\
=\a_d\pi\int_{\AG_d}\dd (u,r)\tyu \rho(u)    (\psi(r,u,1)+\psi(r,u,-1)\uyt^2.
    \eegt
and the limit holds in law as random tempered distributions in $\Sch'(\TAG_2)$: 
\begt\label{eq:XR-law_2}
R^{\frac{1}2} \sqrt{-\Delta_r}\wxrayfR
\lawR{\Sch'(\TAG_2)}
 (\rho\Wnoise_{\AG_2})\delta(|\eta|-1),
\eegt
where $\rho \Wnoise_{\AG_2}$ is the white noise with respect to the measure $\rho(u)^2\minv(\dd (r,u))$ on $\AG_2$.
\end{theorem}
The proof of \cref{thm:ultralimit} above, is deferred to the end of \cref{sec:ultraproof}, after \cref{rem:end_of_tightgaus}.
\subsubsection{Other types of spectral measures}\label{ss:otherpsectra}
Following the steps of the proof of \cref{thm:ultralimit}, in particular from \cref{lem:covaraince}, it is not hard to show that a distributional limit, as the one stated in \cref{thm:ultralimit}, exists for a variety of more general spectral measures. The following two examples show how the full distributional limit of $\Rad^*\sqrt{-\Delta_r}\wxrayfR$ retains both the oscillatory content of $f$, and the spatial part, that is, the limit of $\wick{f(R\cdot)}$.
\begin{example} 
Consider a Gaussian field on $\R\times \R^{d-1}$, of the form $f(x,y)=\cos(\lambda x)\gamma_1(y)+\sin(\lambda x)\gamma_2(y)$, where $\lambda >0$ and $\gamma_1,\gamma_2$ are two independent copies of a stationary Gaussian field with spectral measure $\spec(\dd y)=\rho(|y|)dy$, e.g. the Annulus field, or Bargmann-Fock field. Starting from \cref{lem:covaraince} and proceeding analogously as in the rest of the proof of \cref{thm:ultralimit}, we find that the limit is of the form 
\begt 
\lim_{R\to +\infty}R^{d-1}\E\qwe |\langle \sqrt{-\Delta_r}\wxrayfR\ , \psi\rangle|^2 \ewq  
= \E\qwe |\langle  \rho\Wnoise_{\R^{2d-2}}\circ F_{\lambda},\psi\rangle|^2\ewq
\\
=
c\int_{\R^{d-1}}\dd s \int_{\R^{d-1}} \dd z \left| \rho(|z|) \int_{S^{d-2}} \psi\tyu (0,s),(1,0),\sqrt{\lambda^2+z_1^2} \theta_1\uyt \dd \theta \right|^2
,
\eegt
for an operator $\psi\mapsto F_{\lambda}(\psi)\in \Sch_\C(\R^{d-1})$, similar to that of \cref{thm:ultralimit}(2): $F_{\lambda}(\psi)(s,z)= \sqrt{c}\int_{S^{d-2}} \psi\tyu (0,s),(1,0),\sqrt{\lambda^2+z_1^2} \theta_1\uyt \dd \theta $; and where $\rho\Wnoise_{\R^{2d-2}}$ is a white noise on $\R^{2d-2}$, with respect to the measure $\rho(|z|)^2\dd s\dd z$, Note that, from $F_\lambda$ we can recover the fixed frequency $\lambda$ in the $x$ direction. Moreover, $\tyu \int \dd \eta \rho \Wnoise_{\R^{2d-2}}\circ F_{\lambda}\uyt(x,y)$ has the same law as a white noise $c\Wnoise_{\R^{d-1}}(y)$ in the variable $y\in \R^{d-1}$, constant in the $x$ direction. 
\end{example}
\begin{example}
A second noteworthy example is the case of a discrete spectral measure $\spec(\dd \xi)=\sum_{\lambda\in \Lambda}\delta(\xi-\lambda)\dd \xi$, with $\Lambda\subset \R^d$ being a finite symmetric subset. Such spectral measure corresponds to a field of the form $f(x)=\sum_{\lambda \in \Lambda}\gamma_\lambda \prod_{i=1}^d\cos(\lambda_ix_i)$, where $(\gamma_\lambda)_{\lambda \in \Lambda}$ are independent real Gaussian variables of same $\sigma$. Such a field enjoys clear oscillatory patterns.
From \cref{lem:covaraince}, we can easily compute the full limit covariance as
\begt 
\lim_{R\to +\infty}\E\qwe |\langle \sqrt{-\Delta_r}\wxrayfR\ , \psi\rangle|^2 \ewq  
= 
\E\qwe |\langle  \sum_{\lambda \in \Lambda}\gamma_{\lambda}\delta(\eta+\scal{\lambda,u}),\psi\rangle|^2\ewq
\\
=
\sum_{\lambda \in \Lambda}\left| \int_{\AG_d} \dd (u,r)\psi_\sym(r,u,\scal{\lambda,u})\right|^2
\eegt
The limit field is thus the Gaussian distribution $G=\sum_{\lambda \in \Lambda}\gamma_{\lambda}\delta(\eta+\scal{\lambda,u})\randin \Sch_\C'(\TAG_d)$, with spatial projection: $\int \dd \eta G =\gamma_\Lambda$ for some single Gaussian variable $\gamma_\Lambda\sim \m N(0, \mathrm{Card}(\Lambda))$. 
\end{example}

\subsection{Heuristic interpretation of \cref{thm:ultralimit}}\label{ss:interpretationofultralimit}
Let $c_R(t)$ be a function with the property that $\int_\R \dd \eta |\hat{c_R}(\cdot)|^2\tau(\eta)\to E_d(\tau)$ for any $\tau\in \Sch_\C(\R)$. The choice of $c_R$ is not unique. When $d\ge 3$, one can take $c_R=c$, where 
\be 
\hat{c}(\eta)=\tyu1-\eta^2\uyt^{\frac{d-4}{4}}1_{[-1,1]}(\eta),
\ee 
while if $d=2$, since $|\hat{c}(\eta)|^2=\delta(|\eta|-1)$ we need to approximate with a mollifier $\rho(\cdot R)R\to \delta$:
\be 
c_R(t)=(2\pi)^{-1}\int_{\R}\sqrt{\rho((|\eta|-1)R)R}e^{it\eta}\dd \eta
\approx \frac{\cos(t)}{\sqrt{R}}.
\ee 
On each line on which $f$ concentrates, the oscillation profile $f|_\ell$ should be a function $c_R$ of such form. 
Such a claim is based on the next heuristic discussion.

We can observe the following. Take $\Phi(r,u,\eta)=\varphi(r,u)\tau(\eta)$ in \cref{thm:ultralimit}, and $\varphi$ in \cref{cor:rado}. We obtain that the next two objects have the same (non-zero) limit in law:
\be 
\langle R^{\frac{d-1}2} \Rad^* \sqrt{-\Delta_r}\wxrayfR , \varphi\tau\rangle \approx \langle R^{\frac{d-1}2} \wick{f(R\cdot)^2} , \varphi\rangle \int_\R \dd \eta |\hat{c_R}(\eta)|^2\tau(\eta)
\ee
Recall that $\Xray_Rf(r,u, \eta)=\Fouvar{t}{\eta}(f(Rr+tu)))$, therefore, since the above identity holds for all test functions $\varphi,\tau$, we can \emph{heuristically} deduce that
\be 
\wick{\left|\Fouvar{t}{\eta}(f(Rr+tu)))\right|^2} = o\tyu R^{-\frac{d-1}2}\uyt
+\tyu\wick{\Rad f(R\cdot)^2}\uyt (r,u)|\hat{c_R}(\eta)|^2,\quad \forall r,u,\eta
\ee
where we used \cref{eq:radon_inversion} to \emph{heuristically} infer that $\Rad^* \sqrt{-\Delta_r}=\Rad^{-1}$ (such an identity would require a thorough justification). At this point we can observe that for any fixed $(r,u)$ the above says that two functions $a,c$ of $\eta$ satisfy $|\hat{a}(\eta)|^2=|\hat{c}(\eta)|^2$ for all $\eta$.

\textbf{When $d=2$}, one has $|\hat{c}(\eta)|^2 = \delta(|\eta|-1)$, 
so that any real function $a$ satisfying $|\hat{a}|^2 = |\hat{c}|^2$ 
must be of the form $a(t) = \cos(t - t_0)$ for some  $t_0\in\R$. Thus, for any fixed $(r,u)$, corresponding to a line $\ell\in \AG_d$, we deduce that there exist a $t_0\in \R$, and an amplitude $A>0$, both depending on $(f,R,r,u)$ such that $f(Rr+tu)=A \cos(t-t_0)$. In other words, for all $\ell\in \AG_2$, we have
\be 
f|_{R\ell} \approx \cos, \quad \text{up to a translation and a scalar factor, depending on $R,\ell,f$.}
\ee
This seems to suggest that there exist random fields $T_R(x)\randin \R$ and $U_R(x)\randin S^{d-1}$ such that
\be 
\wick{ f\tyu R\tyu x+(t-T_R(x))U_R(x))\uyt\uyt f(Rx)}=c_2\cos(Rt)\FGF_{1/2}(x)+o\tyu R^{-\frac{d-1}2}\uyt,
\ee
in the sense of convergence in law in the space $\Sch_\C'(\R^d)$, for some $c_2>0$. Notice that indeed $2\cos(Rt)^2=1-\cos(2Rt)\to 1$ in $\Sch'(\R)$, by the Riemann-Lebesgue Lemma.

\textbf{When $d\ge 3$}, there is not a unique solution $a$ to the \emph{phase retrieval} 
problem: $|\hat{a}(\eta)|^2 = |\hat{c}(\eta)|^2$, since $|\hat{c}|^2$ is not a delta. However, what we can observe in this case is that the oscillation profile of $f|_{R\ell}$, along the enhanced lines, is not a pure cosine, and it is instead a mix of frequencies.

\section{Proof of \texorpdfstring{\cref{thm:ultralimit}}{}}\label{sec:ultraproof}
\subsection{{Covariance structures at fixed $R$}}

\begin{definition}\label{def:M}
    Let us define for any $s,z\in\R^d$ and $\psi\in\Sch_\C(\TAG_d)$ (recall \cref{eq:kernel_h})
\begt \label{eq:M}
\langle M(s,z),\psi\rangle:=h_\psi(s+z,s-z)
= \a_d\int_{S(s^{\perp})}   \Fouvar{r}{2s}^{u^\perp} \ 
\tyu \psi(r,u,-\scal{z,u})\uyt
\dd u.
\eegt 
\end{definition}
\begin{remark}\label{rem:symM}
Note that since $\psi\in\Sch'_\C(\TAG_d)$ respects the symmetry at \cref{eq:symm}, it follows that $\langle M(-s,z),\psi\rangle=\langle M(s,z),\psi\rangle$.
\end{remark}
\begin{lemma}\label{lem:covaraince}
    The {covariance functional} of {$\sqrt{-\Delta_r}\wxrayfR$} is given by the following expression: for any $\psi\in \Sch_\C(\TAG_d)$, let $2\psi_{\text{sym}}(r,u,\eta):=\psi(r,u,\eta)+\psi(r,u,-\eta)$. 
    \begt 
\E\qwe |\langle \sqrt{-\Delta_r}\wxrayfR\ , \psi\rangle|^2 \ewq=2\int_{\R^d\times \R^d}\spec^2\tyu d\tyu {z+R^{-1}s},{-z+R^{-1}s}\uyt\uyt\  \left|\langle M(s,z),\psi_\sym\rangle \right|^2
\\
=2\int_{\R^d\times \R^d}\spec^2\tyu \dd \tyu \xi_1,\xi_2\uyt\uyt\  \left|\langle M\tyu R\tyu \frac{\xi_1+\xi_2}2\uyt,\frac{\xi_1-\xi_2}2\uyt ,\psi_\sym\rangle \right|^2
.
    \eegt
(The first formula is meant as a change of variables, in the sense of \cref{sec:changofvar}.)
\end{lemma}
\begin{proof}
Let
\be 
h_R(\xi_1,\xi_2):=\langle M\tyu \frac{\xi_1+\xi_2}{2},\frac{\xi_1-\xi_2}{2}\uyt,\psi_\sym(\cdot,\cdot,R^{-1}\cdot) \rangle,
\ee 
where $M$ is as in \cref{eq:M}.
Then, by \cref{thm:Wicked_wp}, considering \cref{eq:scaledXray}, 
\be 
\langle \sqrt{-\Delta_r}\wxrayfR, \psi\rangle = 
\int_{\R^d\times\R^d} Z_R^{:2:}(\dd (\xi_1,\xi_2)) h_R(\xi_1,\xi_2),
\ee
where $Z_R=\widehat{f(R\cdot)}$.
We now compute the variance of {the right-hand side} of such an expression using \cref{lem:Z2cov}, {which is given by}
 
\begt\label{e:fox}
 \int_{\R^d\times \R^d}\spec_R^2(\dd (\xi_1,\xi_2))\  h_R(\xi_1,\xi_2)
\overline{h_R(\xi_1,\xi_2)}2.
\eegt
Recall that the spectral measure is $\spec_R=\widehat{K(R\cdot)}$, where $K(x)=\E\kop f(x)f(0)\pok$. Therefore, it scales according to the rule:
\be 
\spec_R(\dd \xi)=\spec(R^{-1}\dd \xi)R^{-d}=\spec(\dd (R^{-1}\xi)).
\ee
Applying the Schwartz change of variables,$(\xi_1,\xi_2)=\Xi_R(r,s)=(s+Rz,s-Rz)$ (see \cref{sec:changofvar}), {one infers that the expression in \eqref{e:fox} equals}
\begt
 2\int_{\R^d\times \R^d}\spec^2\tyu \dd R^{-1}\tyu {s+Rz},{s-Rz}\uyt\uyt\  \left|\langle M(s,Rz),\psi_\sym(\cdot,\cdot,R^{-1}\cdot)\rangle \right|^2
 \\
 =2\int_{\R^d\times \R^d}\spec^2\tyu \dd \tyu {z+R^{-1}s},{-z+R^{-1}s}\uyt\uyt\  \left|\langle M(s,z),\psi_\sym\rangle \right|^2,
\eegt
which coincides with the formula in the statement.
\end{proof}

\subsection{The limit covariance exists}
\begin{proof}[Proof of \texorpdfstring{\cref{thm:ultralimit}}{} -- Point {\rm (1)}]
Let $N,N'\subset \R^n$ be $k$ dimensional submanifolds and let $Y\colon N \to \R^n$ be smooth and injective and such that $Y(N)=N' \cup D$, where $D$ has zero $k$-dimensional Hausdorff measure. Let $\delta_N$ denote the measure of integration over $N$. The coarea formula implies that 
\be 
\delta_{N' }(Y(\dd x))=\delta_{N}(\dd x){\Jac (Y)}.
\ee
Consider the map $Y_R\colon B_R(TS^{d-1})\to S^{d-1}\times S^{d-1}$ such that 
\bega 
Y_R(v,s):&=\tyu v\sqrt{1-|s/R|^2}+s/R, -v\sqrt{1-|s/R|^2}+s/R\uyt, \quad \forall v\in S^{d-1}, \ s\in v^\perp, |s|< R
\\
&=\tyu\exp_v\tyu \frac{s}R\frac{\arcsin{\frac{|s|}R}}{|s|}\uyt,\exp_{-v}\tyu \frac{s}R\frac{\arcsin{\frac{|s|}R}}{|s|}\uyt \uyt.
\eega
$Y_R$ is smooth and injective, since $s\mapsto s|s|^{-1}\arcsin{|s|}$and the exponential map $\exp_v$ of $S^{d-1}$ are.
Observe that $Y_R(v,s)=\Xi_R(s,v\sqrt{1-|s/R|^2})$, where $\Xi_R(s,z)=(z+R^{-1}s,-z+R^{-1}s)$. 
The above formula applies to the map $Y_R$ from $N=B_R(TS^{d-1})$ (the set of all tangent vectors or length less than $R$) to $N'=S^{d-1}\times S^{d-1}\subset \R^{2d}$, with $D$ being the diagonal. Let $\rho^{\otimes 2}(\xi_1,\xi_2)=\rho(\xi_1)\rho(\xi_2)$. We obtain
\begt 
\E\qwe |\langle \sqrt{-\Delta_r}\wxrayfR\ , \psi\rangle|^2 \ewq
=2\int
\delta_N(\dd (\Xi_R(s,z))) \rho^{\otimes 2}(\Xi_R(s,z))\ \left|\langle M(s,z),\psi_\sym\rangle \right|^2
\\
2\int
\delta_N(\dd (Y_R(s,z))) \rho^{\otimes 2}(Y_R(s,z))\ \left|\langle M(s,v\sqrt{1-|s/R|^2}),\psi_\sym\rangle \right|^2
\\
2\int_{S^{d-1}}\dd v\int_{B_R(v^\perp)}\dd s\
{\Jac (Y_R)(v,s)}\ 
\rho^{\otimes 2}(Y_R(s,z))\ \left|\langle M(s,v\sqrt{1-|s/R|^2}),\psi_\sym\rangle \right|^2.
\eegt
The Jacobian of $Y_R$ is computed in \cref{lem:jacob} below, and is $\Jac Y_R(v,s)= \tyu\frac{R
}{2}\uyt^{1-d}\tyu1-\frac{|s|^2}{R^2}\uyt^{\frac{d-3}2}$. Therefore, we conclude that passing to the limit $R\to+\infty$, we get:
\begt\label{eq:exprspherical}
\lim_{R\to+\infty }R^{d-1}\E\qwe |\langle \sqrt{-\Delta_r}\wxrayfR\ , \psi\rangle|^2 \ewq
=
\frac{2}{2^{d-1}}\int_{S^{d-1}}\dd v\int_{v^\perp}\dd s \ 
\rho(v)^2\left|\langle M(s,v),\psi_\sym\rangle \right|^2
\eegt 
We conclude by \cref{lem:planchelike}.
\end{proof} 

\begin{proof}[Proof of \texorpdfstring{\cref{thm:ultralimit}}{}, Point {\rm (2)}] In this case, we can use directly that $\spec^2(\dd (\Xi_R(s,z)))=\rho^{\otimes 2}\circ \Xi_R\Jac(\Xi_R)\dd s\dd z$, and easily compute $\Jac(\Xi_R)=R^{-d}2^{d}$, to conclude:
\begt \label{eq:exprLebesgue}
\lim_{R\to +\infty}R^d\E\qwe |\langle \sqrt{-\Delta_r}\wxrayfR\ , \psi\rangle|^2 \ewq
=\frac{2}{2^d}\int_{\R^{2d}}
\dd s\dd z\  \rho(z)^2\ \left|\langle M(s,z),\psi_\sym\rangle \right|^2.
\eegt
We conclude by \cref{lem:planchelike}.
\end{proof}
\begin{remark}
    In the previous proofs, we used the assumption that $\rho$ is almost everywhere continuous to {infer} the a.s. limits: $\rho^{\otimes 2}\circ Y_R(v,s)\to \rho(v)^2$ and $\rho^{\otimes 2}\circ \Xi_R(z,s)\to \rho(z)^2$.
\end{remark}
\begin{lemma}\label{lem:jacob} One has that
$\Jac\,  Y_R(v,s)= \tyu\frac{R
}{2}\uyt^{1-d}\tyu\sqrt{1-\frac{|s|^2}{R^2}}
\uyt^{d-3}$.
\end{lemma}
\begin{proof}
To compute the Jacobian in the statement, we evaluate the volume of the image of an orthonormal basis. Choose orthonormal bases adapted so that
\begt
v =  e_1, \qquad
s =  s_2e_2, \quad \text{then}
\\
T:=T_{(v,s)}TS^{d-1} = \mathrm{span}\kop
(e_2,-|s|e_1)
,(e_3,0) ,\dots, (e_{d},0), (0, e_2) ,\dots, (0,e_{d}) \pok.
\eegt 
The matrix representing $DY_R\colon T\to \R^{2d}$ in such a frame is
\be\label{eq:DYR_v3}
DY_R=\begin{pmatrix}
-\frac{s_2}{R} & 0^T & -\frac{s_2}{R^2\a} & 0^T \\[6pt]
\a\, e_2^T & \a\,\id_{d-2} & \frac{1}{R}\, e_2^T & \frac{1}{R}\,\id_{d-2} \\[6pt]
-\frac{s_2}{R} & 0^T & \frac{s_2}{R^2\a} & 0^T \\[6pt]
-\a\, e_2^T & -\a\,\id_{d-2} & \frac{1}{R}\, e_2^T & \frac{1}{R}\,\id_{d-2}
\end{pmatrix}
\ee
Computing the Jacobian of this matrix gives the result.
\end{proof}
\subsection{{Identifying the limit covariance}}

\begin{lemma}[Plancherel-like formula] \label{lem:planchelike}
For all $\Phi\in \Sch_\C(\R^d\times \R)$ such that $\Phi(x,t)=\Phi(x,-t)$, we have:

\begt\label{eq:uglymitcov_berry} 
\int_{S^{d-1}}\dd v\int_{v^\perp}\dd s \ 
\rho(v)^2\left|\langle M(s,v),\Rad \Phi\rangle \right|^2  
=\a_d\pi^{d-1}\int_{\AG_d}\dd (u,r)\tyu \rho(u) \int_{S^{d-2}}   \Rad\Phi(r,u,\theta_1)
\dd \theta\uyt^2;
\\
\int_{\R^{2d}}
\dd s\dd z\  \rho(z)^2\ \left|\langle M(s,z),\Rad \Phi\rangle \right|^2=
\frac{\a_d^2}{2^d}\int_{\R^d}\dd s \int_{\R^{d}}
\dd z\  \left|
 \rho(z)\int_{S^{d-2}}   \Fouvar{r}{s}\tyu \Phi(r,{|\Pi_{s^{\perp}}(z)|\theta_1})\uyt\
\dd \theta \right|^2.
\eegt
\end{lemma}
\begin{proof}
The key observation is that the next quantity depends only on $(s,\scal{s,z})$. For all $z\in \R^d$, 
\begt \label{eq:keyobs}
\langle M(s,z),\Rad \Phi\rangle =h_{\Rad \Phi}(s+z,s-z)
=\a_d\int_{S(s^\perp)}\hat\Phi(2s,-\scal{z,u}) \dd u
\\
=
\a_d\int_{S^{d-2}}\hat\Phi(2s,|\Pi_{s^{\perp}}(z)|\theta_1) \dd \theta
=
 \a_d \Fouvar{r}{2s}^{u^\perp} \ 
\tyu \int_{S^{d-2}} \Rad\Phi(r,u,|\Pi_{s^{\perp}}(z)|\theta_1)\dd \theta \uyt,
\eegt
the latter identity being due to \cref{eq:fourier_slice_Rad}, valid for any $u\in s^\perp$. In the first formula in the statement, since $v\in s^\perp$, the spherical integral on the left-hand side is independent of both $v$ and $s$. Thus, we can apply Plancherel identity in the variable $s\in v^\perp$ and obtain the desired formula in the r.h.s. The second formula follows directly from the third identity in \cref{eq:keyobs}.
\end{proof}
\subsection{Gaussianity and tightness}
\begin{proof}[Conclusion of the proof of \cref{thm:ultralimit}: Point {\rm (3)}]

By \cite[Corollary 2.4]{Bierme_CSA} (L\'evy continuity theorem for random tempered distributions): convergence in law in $\Sch'_\C(\R^d\times\R)$ holds if and only if, for every $\psi\in\Sch_\C(\TAG_d)$, the real-valued random variable $\langle\sqrt{-\Delta_r}\wxrayfR,\psi\rangle$ converges in law. Moreover, the limit characteristic functional is automatically continuous at $0$ (since the limit is Gaussian with a continuous covariance form). We will prove that, for every fixed $\psi\in\Sch(\TAG_d)$:
\be\label{eq:marginal_conv}
F_R:=R^{d_i/2}\langle\sqrt{-\Delta_r}\wxrayfR, \psi\rangle\;\xrightarrow{\;\rm law\;}\;\mathcal{N}(0,\sigma^2(\psi)),
\ee
where $\sigma^2(\psi)$ is the limit covariance computed in \cref{thm:ultralimit}, points (1) and (2). By \cref{thm:Wicked_wp}, $\langle\sqrt{-\Delta_r}\wxrayfR,\psi\rangle=\a_d I_2(g_R)$, where $I_2$ denotes the second Wiener--It\^o integral with respect to the isonormal process associated with the spectral measure $\spec$, and  $g_R\in L^2(\spec^2;\C)$ is the symmetric kernel
\be\label{eq:gR_explicit}
g_R(\xi_1,\xi_2)=\int_{S((\xi_1+\xi_2)^{\perp})} \Fouvar{r}{R(\xi_1+\xi_2)}^{u^\perp}\tyu\psi_\sym\tyu r,u,\scal{\xi_2,u}\uyt\uyt \,\dd u,
\qquad \xi_1,\xi_2\in \R^d,
\ee
obtained from \cref{eq:kernel_h} by evaluating $h_{\psi_\sym(\cdot,\cdot,R^{-1}\cdot)}$ at $(R\xi_1,R\xi_2)$. The isometry property of $I_2$ gives $\E[F_R^2]=2\a_d^2 R^{d_i}\|g_R\|_{\mathfrak{H}^{\otimes 2}}^2$, and $\E[F_R^2]\to\sigma^2(\psi)$ was established above. We now identify a uniform decay estimate for the kernel $g_R$. Since $\psi\in\Sch(\TAG_d)$, for each $u\in S^{d-1}$ the mapping $r\mapsto\psi_\sym(r,u,\eta)$ is a Schwartz function of $r\in u^\perp$, uniformly in $u\in S^{d-1}$ and $\eta\in\R$. Standard estimates on the Fourier transform yield: for every $N\geq 1$, there exists a constant $C_N>0$ (depending on $\psi$ and $N$) such that
\be\label{eq:psihat_decay}
\left|\Fouvar{r}{\rho}^{u^\perp}\tyu\psi_\sym(r,u,\eta)\uyt\right|\le \frac{C_N}{(1+|\rho|)^N},
\qquad \forall\, u\in S^{d-1},\quad\forall\,\eta\in\R.
\ee
Indeed, by integration by parts, $(1+|\rho|)^N|\hat\psi_\sym(\rho,u,\eta)|$ is bounded by a sum of $u^\perp$-integrals of derivatives $\partial_r^\alpha\psi_\sym(\cdot,u,\eta)$, each of which is finite and bounded uniformly in $(u,\eta)$ by the Schwartz decay of $\psi$. Since the sphere $S((\xi_1+\xi_2)^\perp)$ has uniformly bounded volume for all $\xi_1,\xi_2$ (including $\xi_1+\xi_2=0$), we deduce from \eqref{eq:gR_explicit} that, for every $N> d$,
\be\label{eq:gR_decay}
|g_R(\xi_1,\xi_2)|\le \frac{C'_N}{(1+R|\xi_1+\xi_2|)^N},\qquad \xi_1,\xi_2\in\R^d.
\ee
We now verify the Schur-like condition \eqref{e:condizione} of \cref{l:venerdi}, which yields the desired asymptotic Gaussianity.

{\it Case {\rm (1)}:} $\spec=\rho\,\sigma$ on $S^{d-1}$, with $\rho$ bounded. Applying the change of coordinates $a_R(z,\e)=\frac{z}{R}+\e y\sqrt{1-\frac{|z|^2}{R^2}}$, with $z\in y^\perp$, $|z|<R$, $\e\in\{-1,1\}$. We get:
\bega\label{e:condizione_1}
M_R &= \sup_{y\in S^{d-1}:\ y\in\supp(\spec)} \int_{S^{d-1}}|g_R(y,a)|\, \rho(a)\sigma(\dd a)
\\
&\le
\sup_{y\in S^{d-1}} \int_{B_R(y^\perp)}\sum_{\e\in \{-1,1\}}\frac{C'_N}{(1+|R(y+a_R(z,\e))|)^N}\, \rho(a_R(z,\e)) \Jac(a_R) (z,\e)\, \dd z,
\eega
where we have used the bound \eqref{eq:gR_decay}. The Jacobian is $\Jac(a_R) (z,\e)=(1-\frac{|z|^2}{R^2})^{-\frac12}\frac{|z|}{R^d}$. As $R\to\infty$, we have that $R(y+a_R(z,-1))\to z$ and $R(y+a_R(z,+1))\approx z+2Ry$. Expanding separately the two terms in $\e$,
\bega
M_R&\lesssim
\|\rho\|_\infty\sup_{y\in S^{d-1}} \int_{y^\perp}\frac{C'_N}{(1+|z|)^N}  \frac{|z|}{R^d}\, \dd z
+
\|\rho\|_\infty\sup_{y\in S^{d-1}} \int_{y^\perp}\frac{C'_N}{(1+|z+2Ry|)^N} \frac{|z|}{R^d}\, \dd z
\\
&\lesssim
\frac{1}{R^d}+\frac{o(R^{-k})}{R^d}=O(R^{-d})=o(R^{-\frac{d-1}{2}})=o\big(\|g_R\|_{\mathfrak{H}^{\otimes 2}}\big),
\eega
since $d\geq 2$ implies $d>(d-1)/2$, and we used the rapid decay \eqref{eq:gR_decay} with $N>d$.

{\it Case {\rm (2)}:} $\spec=\rho\,\dd \xi$ on $\R^d$, with $\rho$ bounded. Using \eqref{eq:gR_decay} directly and substituting $b=a+y$:
\bega\label{e:condizione_2}
M_R &= \sup_{y\in \R^{d} \,: \, y\in {\rm supp}(\spec)} \int_{\R^{d}}|g_R(y,a)|\, \rho(a)\, \dd  a
\\
&
\lesssim
\|\rho\|_\infty
\int_{\R^{d}} \frac{C'_N}{(1+R|b|)^N}\, \dd  b
=O(R^{-d})=o(R^{-\frac{d}{2}})=o\big(\|g_R\|_{\mathfrak{H}^{\otimes 2}}\big),
\eega
for $N>d$, completing the verification of \eqref{e:condizione} in both cases.
\end{proof}

\begin{remark}\label{rem:end_of_tightgaus}
The proof above works for arbitrary $\psi\in\Sch_\C(\TAG_d)$: not only for test functions of the form $\psi=\Rad\Phi$. The key estimate \eqref{eq:gR_decay} uses only the rapid decay of the Fourier transform of $\psi$ in the variable $r\in u^\perp$, which is guaranteed by the Schwartz regularity of $\psi$. For test functions of the form $\psi=\Rad\Phi$, the kernel $g_R$ enjoys the stronger decay $|g_R(\xi_1,\xi_2)|\le C_N(1+R|\xi_1+\xi_2|+|\langle\xi_1,u\rangle|)^{-N}$ (since $\hat\Phi$ is Schwartz in both variables), but this additional decay is not needed for the Schur test.
In particular, the argument justifies the claim in \cref{prop:dense_TAG}: the convergence in \cref{eq:XR-law} can be lifted from $\Rad\Sch'_\C(\R^d\times\R)$ to $\Sch'_\C(\TAG_d)$.
\end{remark}
\begin{proof}[Proof of \texorpdfstring{\cref{thm:ultralimit_2}}{}] In the proof of \cref{thm:ultralimit}(1), at \cref{eq:exprspherical}, we obtain the convergence of the variance in the l.h.s. of  \cref{eq:var_limit_2}. In the two-dimensional case, this formula simplifies easily, for all test functions $\psi\in \Sch_\C(\TAG_2)$.
\begt\label{eq:exprspherical_2}
\lim_{R\to+\infty }R\E\qwe |\langle \sqrt{-\Delta_r}\wxrayfR\ , \psi\rangle|^2 \ewq
=
\int_{S^{1}}\dd v\int_{v^\perp}\dd s \ 
\rho(v)^2\left|\langle M(s,v),\psi_\sym\rangle \right|^2
\\
=\int_{S^{1}}\dd v\int_{v^\perp}\dd s \ 
\rho(v)^2\left|
\a_2\sum_{u\in \{v,-v\}}   \Fouvar{r}{2s}^{v^\perp} \ 
\tyu \psi_\sym(r,u,-\scal{v,u})\uyt
\dd u
\right|^2
\\
=\frac12\a_2^2 \int_{S^{1}}\dd v\int_{v^\perp}\dd s \ 
\rho(v)^2\left|
 \Fouvar{r}{s}^{v^\perp} \ 
\tyu \psi_\sym(r,u,1)\uyt
\dd u
\right|^2.
\eegt 
In the last step we used that $\psi(r,-u,-\eta)=\psi(r,u,\eta)$. We conclude by Plancherel in $v^\perp$. The Gaussianity and tightness follow from the same Schur test argument as in the proof of \cref{thm:ultralimit}(3), see  \cref{rem:end_of_tightgaus} and \cref{prop:dense_TAG}.
\end{proof}
\section{Miscellaneous Proofs}

\subsection{Proof of Theorem \ref{t:poisson}}\label{ss:PoissonProof} [Part {\bf (a)}] First observe that, if $d\geq 2$ and $\varphi\in\mathcal S(\R^d)$, then
\[
\AG_d\to \R: \ell\mapsto \Rad\varphi(\ell)
\]
belongs to $L^p(\AG_d,\mu)$ for every $p\geq 1$. Indeed, writing
$\ell=r+\R u$, with $u\in S^{d-1}$ and $r\in u^\perp$, one has
\[
\Rad\varphi(\ell)= \Rad\varphi(r,u)=\int_\R \varphi(r+tu)\,dt,
\]
and the rapid decay of $\varphi$ implies that, for every $N>0$,
\[
|\Rad\varphi(r,u)|\leq C_N(1+\|r\|)^{-N}.
\]
The conclusion follows by integrating over $u\in S^{d-1}$ and
$r\in u^\perp\simeq\R^{d-1}$.
Now fix $t>0$, and let $\eta_t\sim {\rm PLP}_t$, in such a way that
$$
\mathcal{R}^*\eta_t(\varphi_R) = \int_{\AG_d}\, \mathcal{R}\varphi_R (\ell) \, \eta_t(\dd\ell).
$$
By Campbell's formula (see e.g. \cite[Proposition 2.7]{LastPenroseBook}), one has
\[
\mathbb{E}[\mathcal{R}^*\eta_t(\varphi_R)]
=
t\int_{\AG_d}\mathcal{R}\varphi_R(\ell)\,\mu(\dd\ell)
=
t\gamma\int_{\R^d}\varphi_R(x)\,\dd x
=
t\gamma R^d\int_{\R^d}\varphi(x)\,\dd x,
\]
where $\gamma>0$ is a constant depending only on $d$, and the second identity follows by Fubini, yielding that the invariant measure $\mu$ satisfies
\[
\int_{\AG_d}\int_{\ell}\varphi_R(x)\,\Haus^1(\dd x)\,\mu(\dd\ell)
=
\gamma\int_{\R^d}\varphi_R(x)\,\dd x .
\]
As a consequence e.g. of \cite[Proposition 12.4]{LastPenroseBook}, one has therefore that
$$
\mathcal{R}^*\eta_t(\varphi_R) - t\gamma R^d \int_{\R^d} \varphi(x) \dd x = I_1(\mathcal{R}\varphi_R (\cdot)),
$$
where $I_1$ indicates a Wiener-It\^o integral of order 1 with respect to the compensated measure $\eta_t - t\mu$. It follows that
$$
{\bf Var}(\mathcal{R}^*\eta_t(\varphi_R)) = t \int_{\AG_d} \mathcal{R}\varphi_R(\ell)^2 \mu(\dd\ell) = tC(d)^{-1} R^{d+1} \mathbb{E}[h(\varphi)^2],
$$
where we have used \eqref{e:fgfcov2}. To conclude, we notice that
$$
\int_{\AG_d} |\mathcal{R}\varphi_R(\ell)|^3 \mu(\dd\ell) \asymp R^{2+d} =o((R^{d+1})^{3/2})= o(R^{(3d+3)/2}), \quad R\to\infty,
$$
and deduce the CLT in \eqref{e:UCscaling} by using \cite[Corollary 3.4]{PSTU2010}.

\medskip

\noindent[Part {\bf (b)}] Let $t_1$ be such that $\beta^2 = C(d)^{-1} t_1$, and write
$$
\alpha - t_1\gamma = {\rm sign}(\alpha - t_1\gamma)\cdot |\alpha-t_1\gamma| := \varepsilon \cdot t_2.
$$
Then, writing $${\bf a}_R = \alpha R^d\left( \int_{\R^d} \varphi^{(1)}(x) dx,..., \int_{\R^d} \varphi^{(m)}(x) dx\right),$$
one has that both vectors
$$
\frac{{\bf Y}_R - {\bf a}_R}{\beta R^{(d+1)/2}}, \quad \mbox{and}\quad \frac{{\bf L}_R - {\bf a}_R}{\beta R^{(d+1)/2}},
$$
where we have used the notation in the statement, converge in law to ${\bf J}:= (h(\varphi^{(1)}),..., h(\varphi^{(m)}))$. Since ${\bf J}$ is Gaussian and non-singular, the classical results from
\cite[Section 4]{RangaRao1962} imply that
\[
d_{\rm conv}\left(\frac{{\bf Y}_R - {\bf a}_R}{\beta R^{(d+1)/2}}, {\bf J}\right),
\quad
d_{\rm conv}\left(\frac{{\bf L}_R - {\bf a}_R}{\beta R^{(d+1)/2}}, {\bf J}\right)
\longrightarrow 0,
\]
yielding that, by the triangle inequality,
\[
d_{\rm conv}\left(
\frac{{\bf Y}_R - {\bf a}_R}{\beta R^{(d+1)/2}},
\frac{{\bf L}_R - {\bf a}_R}{\beta R^{(d+1)/2}}
\right)\longrightarrow 0.
\]
Since the convex distance is invariant under common invertible affine transformations of its
arguments, the last relation is equivalent to
\[
d_{\rm conv}({\bf Y}_R,{\bf L}_R)\longrightarrow 0.
\]
The proof is concluded.
\subsection{Proof of Theorem \ref{t:FGFpullback}}\label{ss:proofPullback}  Without loss of generality, throughout this section we will consider test functions $\varphi\in C_c^\infty$ that are supported on the unit ball $B(0,1)$. The idea of the proof is to use the results from \cite{DNPR23} to couple quadratic functionals of Berry's models and of monochromatic waves on the same probability space, and then to exploit Keeler's estimate \eqref{e:keeler} in order to deduce precise numerical bounds. To this end, consider a diverging sequence $\{r_\lambda\}$ such that $r_\lambda = O(\sqrt{\lambda/\log \lambda})$. Combining the estimate \eqref{e:keeler} with the second part of \cite[Theorem 2.1]{DNPR23} (case $j= \lfloor d/2 \rfloor +1$ and {$M=j+k$}), one deduces that, for every $\lambda$, there exists a probability space $(\Omega_\lambda, \mathcal{F}_\lambda, \mathbb{P}_\lambda)$ carrying a two-dimensional jointly Gaussian process $\{X_\lambda(x), Y_\lambda(x) : x\in B(0, r_\lambda)\} $ such that
\begin{eqnarray*}
 && \{X_\lambda(x): x\in B(0, r_\lambda)\} \stackrel{\rm law}{=} \{\tilde{\Psi}_\lambda^{x_0}(x) : x\in B(0, r_\lambda)\} \\
 && \{Y_\lambda(x): x\in B(0, r_\lambda)\} \stackrel{\rm law}{=} \{B(x) : x\in B(0, r_\lambda)\},
\end{eqnarray*}
and
\begin{equation}\label{e:sam1}
\mathbb{E}_\lambda\left[\|X_\lambda - Y_\lambda\|^2_{C_b^k(B(0, r_\lambda))}  \right]\leq C\, r_\lambda^{{2k+2}}\left( \frac{r_\lambda^d}{\log \lambda} + \frac{r_\lambda^{2d+\frac12}}{\sqrt{\log \lambda}}\right),
\end{equation}
where $C$ is a constant only depending on $d,k$; we also mention the following elementary estimate (see e.g. \cite[Proposition C.1]{NPV} for a self-contained proof): as $r_\lambda\to\infty$,
\begin{equation}\label{e:sam2}
\mathbb{E}\left[\|B \|^2_{C_b^k(B(0, r_\lambda))}\right] \leq C \log r^d_\lambda,
\end{equation}
where $C$ is a constant depending on $d,k$. Now write
\begin{eqnarray*}
&& M_\lambda(\varphi) := \int_{\R^d}Q\big(\partial^a X_\lambda(x) : a\in \mathcal{J}\big)\, \varphi_{r_\lambda}(x)\, dx, \\
&& N_\lambda(\varphi) := \int_{\R^d}Q\big(\partial^a Y_\lambda(x) : a\in \mathcal{J}\big)\, \varphi_{r_\lambda}(x)\, dx.
\end{eqnarray*}
The proof is concluded by showing that, when \eqref{e:rlento} is in order, then
\begin{equation}\label{e:L1}
\frac{\mathbb{E}_\lambda \left| M_\lambda(\varphi) -N_\lambda(\varphi)\right|}{r_\lambda^{(d+1)/2}}\longrightarrow 0,
\end{equation}
which implies
\begin{equation}\label{e:L2}
\frac{\mathbb{E}_\lambda \left| M_\lambda(\varphi) -N_\lambda(\varphi)\right|^2}{r_\lambda^{d+1}}\longrightarrow 0,
\end{equation}
(and therefore the desired conclusion), because of hypercontractivity and the fact that the difference $M_\lambda(\varphi) -N_\lambda(\varphi)$ belongs to a finite sum of Wiener chaoses associated with $\{X_\lambda, Y_\lambda\}$ (see e.g. \cite[Corollary 2.8.14 and Remark 2.8.15]{nourdinpeccatibook}). In order to control the left-hand side of \eqref{e:L1}, one writes (dropping the dependence on $\mathcal J$ in the second line to simplify the notation)
\begin{eqnarray*}
&& Q\big(\partial^a X_\lambda(x) : a\in \mathcal J \big) -Q\big(\partial^a Y_\lambda(x) : a\in \mathcal J \big)   \\
&& = L\big(\partial^a X_\lambda(x) -\partial^a Y_\lambda(x)\big) + B\Big(\partial^a X_\lambda(x) +\partial^a Y_\lambda(x) \, ; \,  \partial^a X_\lambda(x) -\partial^a Y_\lambda(x) \Big),
\end{eqnarray*}
where $L(\cdot)$ and $B(\cdot\, ;\, \cdot\cdot)$ are, respectively, a fixed linear mapping and a fixed bilinear form. 
Indicating by $C$ a generic constant independent of $\lambda$, one has that
$$
\mathbb{E}\left[\sup_{x\in B(0,r_\lambda)} |L\big(\partial^a X_\lambda(x) -\partial^a Y_\lambda(x)\big)|\right]\leq C\, \mathbb{E}_\lambda\left[\|X_\lambda - Y_\lambda\|^2_{C_b^k(B(0, r_\lambda))}  \right]^{1/2},
$$
and 
\begin{eqnarray*}
&&\mathbb{E}\left[\sup_{x\in B(0,r_\lambda)} \Big|B\Big(\partial^a X_\lambda(x) +\partial^a Y_\lambda(x) \, ; \,  \partial^a X_\lambda(x) -\partial^a Y_\lambda(x) \Big)\Big|\right]\\
&& \leq C\left\{ \mathbb{E}_\lambda\left[\|X_\lambda - Y_\lambda\|^2_{C_b^k(B(0, r_\lambda))}  \right] + \mathbb{E}\left[\|Y \|^2_{C_b^k(B(0, r_\lambda))}\right]^{1/2}\mathbb{E}_\lambda\left[\|X_\lambda - Y_\lambda\|^2_{C_b^k(B(0, r_\lambda))}  \right]^{1/2}\right\}.
\end{eqnarray*}
Elementary considerations now show that that the left-hand side of \eqref{e:L1} is bounded (up to an absolute multiplicative constant) by
\begin{eqnarray*}
&& \frac{r_\lambda^d}{r_\lambda^{(d+1)/2}}\left\{ \mathbb{E}_\lambda\left[\|X_\lambda - Y_\lambda\|^2_{C_b^k(B(0, r_\lambda))}  \right]^{1/2}+\mathbb{E}_\lambda\left[\|X_\lambda - Y_\lambda\|^2_{C_b^k(B(0, r_\lambda))}  \right] \right.\\
&&\quad \quad \quad\quad \quad \quad \quad \quad \left.+ \mathbb{E}\left[\|Y \|^2_{C_b^k(B(0, r_\lambda))}\right]^{1/2}\mathbb{E}_\lambda\left[\|X_\lambda - Y_\lambda\|^2_{C_b^k(B(0, r_\lambda))}  \right]^{1/2}\right\},
\end{eqnarray*}
and the stated result is obtained through a lengthy but straightforward exponent tracking, exploiting the bounds \eqref{e:sam1}--\eqref{e:sam2}. This concludes the proof.

\bibliographystyle{plain}
\bibliography{scar}
\end{document}